\documentclass[3p,12pt]{elsarticle}

\usepackage{amsthm}
\usepackage{amssymb}
\usepackage{colortbl}
\usepackage{natbib}
\usepackage{float}
\usepackage{caption}
\usepackage{latexsym,amsfonts,amssymb}
\usepackage{graphicx,graphics}
\usepackage[algosection]{algorithm2e}
\usepackage[titletoc]{appendix}
\usepackage{amsmath}
\usepackage{epsfig}
\usepackage{subfigure}
\usepackage[misc]{ifsym}
\usepackage{appendix}

\newtheorem{theorem}{Theorem}[section]
\newtheorem{lemma}{Lemma}[section]
\newtheorem{example}{Example}[section]
\newtheorem{remark}{Remark}[section]

\usepackage{multirow}

\begin{document}
	\captionsetup[figure]{labelfont={bf},labelformat={default},labelsep=period,name={Fig.}}
	\captionsetup[table]{labelfont={bf},labelformat={default},labelsep=period,name={Table}}
	\begin{frontmatter}
		\title{A stabilized semi-implicit Fourier spectral method for  nonlinear space-fractional reaction-diffusion  equations}
		
		\author[]{Hui Zhang\fnref{label1}}
		\ead{m18766170439@163.com}
		\author[]{Xiaoyun Jiang\fnref{label1}}
		\ead{wqjxyf@sdu.edu.cn}
		\author[]{Fanhai Zeng\corref{cor1}\fnref{label2}}
		\ead{fanhaiz@foxmail.com}
		\cortext[cor1]{Corresponding author.}
		\author[]{George Em Karniadakis\fnref{label3}}
		\ead{george\_karniadakis@brown.edu}
		\address[label1]{School of Mathematics, Shandong University, Jinan 250100, PR China.}
		\address[label2]{Department of Mathematics, National University of Singapore, Singapore 119076.}
		\address[label3]{Division of Applied Mathematics, Brown University, Providence RI, 02912.}
		\begin{abstract}
			The reaction-diffusion model can generate a wide variety of spatial patterns, which
				has been widely applied in chemistry, biology, and physics, even used to explain self-regulated pattern formation in the developing animal embryo \cite{Kondo2010}. In this work, a second-order stabilized semi-implicit time-stepping Fourier spectral method is presented for the reaction-diffusion systems of equations with space described by the fractional Laplacian. We adopt the temporal-spatial error splitting argument to illustrate that the proposed method is stable without imposing the CFL condition, and we prove an optimal $L^2$-error estimate. We also analyze the linear stability of the stabilized semi-implicit method and obtain a practical criterion to choose the time step size to guarantee the stability of the semi-implicit method. Our approach is illustrated by solving several problems of practical interest, including the fractional Allen-Cahn, Gray-Scott and FitzHugh-Nagumo models, together with an analysis of the properties of these systems in terms of the fractional power of the underlying Laplacian operator, which are quite different from the patterns of the corresponding integer-order model. 
		\end{abstract}
	\end{frontmatter}
	\section{Introduction}
	\label{intro}
	We consider the following two-dimensional nonlinear space-fractional reaction-diffusion equation \cite{Pindza2016,Simmons2016,Lee2018},
	\begin{equation}\begin{aligned}\label{e1.1}
	\left\{ \begin{array}{ll}
	{\partial_t}u  (x,t) =-K_{\alpha}(-\Delta)^{\alpha/2}u(x,t)+G(u(x,t)), \ (x,t)\in \Omega\times I,\\
	u(x,0)=u^{0}(x),\  x\in \Omega,\\
	u(\cdotp,t) \mathrm{\ is \ 2\pi-periodic}, \ x\in \partial\Omega,\  \forall t\in I,\\
	\end{array}\right.																					
	\end{aligned}\end{equation}where $\Omega=(0,2\pi)^{2}$, $I=(0,T]$, $x=(x_{1},x_{2})$, $K_{\alpha}>0$ is the diffusion coefficient, $G(u)$ is a nonlinear function that satisfies the local Lipschitz condition, and $(-\Delta)^{\alpha/2}$ $(1<\alpha\leq2)$ is the fractional Laplacian operator defined by \cite{Ainsworth2017}
	\begin{equation} (-\Delta)^{\alpha/2}u(x)=\sum_{k,l\in\mathbb{Z}}(k^{2}+l^{2})^{\alpha/2}\hat{u}_{kl}e^{ikx_{1}+ilx_{2}},
	\end{equation}
	in which $i^2=-1$ and $\hat{u}_{kl}$ are the Fourier coefficients,
	\begin{equation}
	\hat{u}_{kl}=(u,e^{ikx_{1}+ilx_{2}})=\frac{1}{(2\pi)^{2}}\int_{\Omega}u e^{-(ikx_{1}+ilx_{2})}dx.
	\end{equation}
	For any function $u(x)\in L_{per}^{2}(\Omega)$, we have \cite{Zhang2014,Duo2016}
	\begin{equation}
	u(x)=\sum_{k,l\in\mathbb{Z}}\hat{u}_{kl}e^{ikx_{1}+ilx_{2}}.
	\end{equation}
	
	The fractional reaction-diffusion equation has attracted much attention in the last few years due to its wide applications in many fields of science and engineering \cite{Lee2018,Meerschaert1999,Bu2014,Mao2018,Arshad2017,Song2016,Acetoa2017,Kilbas2006,Metzler2000,Meerschaert2012,Roop2004,Bueno2014}. For $G(u)=u-u^3$, \eqref{e1.1} becomes the space-fractional Allen--Cahn equation \cite{Song2016}, which describes the mixture of two incompressible fluids. The fractional order controls the sharpness of the interface, which is typically diffusive in integer-order phase-field models. As shown in  \cite{Wangting2019}, the space-fractional Gray--Scott model is a coupled counterpart of \eqref{e1.1}, which represents an autocatalytic reaction-diffusion process between two chemical species \cite{Bueno2014}. The space-fractional FitzHugh-Nagumo model describes the ionic fluxes as a function of the membrane potential \cite{Klevin2015}.
	
	Numerical methods for space-fractional reaction-diffusion equations similar to \eqref{e1.1} have been extensively investigated.   Ervin et al. \cite{Ervin2007} considered the numerical approximation for space-fractional diffusion equations containing a nonlocal quadratic nonlinearity, in which an optimal $H^{\alpha}$-error estimate was presented under the time step restriction $\tau\leq ch$, where $\tau$ is the temporal step size and $h$ is the spatial mesh size. An alternating direction implicit Galerkin-Legendre spectral method for the two-dimensional Riesz space fractional nonlinear reaction-diffusion equation was proposed in \cite{Zeng2014}, which requires $\tau^2 N\leq c$ to ensure stability or convergence, where $N$ is the degree of polynomial spatial basis. Similar results were obtained in \cite{Ma2002,Wu2004,Cannon1988,Hayes1981}, in which numerical methods for integer-order partial differential equations were proposed and analyzed. The conditions as $\tau\leq ch$ and $\tau^2 N\leq c$ were required to prove the stability and convergence of the aforementioned works due to the use of the inverse inequality in the numerical analysis. However, the conditions as $\tau\leq ch$ and $\tau^2 N\leq c$ in \cite{Zeng2014,Ma2002,Wu2004,Cannon1988,Hayes1981} can be weakened by a temporal-spatial error splitting argument  \cite{Li2012}. In this work, we extend the temporal-spatial error splitting argument to analyze the stabilized semi-implicit Fourier spectral method for the two-dimensional nonlinear space-fractional reaction-diffusion equations, which has not been fully studied. Stabilized  time-stepping methods  have been applied for solving nonlinear integer-order partial differential equations, such as Allen-Cahn and Cahn-Hilliard equations, see e.g. \cite{He2007,Wang2018,Shen2010,Wu2014}.
	
	
	The main contribution of this work is to prove that the stabilized semi-implicit method is stable without imposing the CFL condition by adopting the temporal-spatial error splitting argument \cite{Li2012,Li2013,Li2014}, and present an optimal error estimate. We also analyze the linear stability of the
		stabilized semi-implicit method and obtain the explicit form of the stability condition criterion
		(see \eqref{e33.04-2} and \eqref{e33.05}),
		which is verified by numerical simulations.
	The temporal-spatial error splitting technique was originally developed for numerical analysis of integer-order equations. Up to now, this method has been applied to the numerical analysis of time-fractional partial differential equations \cite{Li2017b,Li2018}. In this paper, we extend this method to the two-dimensional nonlinear space-fractional partial differential equations, which has not been fully explored. Our approach is illustrated by solving several problems of practical interest, including the fractional Allen-Cahn, Gray-Scott and FitzHugh-Nagumo models.
	
	This paper is organized as follows. Section \ref{sec2} develops a second-order stabilized semi-implicit time-stepping Fourier spectral method to solve the two-dimensional nonlinear space-fractional reaction-diffusion equation. In Section \ref{sec4}, we analyze the linear stability of the stabilized semi-implicit method and obtain a practical criteria to choose the time step size to guarantee the stability of the semi-implicit method. The semi-implicit method is extended to a system of space-fractional reaction-diffusion models in Section \ref{sec55}. Numerical examples are given in Section \ref{sec5} to confirm the theoretical analysis. Finally, the conclusions are summarized in Section \ref{sec6}.
	\section{The stabilized semi-implicit Fourier method}
	\label{sec2}
	Let $C_{per}^{\infty}(\Omega)$ be the set of all restrictions onto $\Omega$ of all complex-valued, $2\pi$-periodic, $C^{\infty}$-functions on $\mathbb{R}^{2}$.
	For a nonnegative real number $r$, let $H_{per}^{r}(\Omega)$ be the closure of $C_{per}^{\infty}(\Omega)$ with the norm $\|\cdot\|_{r}$ and semi-norm $|\cdot|_{r}$
	defined by (see \cite{Ainsworth2017})
	\begin{equation}
	|u|^{2}_{r}=\sum_{k,l\in \mathbb{Z}}\hat{u}_{kl}^{2}(k^{2}+l^{2})^{r},\quad\|u\|^{2}_{r}=\sum_{k,l\in \mathbb{Z}}\hat{u}_{kl}^{2}(1+k^{2}+l^{2})^{r}.
	\end{equation}
	For simplicity of the numerical analysis, we also assume that $\hat{u}_{00}=0$, which implies that the
	norm $\|\cdot\|_{r}$ and semi-norm $|\cdot|_{r}$ are equivalent.

	For a positive integer $N$, the  function space is denoted by
	$$X_{N}=\mathrm{span}\left\{e^{ikx_{1}+ilx_{2}}:-{N}/{2}\leq k,l\leq {N}/{2}-1\right\}.$$
	Define the orthogonal projection operator $P_{N}$ as follows
	\begin{equation}\begin{aligned}
	\label{e2.7}
	(u-P_{N}u,v)=0,\quad  \forall v\in X_{N},\quad u\in L_{per}^{2}(\Omega).
	\end{aligned}\end{equation}

	For the temporal discretization, we divide the interval $[0,T]$ into $K$ subintervals with a time step size $\tau=T/K$. Let $t_n=n\tau$, $u^{n}=u(x,t_{n})$, $0\leq n\leq K$, and denote
	\begin{equation}
	\label{e2.2}
	D_{1}u^{n}=\frac{u^{n}-u^{n-1}}{\tau},\quad
	D_{2}u^{n}=
	\frac{3}{2}D_{1}u^{n}-\frac{1}{2}D_{1}u^{n-1}.\\
	\end{equation}
	
	The time derivative is discretized by the second-order backward difference formula,
	we obtain
	$$ D_{2}u^{n}+K_{\alpha}(-\Delta)^{\alpha/2}u^n=G(u^{n}) + O(\tau^2),\quad n\geq2.$$
	Dropping the truncation error and applying the Fourier spectral method (other methods  can be applied, i.e., finite difference/element method) in space, we can
	obtain a fully implicit numerical scheme, which can be solved iteratively.
	Compared with the explicit method or the linearized method,
	the fully implicit method is costly and its implementation is a little complicated.
	The extrapolation has been widely used to obtain
	the semi-implicit method, which needs to solve a linear system.
	For example, The  second-order extrapolation  $G(u^n)=2G(u^{n-1})-G(u^{n-2})+O(\tau^2)$
	can be applied, which preserves second-order accuracy in time, but
	it needs a small time step to ensure stability.  In order to balance the stability and convergence, we follow \cite{He2007} to develop the stabilized method for the time discretizetion,
	that is, we replace  $G(u^n)$ with $2G(u^{n-1})-G(u^{n-2})-\tau\kappa(D_1u^n-D_1u^{n-1})$
	to obtain
	\begin{equation}
	\label{e3.4}
	D_{2}u^{n}+K_{\alpha}(-\Delta)^{\alpha/2}u^n=2G(u^{n-1})-G(u^{n-2})
	- \tau\kappa(D_1u^n-D_1u^{n-1}) + R^{n},\ 2\leq n\leq K,\\
	\end{equation}
	where $R^{n} = O(\tau^{2})$, $\kappa\geq 0$, and
	$\tau\kappa(D_1u^n-D_1u^{n-1})=\kappa(u^n-2u^{n-1}+u^{n-2})=O(\tau^2)$ is a penalty term that balances the stability and accuracy.
	
	From \eqref{e3.4}, the fully discrete Fourier spectral method for \eqref{e1.1}
	is given by: Find $U_{N}^{n}\in X_{N}$ for $n\geq2$ such that for all $v\in X_{N}$,
	\begin{equation}\label{e2.3}
	\left\{\begin{aligned}
	&\mathcal{A}^n(U_{N},v) =(P_N(2G(U_{N}^{n-1})-G(U_{N}^{n-2})),v)
	-\tau\kappa(D_1U_{N}^n-D_1U_{N}^{n-1},v),\\
	&U_{N}^{0}=P_{N}u^{0},\\
	&U_{N}^{1}=P_{N}u^{1},
	\end{aligned}\right.
	\end{equation}
	where
	\begin{equation} \label{Ann}
	\mathcal{A}^n(u,v) = (D_{2}u^n,v)+K_{\alpha}((-\Delta)^{\alpha/2}u^{n},v).
	\end{equation}
	
	\begin{remark}\label{re2.1}
		In real applications, $U_N^1$ can be obtained by the following simple approach
		$$U_N^1=P_N(u_0+\tau\partial_tu(0)).$$
		Other approaches can be applied. For example, $U_N^1$ can be obtained
		by the fully implicit Crank--Nicolson Fourier spectral method.
	\end{remark}
	
	An optimal error estimate of the numerical scheme (\ref{e2.3}) is provided
	in the following theorem, the relevant proof can be seen in Appendix A.
	\begin{theorem}\label{th2.1}
		Suppose that \eqref{e1.1} admits a unique solution $u(\cdot,t)\in  H^{r}(\Omega)$ satisfying \eqref{e2.5}.
		Let $U_{N}^{n}(2\leq n\leq K)$ be the solution of \eqref{e2.3} and $G'(z)$ is bounded
		for $|z|\leq M$, $M$ is a suitable positive constant.
		Then there exist two positive constants
		$\tau^*$ and $N^*$, when $\tau\leq \tau^*$ and $N\geq N^*$, it holds
		\begin{align}
		&\|u^{n}-U_{N}^{n}\|\leq C(\tau^{2}+N^{-r}). \label{eq2.6}
		\end{align}
	\end{theorem}
	\section{Linear stability}
		\label{sec4}
		We investigate the linear stability of the method (\ref{e2.3}) in this section.
		Let	$G(u)=\rho u$ and $\rho<0$, we have
		\begin{equation}\begin{aligned}
		\label{e33.01}
		\mathcal{A}^n(U_{N},v) =\rho(2 U_{N}^{n-1}-U_{N}^{n-2},v)-\kappa(U^n_{N}-2U^{n-1}_{N}+U^{n-2}_{N},v).
		\end{aligned}\end{equation}
		The above equation can be equivalently written as
		\begin{equation}\label{e33.02}\begin{aligned}
		\left(\frac{3}{2}+\mu_{k,l}\tau +\kappa \tau \right)\hat{U}_{k,l}^n-\left(2+2\rho\tau+2\kappa\tau\right)\hat{U}_{k,l}^{n-1}+\left(\frac{1}{2}+\rho\tau+\kappa\tau\right)\hat{U}_{k,l}^{n-2}=0,
		\end{aligned}\end{equation}
		where $\mu_{k,l}=K_{\alpha}(k^2+l^2)^{\alpha/2}$. The characteristic equation
		of the above recurrence formula is
		\begin{equation}\begin{aligned}
		\label{e33.03}
		\left(\frac{3}{2}+\mu_{k,l}\tau +\kappa \tau \right)\eta^2-\left(2+2\rho\tau+2\kappa\tau\right)\eta+\left(\frac{1}{2}+\rho\tau+\kappa\tau\right)=0.
		\end{aligned}\end{equation}
		The two roots of \eqref{e33.03} are
		\begin{equation}\label{e33.033}\begin{aligned}
		&\eta^{\pm}=\frac{2+2\rho\tau+2\kappa\tau\pm \sqrt{\triangle}}{3+2\mu_{k,l}\tau+2\kappa \tau },
		\end{aligned}\end{equation}
		where
		\begin{equation}\begin{aligned}
		\label{e33.034}
		\triangle=\left(2+2\rho\tau+2\kappa\tau\right)^2-\left(3+2\mu_{k,l}\tau +2\kappa \tau \right)\left(1+2\rho\tau+2\kappa\tau\right).
		\end{aligned}\end{equation}
		
		The recurrence relation \eqref{e33.02} is stable if $|\eta^{\pm}| < 1$.
		We discuss the following two cases.
		\begin{itemize}
			\item Case I: For $\triangle<0$ and $|\eta^{\pm}| < 1$, we have
			\begin{equation}\label{e33.035}\begin{aligned}
			&-1<\frac{1+2\rho\tau+2\kappa\tau}{3+2\mu_{k,l}\tau +2\kappa \tau}< 1
			\Longrightarrow\kappa > -\frac{\mu_{k,l}+\rho}{2} - \frac{1}{\tau}.
			\end{aligned}\end{equation}
			\item Case II: For $\triangle>0$ and $|\eta^{\pm}| < 1$, we have
			\begin{equation}\begin{aligned}
			\label{e33.0356}
			&4\kappa\tau+3\rho\tau+\mu_{k,l}\tau+4>0
			\Longrightarrow \kappa > -\frac{\mu_{k,l}+3\rho}{4}-\frac{1}{\tau}.
			\end{aligned}\end{equation}
		\end{itemize}
		
		Combining \eqref{e33.035}--\eqref{e33.0356}, we can obtain that
		the method (\ref{e33.01}) is   stable if
		\begin{equation}\begin{aligned}
		\label{e33.04}
		\kappa>\frac{-\mu_{k,l}-3\rho}{4}-\frac{1}{\tau}.
		\end{aligned}\end{equation}
		Since $\mu_{k,l}\geq0$ and $\tau>0$,  we have the following unconditional stability condition
		\begin{equation}\begin{aligned}
		\label{e33.04-2}
		\kappa>-\frac{3\rho}{4}.
		\end{aligned}\end{equation}
		
		In real applications, we need to choose $\tau\leq \tau^*<1$ in the numerical simulations
		to obtain suitably accurate numerical solutions. From \eqref{e33.04}, for the step size
		$\tau\in(0,\tau^*]$, the method  \eqref{e33.01} is stable if
		\begin{equation} \label{e33.05}
		\kappa>-\frac{3\rho}{4}-\frac{1}{\tau^*}.
		\end{equation}
		
		
		\begin{figure}
			\begin{center}
				\begin{minipage}{0.47\textwidth}\centering
					\centering
					\includegraphics[width=3in]{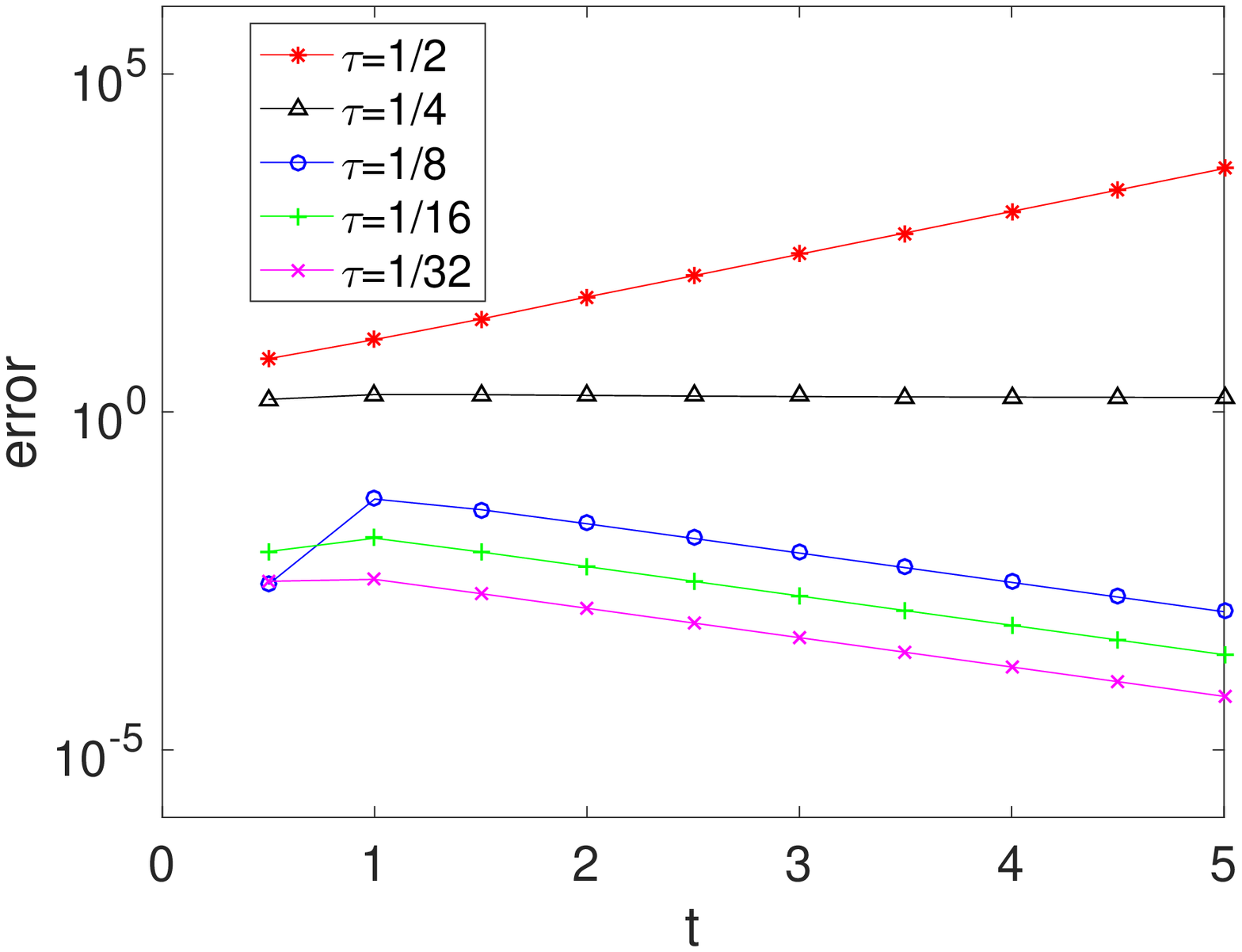}
					\par {(a) Case I.}
				\end{minipage}
				\begin{minipage}{0.47\textwidth}\centering
					\centering
					\includegraphics[width=3in]{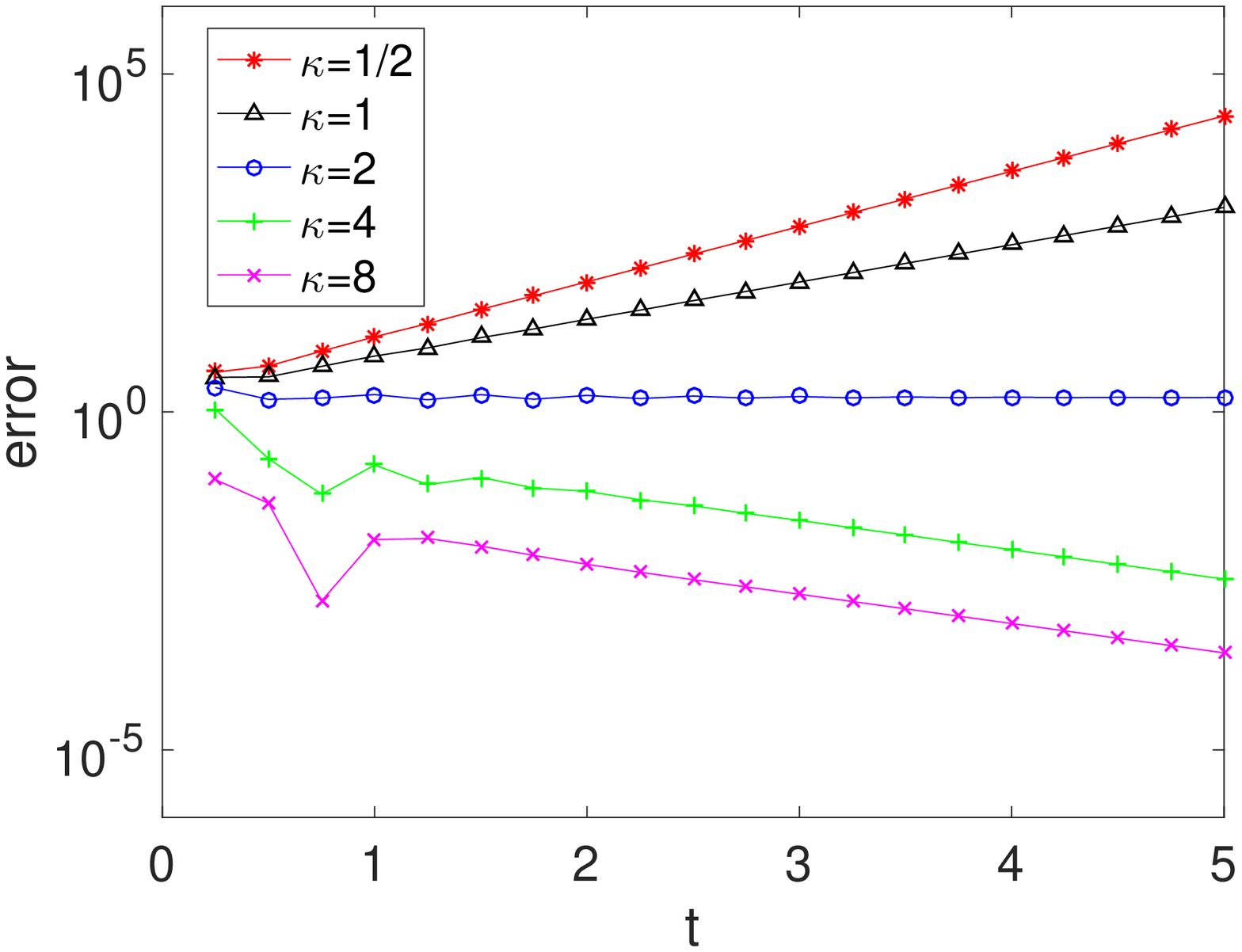}
					\par {(b) Case II.}
				\end{minipage}
			\end{center}
			\caption{The $L^2$ errors  for
				Example \ref{ex3.1}, $\alpha=1.3$.}
			\label{fig3-1}
		\end{figure}
		
		In the  following example, we verify the stability criterion \eqref{e33.05}.
		\begin{example}
			\label{ex3.1}
		We apply the method \eqref{e2.3} to solve \eqref{e1.1},
			where $G(u)=-8u+\exp(-t)(-\sin{(x_1+x_2)}+8\sin{(x_1+x_2)}+2^{\alpha/2}\ K_\alpha\sin{(x_1+x_2)})$.
			The exact solution of
			\eqref{e1.1} is $u=\exp(-t)\sin{(x_1+x_2)}.$
		\end{example}
		We consider the following two cases:
		\begin{itemize}
			\item Case I: $\kappa=2$, $\tau=2^{-j},1\leq j \leq 5, T = 5$.
			\item  Case II: $\tau = 1/4$, $\kappa=1/2,1,2,4,8$, $T = 5$.
		\end{itemize}
		The error is defined by
		$$e^n=\|U^n_N-u(\cdot,t_n)\|.$$
		
		From \eqref{e33.05}, we derive that the method is stable for this example
		if $\kappa> 6-1/\tau^*$ and $\tau\leq\tau^*$.
		For Case I, we conclude that
		the method is stable when $\tau\leq\tau^*=1/4$ and is unstable when $\tau>\tau^*=1/4$,
		which is also verified in Fig. \ref{fig3-1} (a).
		For Case II, we fix the step size $\tau=1/4$. It is easy to
		obtain that the method is stable if $\kappa>2$, which is also
		demonstrated from Fig. \ref{fig3-1} (b). We have also verified that the stability condition \eqref{e33.05} is independent of the fractional order $\alpha\in(1,2]$, but these results are not shown here.
	\section{Extension to a system of space-fractional reaction-diffusion models}
	\label{sec55}
	We extend the semi-implicit method to the following coupled system,
	\begin{equation}\begin{aligned}\label{eq4.9}
	\left\{ \begin{array}{ll}
	{\partial_t}u=-K_{u}(-\Delta)^{\alpha/2}u+G_{1}(u,v) \quad
	(x,t)\in \Omega\times I,\\
	{\partial_t}v=-K_{v}(-\Delta)^{\alpha/2}v+G_{2}(u,v),
	\quad(x,t)\in \Omega\times I.\\
	\end{array}\right.																				
	\end{aligned}\end{equation}
	
	Each equation of system \eqref{eq4.9} is discretized by the second-order stabilized semi-implicit time-stepping Fourier spectral scheme as \eqref{e2.3}. The fully numerical approximation for \eqref{eq4.9} is developed as follows: Find $U_{N}^n, V_{N}^n\in X_{N}$ for $n=2,3,\ldots,K$ such that
	\begin{equation}\label{eq5.1}
	\left\{\begin{aligned}
	&\mathcal{A}_{1}^n(U_{N},v)=(P_N(2G_1(U_{N}^{n-1})-G_1(U_{N}^{n-2})),v)-\tau\kappa(D_1U_{N}^n-D_1U_{N}^{n-1},v), \quad\forall v\in X_{N},\\
	&\mathcal{A}_{2}^n(V_{N},v)=(P_N(2G_2(V_{N}^{n-1})-G_2(V_{N}^{n-2})),v)-\tau\kappa(D_1V_{N}^n-D_1V_{N}^{n-1},v),\quad\forall v\in X_{N},\\
	&U_{N}^{0}=P_{N}u^{0},\quad U_{N}^{1}=P_{N}(u_0+\tau\partial_tu(0)),\\
	&V_{N}^{0}=P_{N}v^{0},\quad V_{N}^{1}=P_{N}(v_0+\tau\partial_tv(0)),
	\end{aligned}\right.\end{equation}
	where
	\begin{displaymath}\begin{aligned} \label{aaa}
	&\mathcal{A}_{1}^n(u,v) = (D_{2}u^n,v)+K_{u}((-\Delta)^{\alpha/2}u^{n},v),\\
	&\mathcal{A}_{2}^n(u,v) = (D_{2}u^n,v)+K_{v}((-\Delta)^{\alpha/2}u^{n},v).\\
	\end{aligned}\end{displaymath}
	
	In \eqref{eq4.9}, if we take
	\begin{equation}\begin{aligned}\label{eq4.0}
	&G_{1}(u,v)=-uv^{2}+F(1-u), \ G_{2}(u,v)=uv^{2}-(F+\lambda)v, 			
	\end{aligned}\end{equation}
	then we obtain the space-fractional Gray--Scott model.
	
	Reaction and diffusion of chemical species can produce a variety of patterns. A representative reaction and diffusion model is the Gray--Scott model, which is a variant of the autocatalytic Selkov model of glycolysis \cite{Gray1983,Gray1985}. Numerical simulations of this model were performed in \cite{Pearson1993} to find stationary lamellar patterns like those observed in earlier laboratory experiments on iodate-ferrocyanid-sulfite reactions \cite{Lee1993}. The Gray--Scott model corresponds to the following two reactions \cite{Pearson1993},
	\begin{equation}\begin{aligned}\label{eq4.11}
	U + 2V &\rightarrow3V,\\
	V &\rightarrow P,\\
	\end{aligned}\end{equation}
	both reactions are irreversible, so $P$ is an inert product. A nonequilibrium constraint is represented by a feed term for $U$. Both $U$ and $V$ are removed by the feed process. Note that, in the aforementioned literature on pattern dynamics of the Gray--Scott model, the models are all with standard diffusion, i.e., the diffusion operator is the normal Laplacian. When we study the effects of the superdiffusion, we have to employ the fractional Laplacian $(-\Delta)^{\alpha/2}$ for $1<\alpha\leq2$ on pattern formation of this model. This space-fractional Gray--Scott model describes an autocatalytic reaction-diffusion process between two chemical species with concentrations $u$ and $v$ \cite{Bueno2014}.
	
	Taking $K_{v}=0$ and
	\begin{equation}\begin{aligned}
	\label{eq4.1111}
	&G_{1}(u,v)=u(1-u)(u-\mu)-v, \\
	&G_{2}(u,v)=\varepsilon(\beta u-\gamma v-\delta),\\																					
	\end{aligned}\end{equation}we can get the space-fractional FitzHugh--Nagumo model.
	
	Spiral waves in excitable media provide an important example of self-organization phenomena in spatially distributed biological, physical, and chemical systems \cite{Winfree1987}, and can be modeled by the FitzHugh--Nagumo equations \cite{FitzHugh1961,Nagumo1962}.	The fractional FitzHugh--Nagumo model consists of a space-fractional nonlinear reaction-diffusion model and a system of ordinary differential equations, describing the ionic fluxes as a function of the membrane potential \cite{Bueno2014}. Mathematical models of electrical activity in cardiac tissue are becoming increasingly powerful tools in the study of cardiac arrhythmias. In this context, the fractional model presented here represents a new approach to dealing with the propagation of the electrical potential in heterogeneous cardiac tissue. In \cite{Bueno2014}, the authors proposed a fundamental rethinking of the homogenization approach via the use of a fractional Fick's law \cite{Magin2008,Meerschaert2006} and, in particular, a fractional FitzHugh--Nagumo model is introduced to capture the spatial heterogeneities and spatial connectivities in the extracellular domain through the use of fractional derivatives. A semi-alternating direction method for a 2-D fractional FitzHugh-Nagumo monodomain model on an approximate irregular domain was considered in \cite{Klevin2015}.
	
	\section{Numerical results}
	\label{sec5}
	In this section, we provide some numerical results to confirm our theoretical analysis. The error at $t=t_n$ is denoted by
	\begin{equation}
	\label{eqq2.5}
	e^n(\tau,N)=U_{N}^n-u^n_{ref},
	\end{equation}
	where $u_{ref}$ is the reference solution obtained
	using a smaller time step size $\tau=0.0005$ for a fixed $N$,
	or a larger $N=1024$ for a fixed time step size if the exact solution
	is not explicitly given.
	%
	The $L^{2}$-norm convergence orders in time and space are defined as
	\begin{equation}
	\label{eq:2.5}
	\mathrm{order}=\left\{ \begin{array}{ll}
	\frac{\log{(\|e^n(\tau_1,N)\|/\|e^n(\tau_2,N)\|)}}{\log{(\tau_{1}/\tau_{2})}}\quad \mathrm{in\  time},\\
	\frac{\log{(\|e^n(\tau,N_{1})\|/\|e^n(\tau,N_{2})\|)}}{\log{(N_{2}/N_{1})}}\quad \mathrm{in\  space},\\
	\end{array}\right.
	\end{equation}
	where $\tau_{1}\neq\tau_{2}$ and $N_{1}\neq N_{2}$.
	\begin{figure}
		\centering
		\subfigure[Case I.]{
			\begin{minipage}[t]{0.43\textwidth}
				\centering
				\includegraphics[width=2.5in]{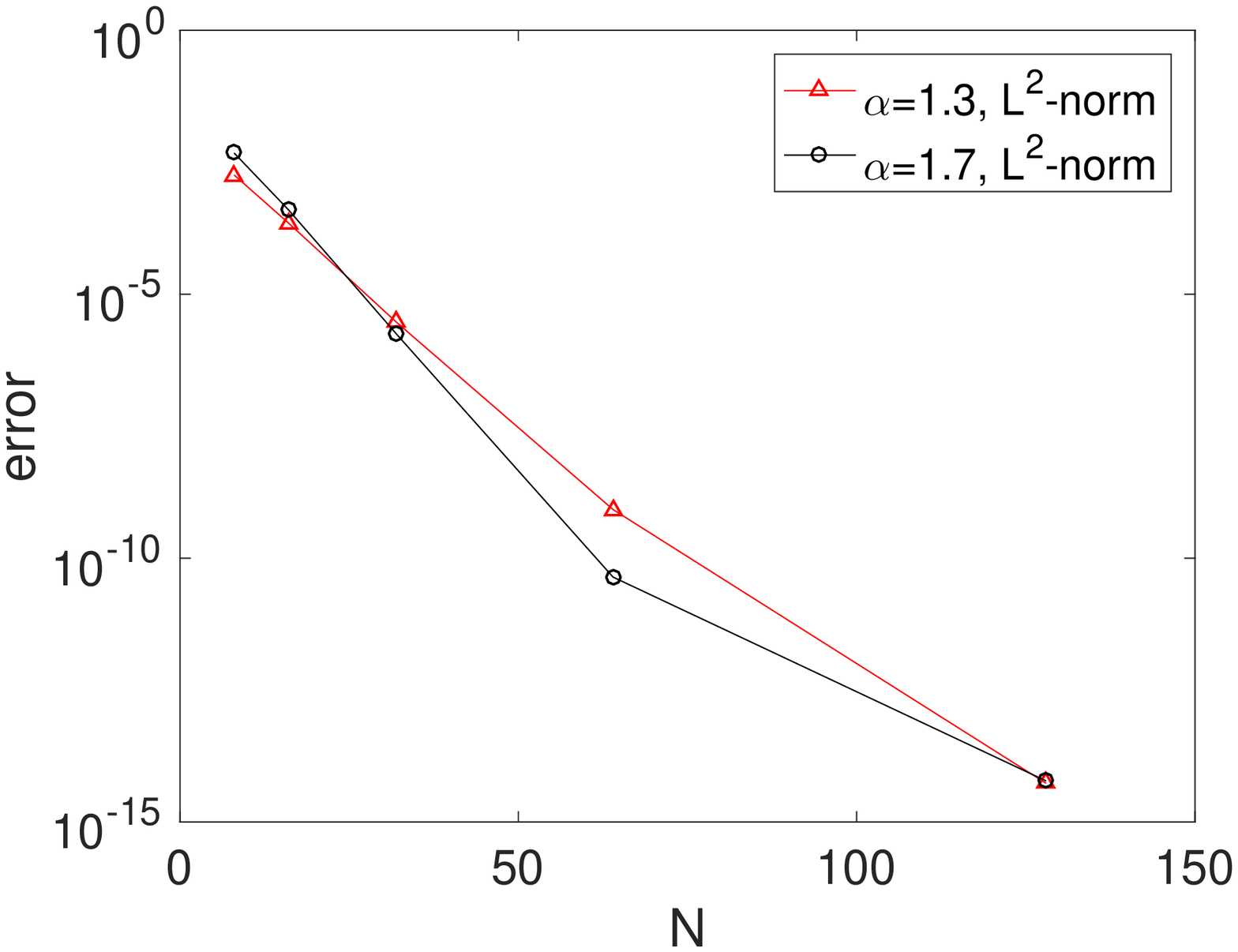}
		\end{minipage}}
		\subfigure[Case II.]{
			\begin{minipage}[t]{0.43\textwidth}
				\centering
				\includegraphics[width=2.5in]{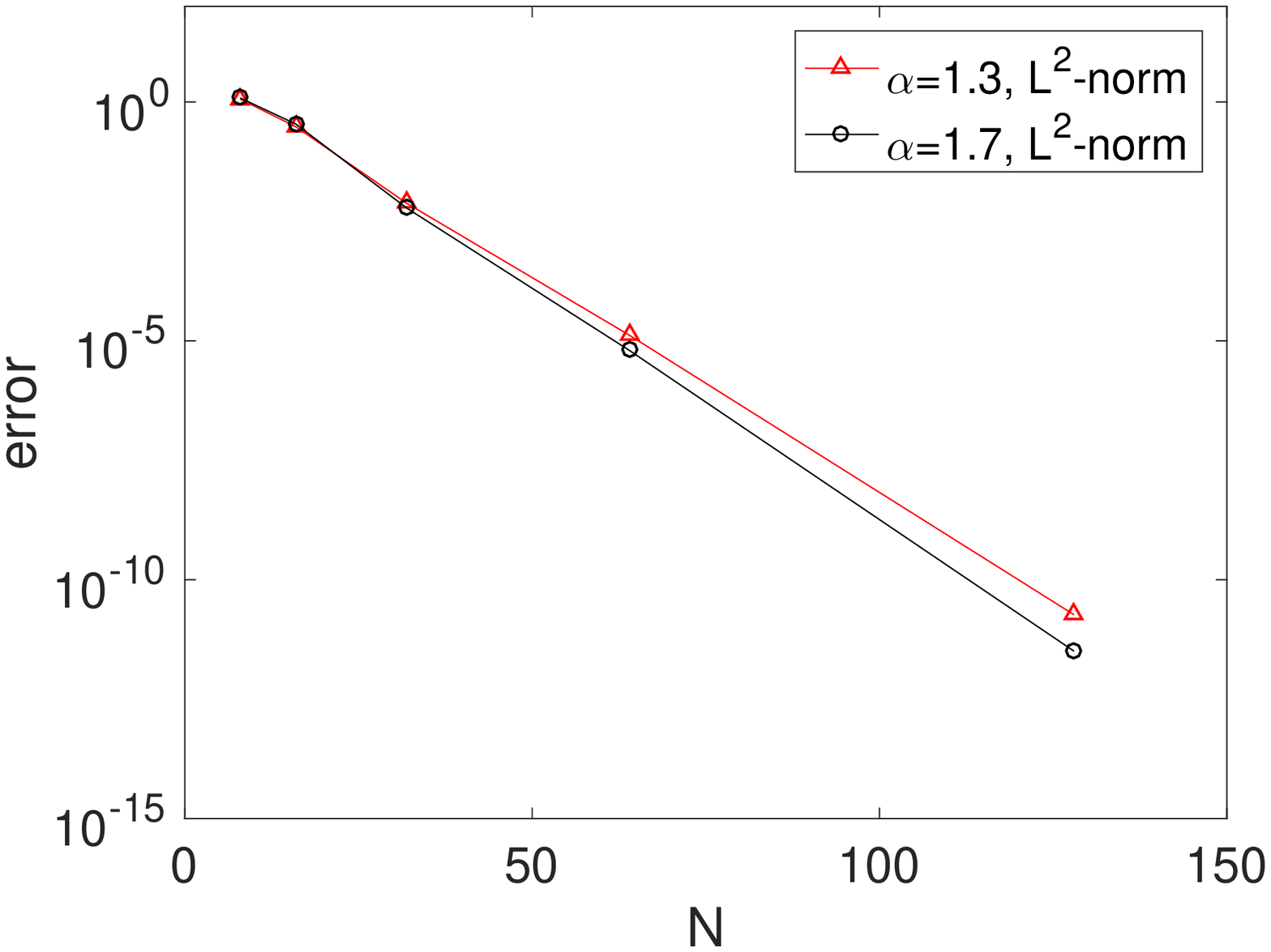}
		\end{minipage}}
		\caption{The $L^2$ error in space for Example \ref{ex4.1}, $\tau=0.01$, $\kappa=1$.}
		\label{fig1}
	\end{figure}
	\begin{table}
		\caption{Case I of Example \ref{ex4.1}: the convergence rate in time for different values of $\kappa$. }
		\label{tab1}
		\centering
		\begin{tabular}{|c|c|c|c|c|c|c|c|}
			\hline
			\multirow{2}{*}{$\alpha$}	&	\multirow{2}{*} {$\tau$} &  \multicolumn{2}{c|}{$\kappa=1$} &  \multicolumn{2}{c|}{$\kappa=3$} &  \multicolumn{2}{c|}{$\kappa=4$} \\ \cline{3-8}& & Error & Order  & Error & Order & Error & Order \\
			\hline
			&1/10 & 5.2105e-3 & -      & 1.3160e-2 & -       & 1.7031e-2 & - \\
			&1/20 & 1.3441e-3 & 1.9548 & 3.4301e-3 & 1.9398  & 4.4529e-3 & 1.9354 \\
			1.3 &1/40 & 3.4169e-4 & 1.9759 & 8.7904e-4 & 1.9643  & 1.1447e-3 & 1.9598 \\
			&1/80 & 8.5820e-5 & 1.9933 & 2.2188e-4 & 1.9862  & 2.8956e-4 & 1.9830 \\
			&1/160& 2.1142e-5 & 2.0212 & 5.4810e-5 & 2.0173  & 7.1619e-5 & 2.0155 \\
			\hline
			&1/10 & 3.1801e-3 & -      & 8.1250e-3 & -       & 1.0553e-2 & - \\
			&1/20 & 8.0715e-4 & 1.9782 & 2.0745e-3 & 1.9696  & 2.6943e-3 & 1.9697 \\
			1.7 &1/40 & 2.0396e-4 & 1.9845 & 5.2881e-4 & 1.9720  & 6.8893e-4 & 1.9675 \\
			&1/80 & 5.1102e-5 & 1.9968 & 1.3329e-4 & 1.9881  & 1.7410e-4 & 1.9844 \\
			&1/160& 1.2576e-5 & 2.0228 & 3.2914e-5 & 2.0178  & 4.3059e-5 & 2.0155 \\
			\hline
		\end{tabular}
	\end{table}
	\begin{example}
		\label{ex4.1}
		The Allen-Cahn equation with a quartic double well potential is a simple nonlinear reaction-diffusion model that arises in the study of formation and motion of phase boundaries \cite{Feng2013,Shen2016}. The space-fractional equation takes the form \cite{Song2016,Lizheng2017}
		\begin{equation}\label{eq4.111}
		{\partial_t}u=-K_{\alpha}(-\Delta)^{\alpha/2}u+u-u^{3},
		\end{equation}
		where $K_{\alpha}$ is a small positive constant.
	\end{example}
	\begin{itemize}
		\item Case I: Choose the following initial condition
		\begin{equation}
		\label{eq4.1}
		u^{0}(x_1,x_2)=\sin{ (2x_1)}\cos{(2x_2)}, \quad (x_1,x_2)\in[0,2\pi]\times [0,2\pi].
		\end{equation}
	\end{itemize}
	
	We choose the values of the parameters $K_{\alpha}=0.01$. For Case I, Table \ref{tab1} presents
	the $L^{2}$-errors for different values of $\kappa$ at $t=2$ when $N=32$. It is obvious that
	second-order accuracy is observed, which is in line with the theoretical analysis. Next, we fix the step size $\tau=0.01$, $\kappa=1$, the $L^{2}$-errors at $t=2$ against $N$ are shown in Fig. \ref{fig1} (a), spectral accuracy is displayed, which agrees with
	the theoretical analysis.
	\begin{itemize}
		\item Case II: Choose the following initial condition
		\begin{equation}
		\label{eq4.2}
		u^{0}(x_1,x_2)=\exp{(-x_1^{2}-x_2^{2})},\quad (x_1,x_2)\in[-20,20]\times [-20,20].
		\end{equation}
	\end{itemize}
	\begin{table}
		\caption{Case II of Example \ref{ex4.1}: the convergence rate in time for different values of $\kappa$.  }
		\label{tab2}
		\centering
		\begin{tabular}{|c|c|c|c|c|c|c|c|}
			\hline
			\multirow{2}{*}{$\alpha$}	&	\multirow{2}{*} {$\tau$} &  \multicolumn{2}{c|}{$\kappa=1$} &  \multicolumn{2}{c|}{$\kappa=3$} &  \multicolumn{2}{c|}{$\kappa=4$} \\ \cline{3-8}& & Error & Order  & Error & Order & Error & Order \\
			\hline
			&1/10 & 1.6244e-2 & -      & 3.6935e-2 & -       & 4.6592e-2 & - \\
			&1/20 & 4.3963e-3 & 1.8856 & 1.0319e-2 & 1.8397  & 1.3196e-2 & 1.8200 \\
			1.3 &1/40 & 1.1374e-3 & 1.9506 & 2.7041e-3 & 1.9321  & 3.4785e-3 & 1.9235 \\
			&1/80 & 2.8767e-4 & 1.9832 & 6.8752e-4 & 1.9757  & 8.8665e-4 & 1.9720 \\
			&1/160& 7.1082e-5 & 2.0168 & 1.7028e-4 & 2.0135  & 2.1984e-4 & 2.0119 \\
			\hline
			&1/10 & 1.6522e-2 & -      & 3.7636e-2 & -       & 4.7481e-2 & - \\
			&1/20 & 4.4665e-3 & 1.8871 & 1.0508e-2 & 1.8406  & 1.3442e-2 & 1.8206 \\
			1.7 &1/40 & 1.1550e-3 & 1.9513 & 2.7525e-3 & 1.9327  & 3.5420e-3 & 1.9241 \\
			&1/80 & 2.9206e-4 & 1.9835 & 6.9973e-4 & 1.9759  & 9.0270e-4 & 1.9723 \\
			&1/160& 7.2160e-5 & 2.0170 & 1.7329e-4 & 2.0136  & 2.2381e-4 & 2.0120 \\
			\hline
		\end{tabular}
	\end{table}

	Considering $K_{\alpha}=0.01$, $N=32$, we present the $L^{2}$-error for different values of $\kappa$ at $t=2$ and
	its convergence order in time in Table \ref{tab2}. Clearly, second-order accuracy
	is still observed for Case II. Fixing $\tau=0.01$, Fig. \ref{fig1} (b) shows spectral accuracy in space for a fixed $\tau=0.01$, $\kappa=1$.
	\begin{figure}[htbp]
		\centering
		\subfigure[t=1000]{
			\begin{minipage}[t]{0.18\textwidth}
				\centering
				\includegraphics[width=1in]{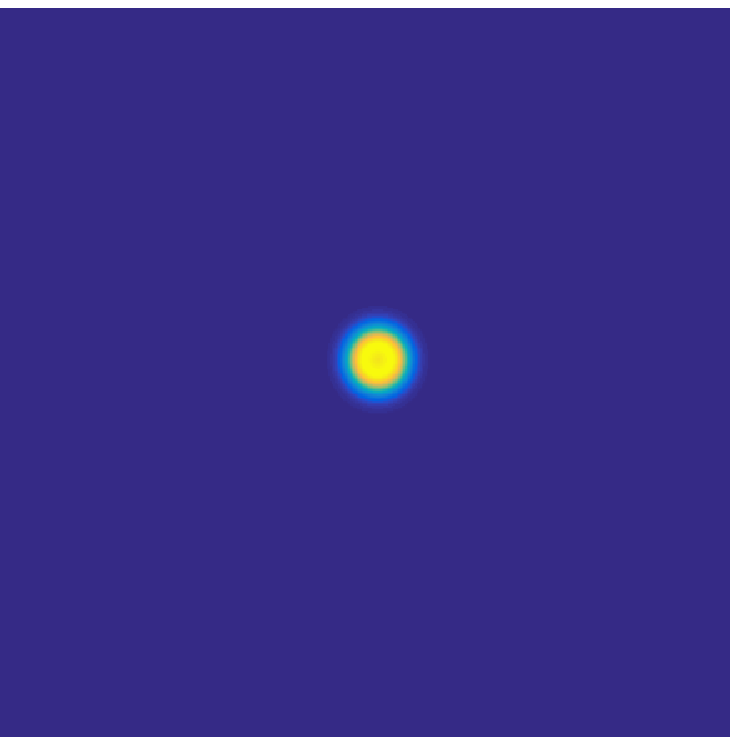}
		\end{minipage}}
		\subfigure[t=6000]{
			\begin{minipage}[t]{0.18\textwidth}
				\centering
				\includegraphics[width=1in]{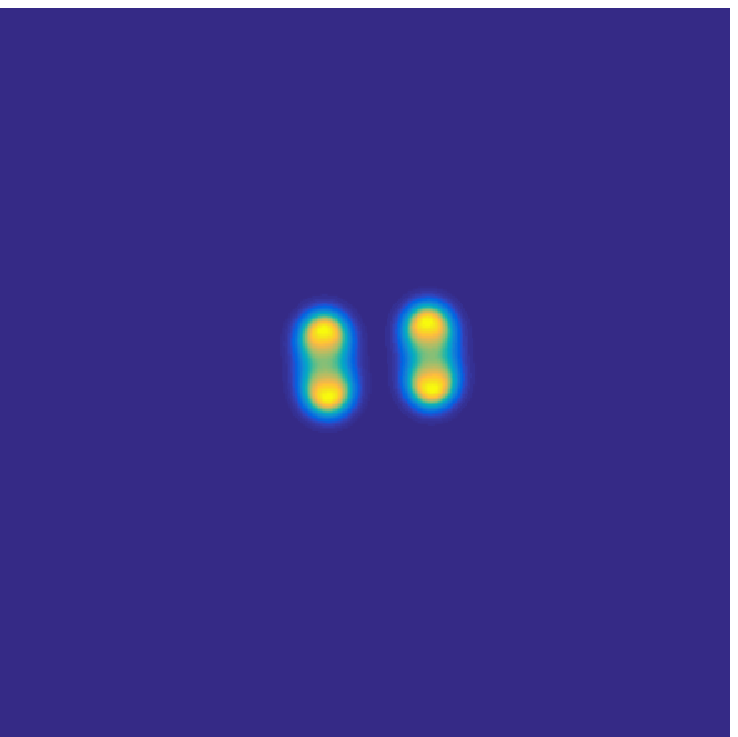}
		\end{minipage}}
		\subfigure[t=6800]{
			\begin{minipage}[t]{0.18\textwidth}
				\centering
				\includegraphics[width=1in]{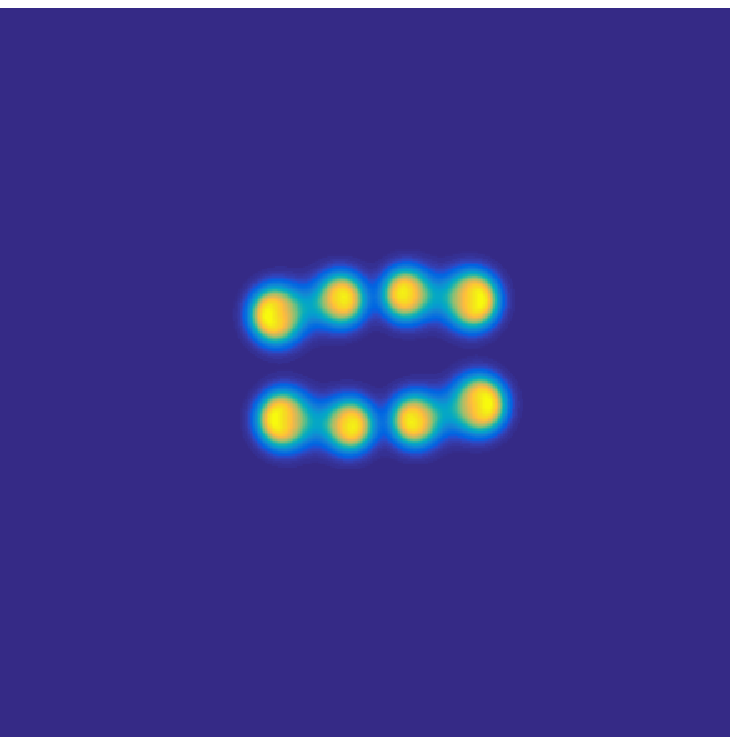}
		\end{minipage}}
		\subfigure[t=9400]{
			\begin{minipage}[t]{0.18\textwidth}
				\centering
				\includegraphics[width=1in]{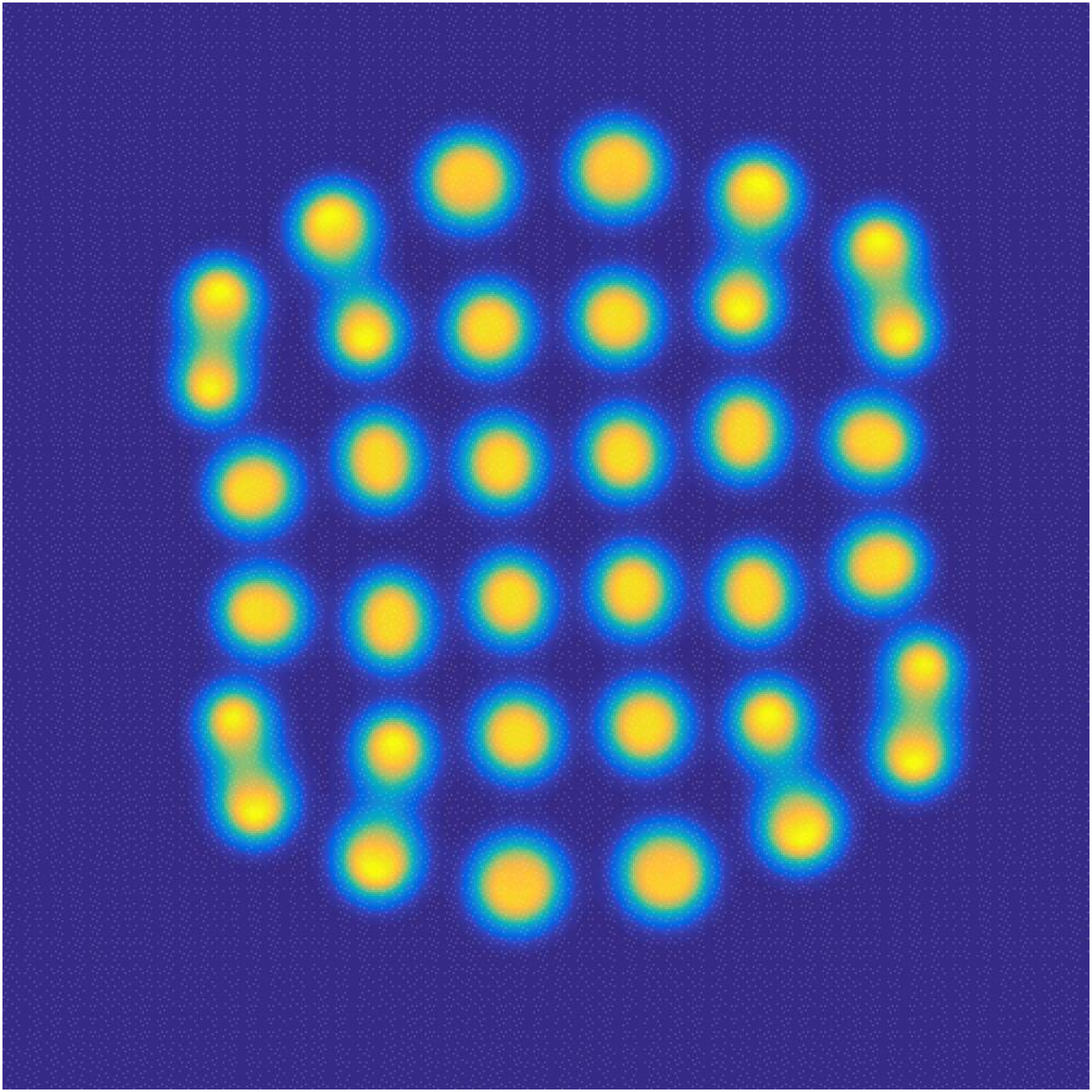}
		\end{minipage}}
		\subfigure[t=30000]{
			\begin{minipage}[t]{0.18\textwidth}
				\centering
				\includegraphics[width=1in]{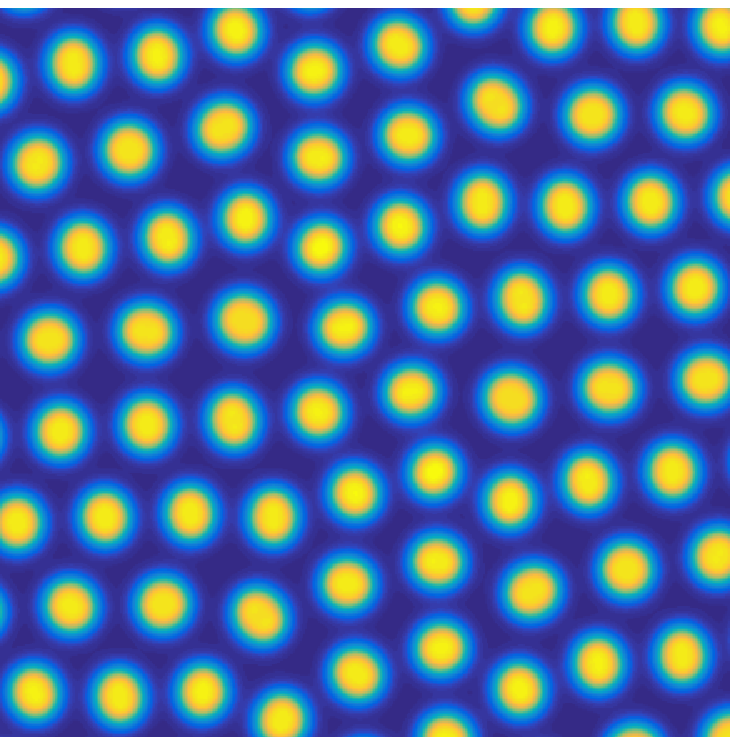}
		\end{minipage}}
		\subfigure[t=200]{
			\begin{minipage}[t]{0.18\textwidth}
				\centering
				\includegraphics[width=1in]{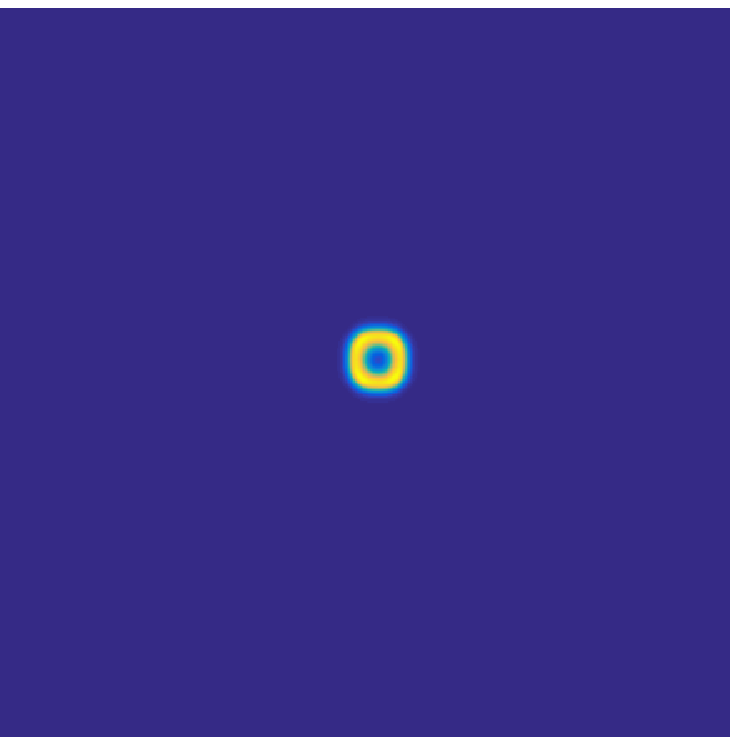}
		\end{minipage}}
		\subfigure[t=2000]{
			\begin{minipage}[t]{0.18\textwidth}
				\centering
				\includegraphics[width=1in]{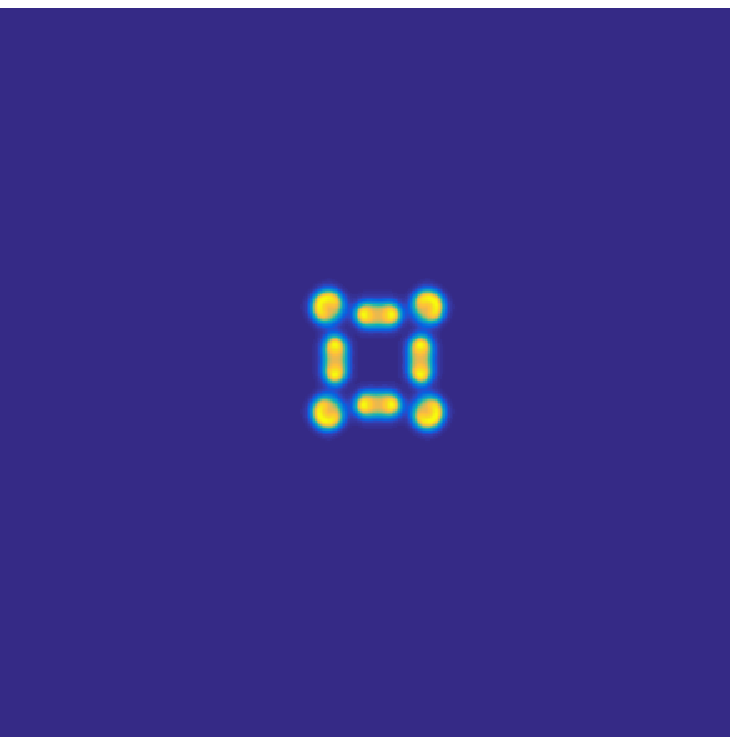}
		\end{minipage}}
		\subfigure[t=3400]{
			\begin{minipage}[t]{0.18\textwidth}
				\centering
				\includegraphics[width=1in]{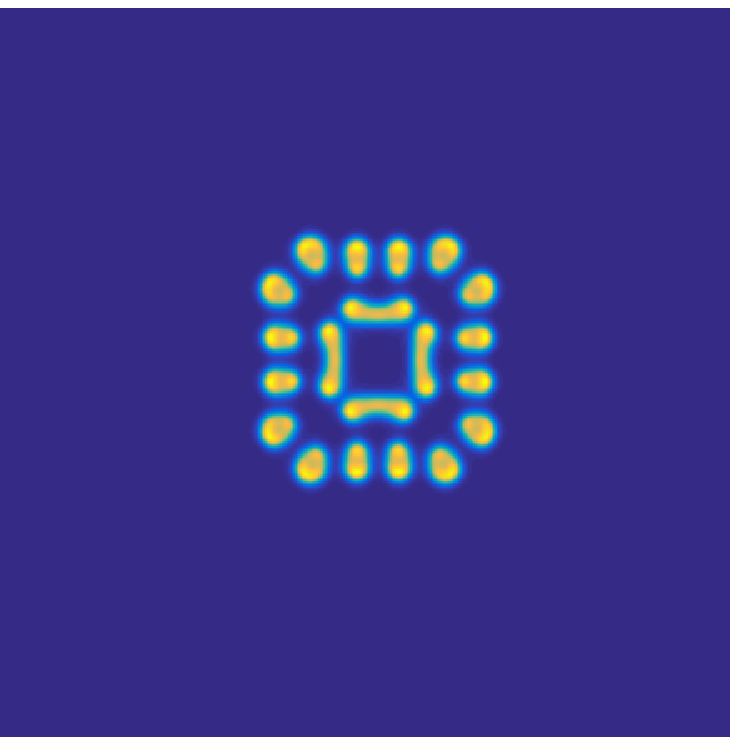}
		\end{minipage}}
		\subfigure[t=6000]{
			\begin{minipage}[t]{0.18\textwidth}
				\centering
				\includegraphics[width=1in]{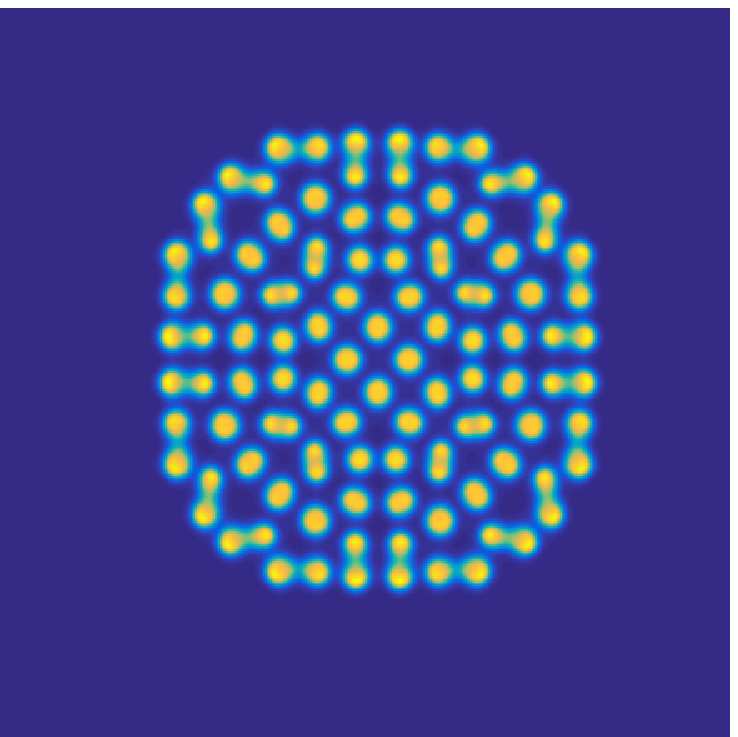}
		\end{minipage}}
		\subfigure[t=30000]{
			\begin{minipage}[t]{0.18\textwidth}
				\centering
				\includegraphics[width=1in]{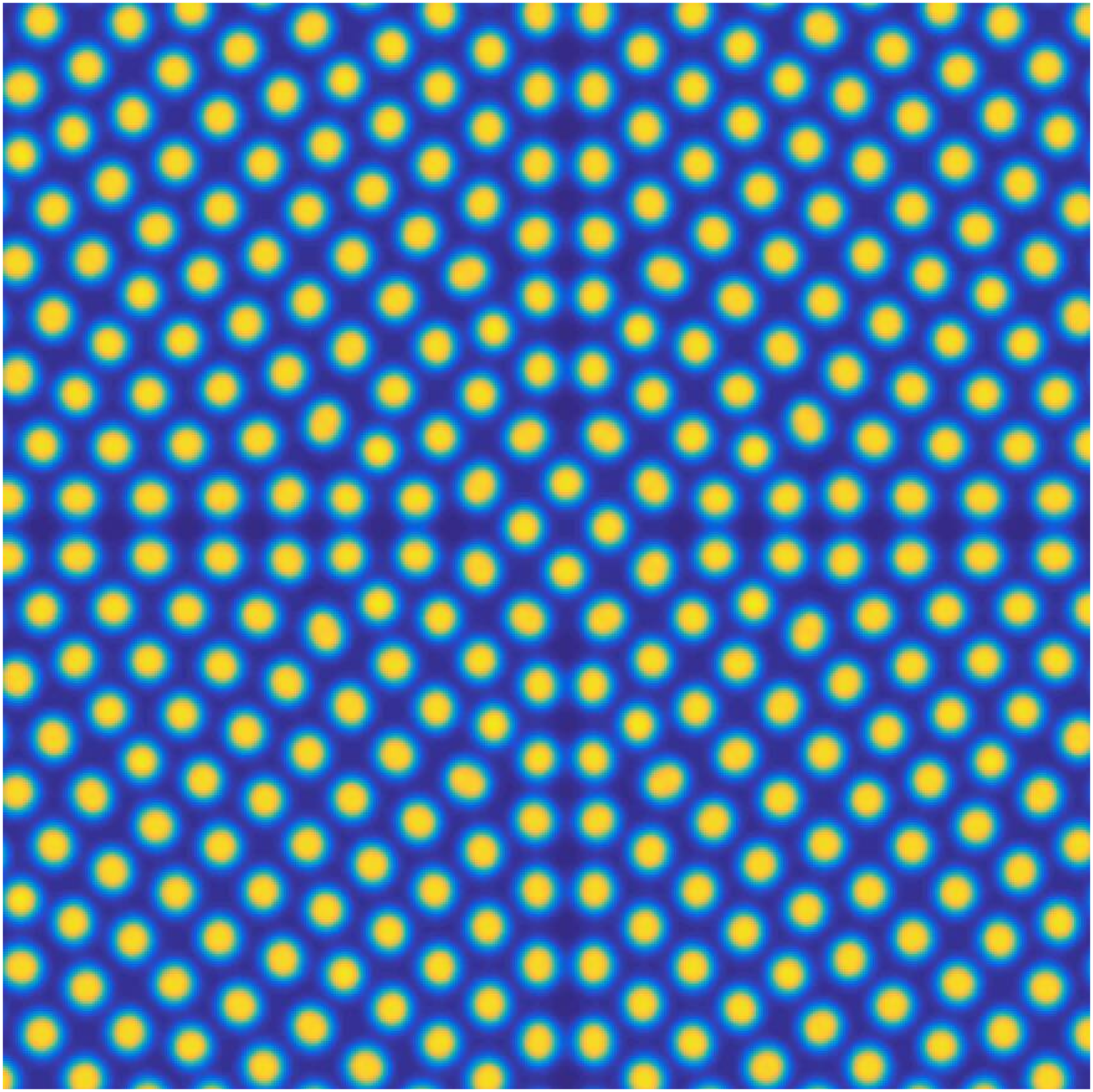}
		\end{minipage}}
		\subfigure[t=2000]{
			\begin{minipage}[t]{0.18\textwidth}
				\centering
				\includegraphics[width=1in]{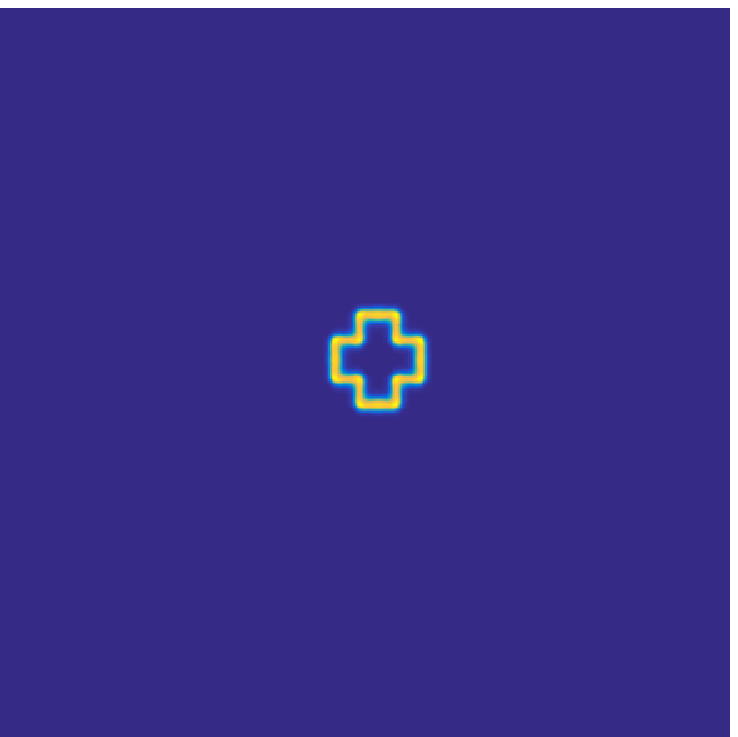}
		\end{minipage}}
		\subfigure[t=3400]{
			\begin{minipage}[t]{0.18\textwidth}
				\centering
				\includegraphics[width=1in]{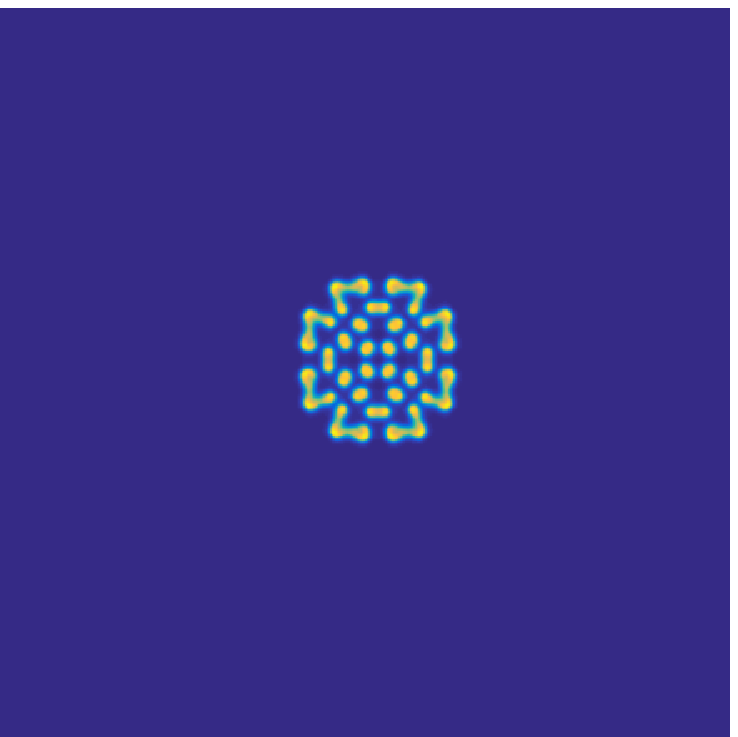}
		\end{minipage}}
		\subfigure[t=5000]{
			\begin{minipage}[t]{0.18\textwidth}
				\centering
				\includegraphics[width=1in]{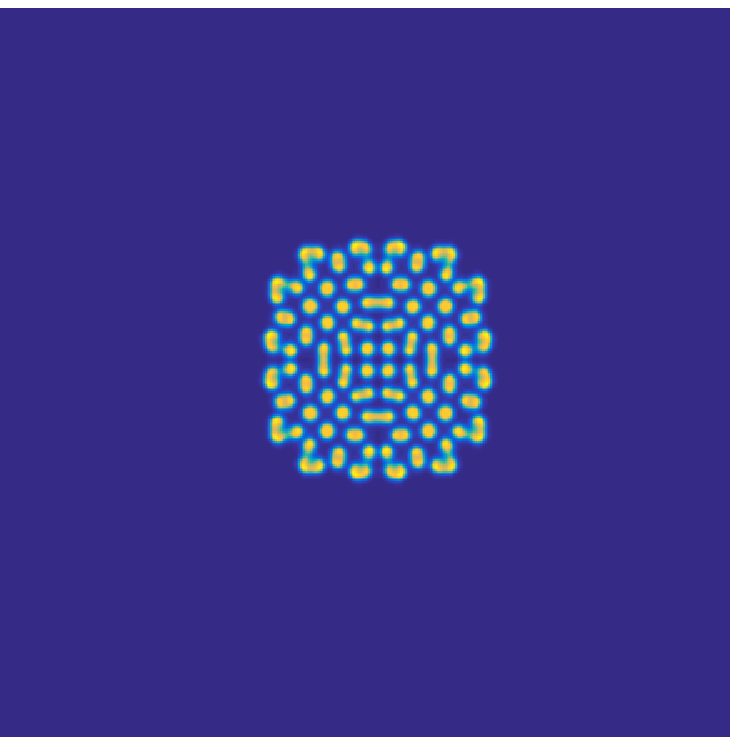}
		\end{minipage}}
		\subfigure[t=9000]{
			\begin{minipage}[t]{0.18\textwidth}
				\centering
				\includegraphics[width=1in]{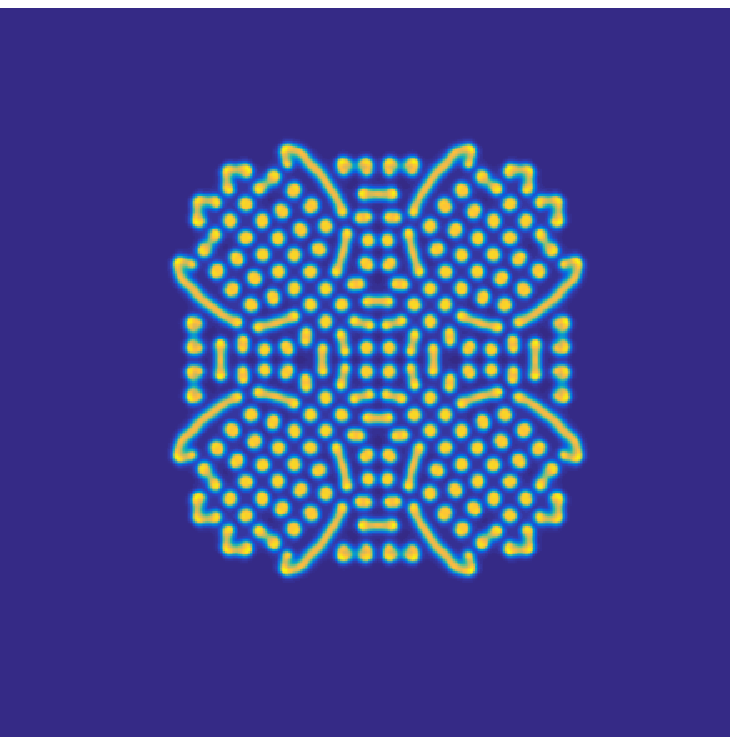}
		\end{minipage}}
		\subfigure[t=30000]{
			\begin{minipage}[t]{0.18\textwidth}
				\centering
				\includegraphics[width=1in]{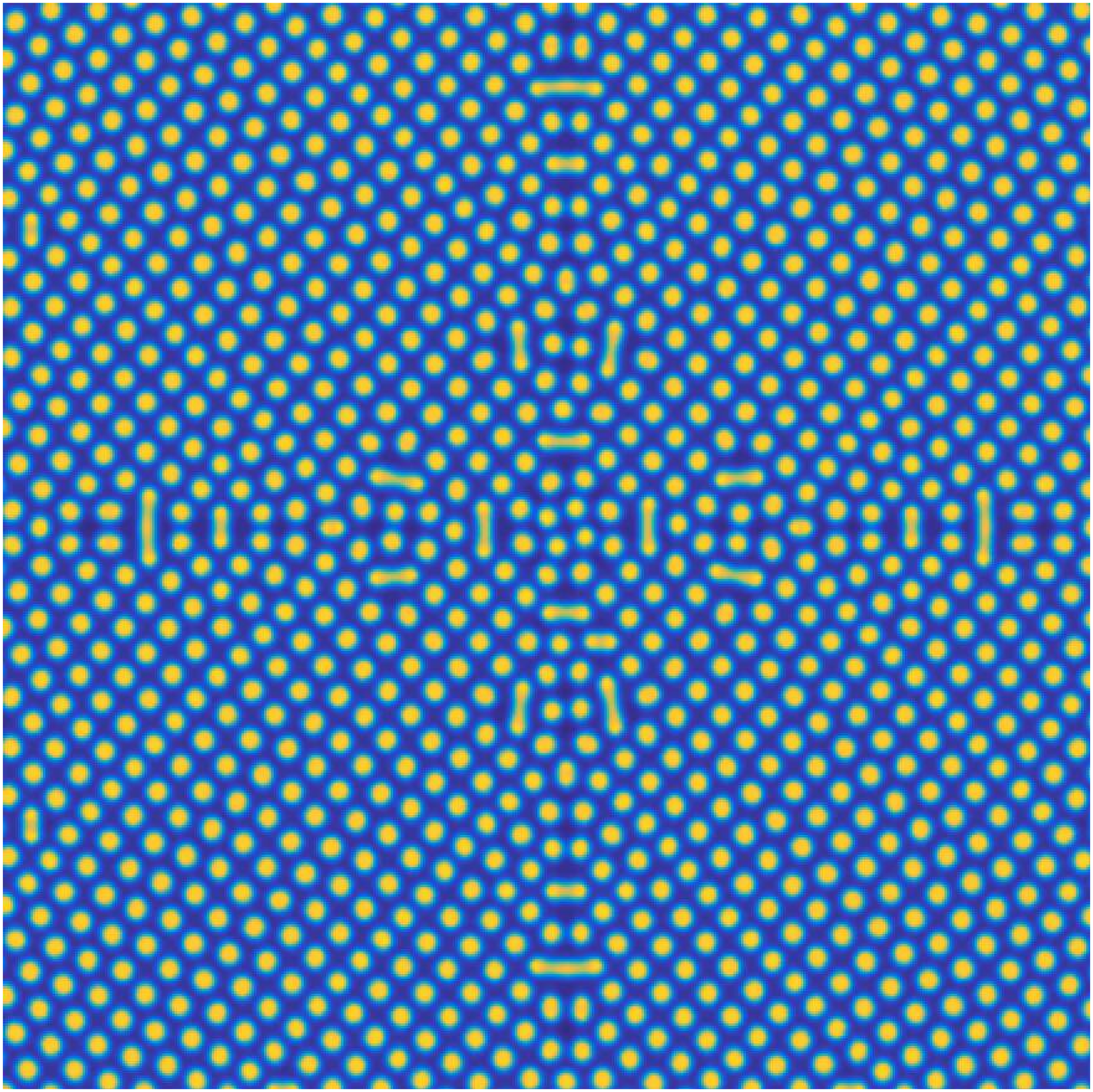}
		\end{minipage}}
		\caption{Example \ref{ex4.2}: Evolution of the solution $v$ with $\lambda= 0.063$: (a)-(e)  $\alpha=2$; (f)-(j) $\alpha=1.7$; (k)-(o) $\alpha=1.5$, the initial conditions are given by equation \eqref{eq4.7}.}
		\label{fig3a}
	\end{figure}
	\begin{example}[Fractional Gray--Scott model]\label{ex4.2}
		For the fractional Gray--Scott model (see \eqref{eq4.9} and \eqref{eq4.0}), $u$ and $v$ are the concentrations of $U$ and $V$ (see \eqref{eq4.11}), respectively, $K_{u}$ and $K_{v}$ are the diffusion rates in the process, $F$ is the dimensionless feed rate, and $\lambda$ is the dimensionless rate constant of the second reaction. In order to illustrate the effect of varying $\alpha$ in the model, we take the spatially initial condition as
		\begin{equation}\begin{aligned}
		\label{eq4.7}
		(u,v)=\left\{ \begin{array}{ll}
		(1,0)\quad (x_1,x_2)\in \Omega\setminus O_{c},\\
		(\frac{1}{2},\frac{1}{4})\quad (x_1,x_2)\in O_{c},
		\end{array}\right.
		\end{aligned}\end{equation}where $\Omega=(-1,2)^{2}$ and $O_{c}=\{(x_1,x_2)|(x_1-0.5)^{2}+(x_2-0.5)^{2}\leq0.04^{2}\}$.
	\end{example}
	\begin{figure}[htbp]
		\centering
		\subfigure[t=1000]{
			\begin{minipage}[t]{0.18\textwidth}
				\centering
				\includegraphics[width=1in]{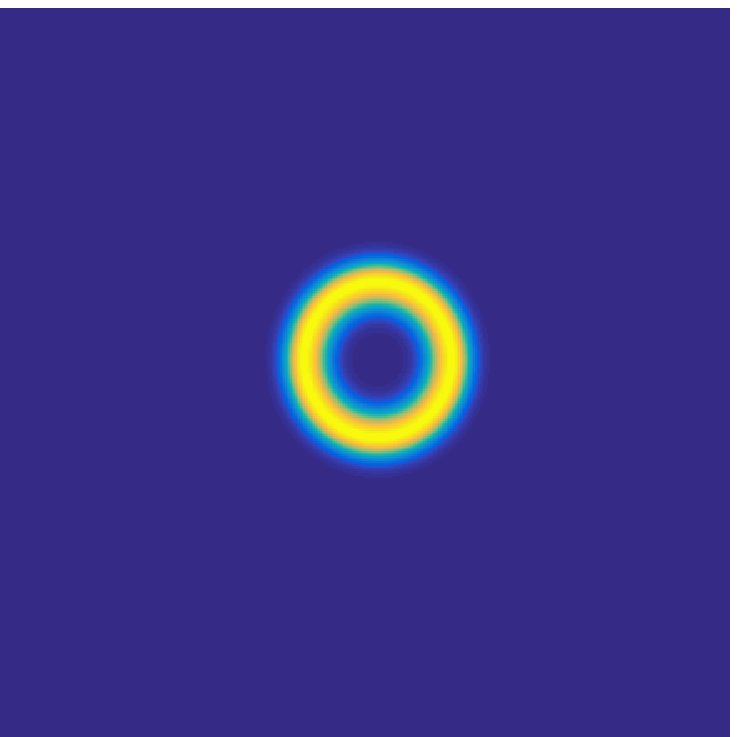}
		\end{minipage}}
		\subfigure[t=6000]{
			\begin{minipage}[t]{0.18\textwidth}
				\centering
				\includegraphics[width=1in]{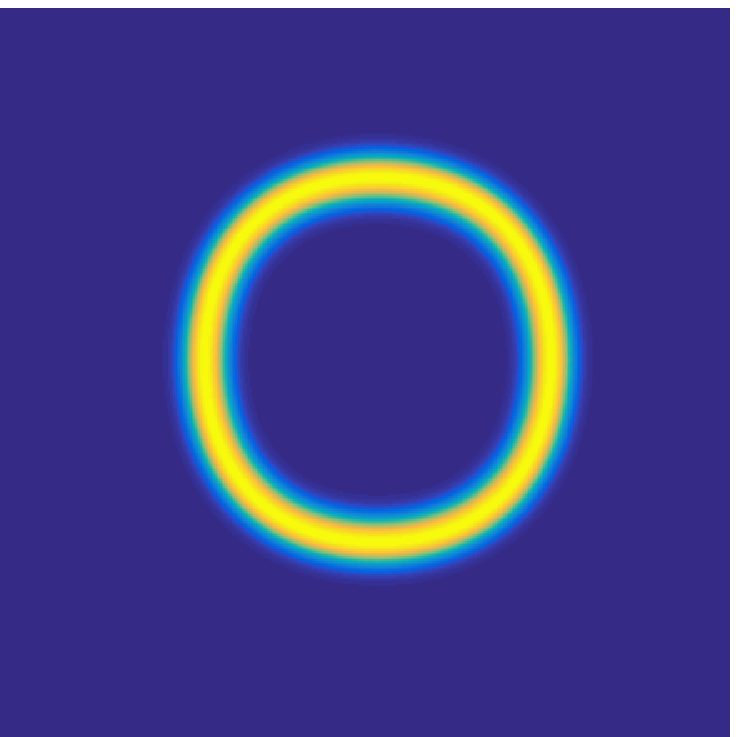}
		\end{minipage}}
		\subfigure[t=9000]{
			\begin{minipage}[t]{0.18\textwidth}
				\centering
				\includegraphics[width=1in]{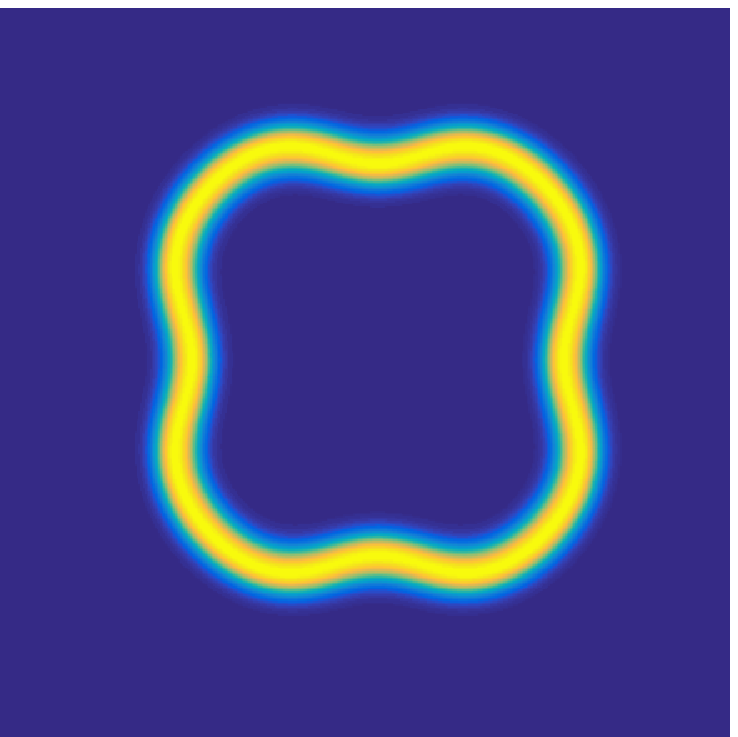}
		\end{minipage}}
		\subfigure[t=15000]{
			\begin{minipage}[t]{0.18\textwidth}
				\centering
				\includegraphics[width=1in]{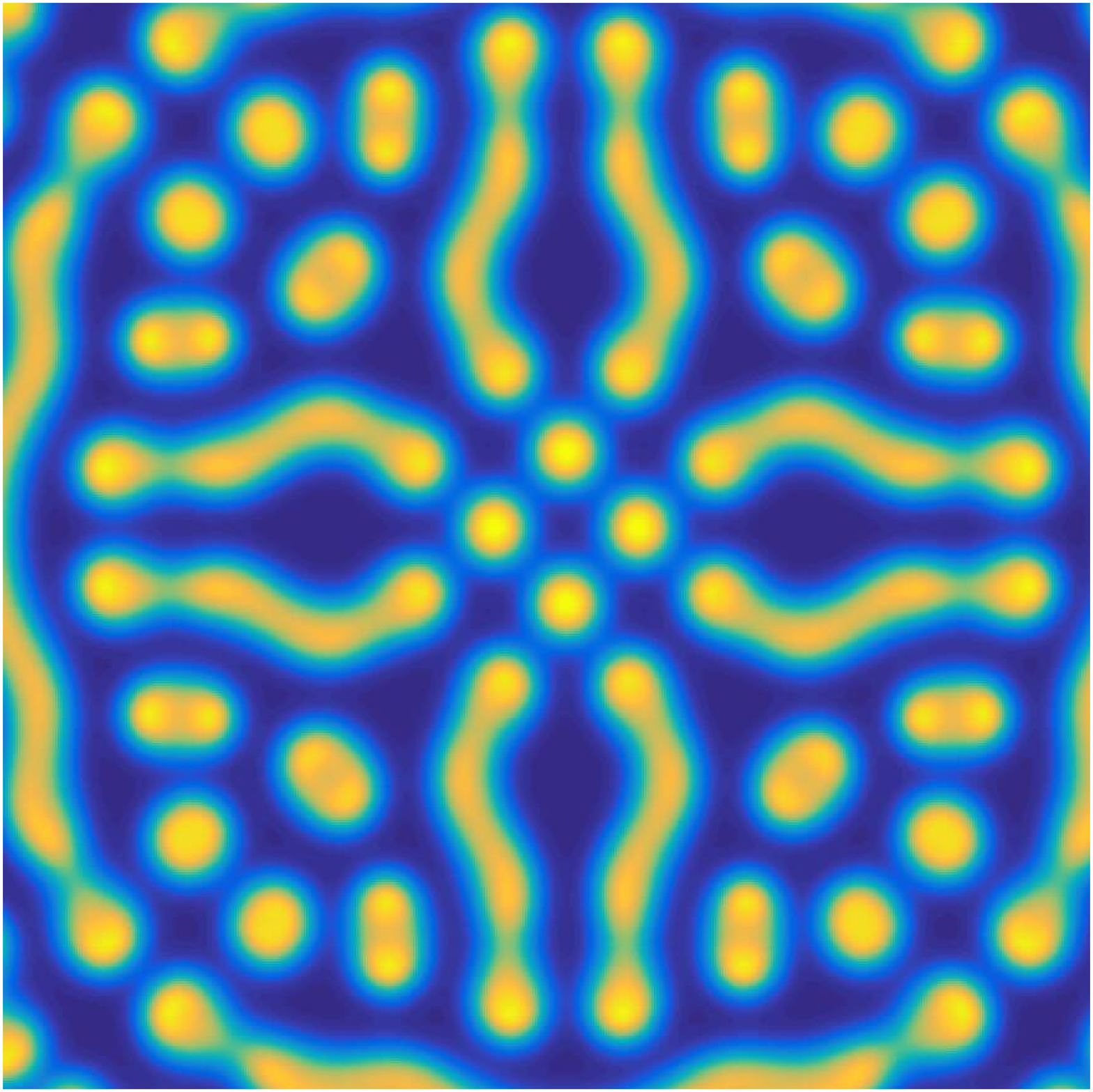}
		\end{minipage}}
		\subfigure[t=30000]{
			\begin{minipage}[t]{0.18\textwidth}
				\centering
				\includegraphics[width=1in]{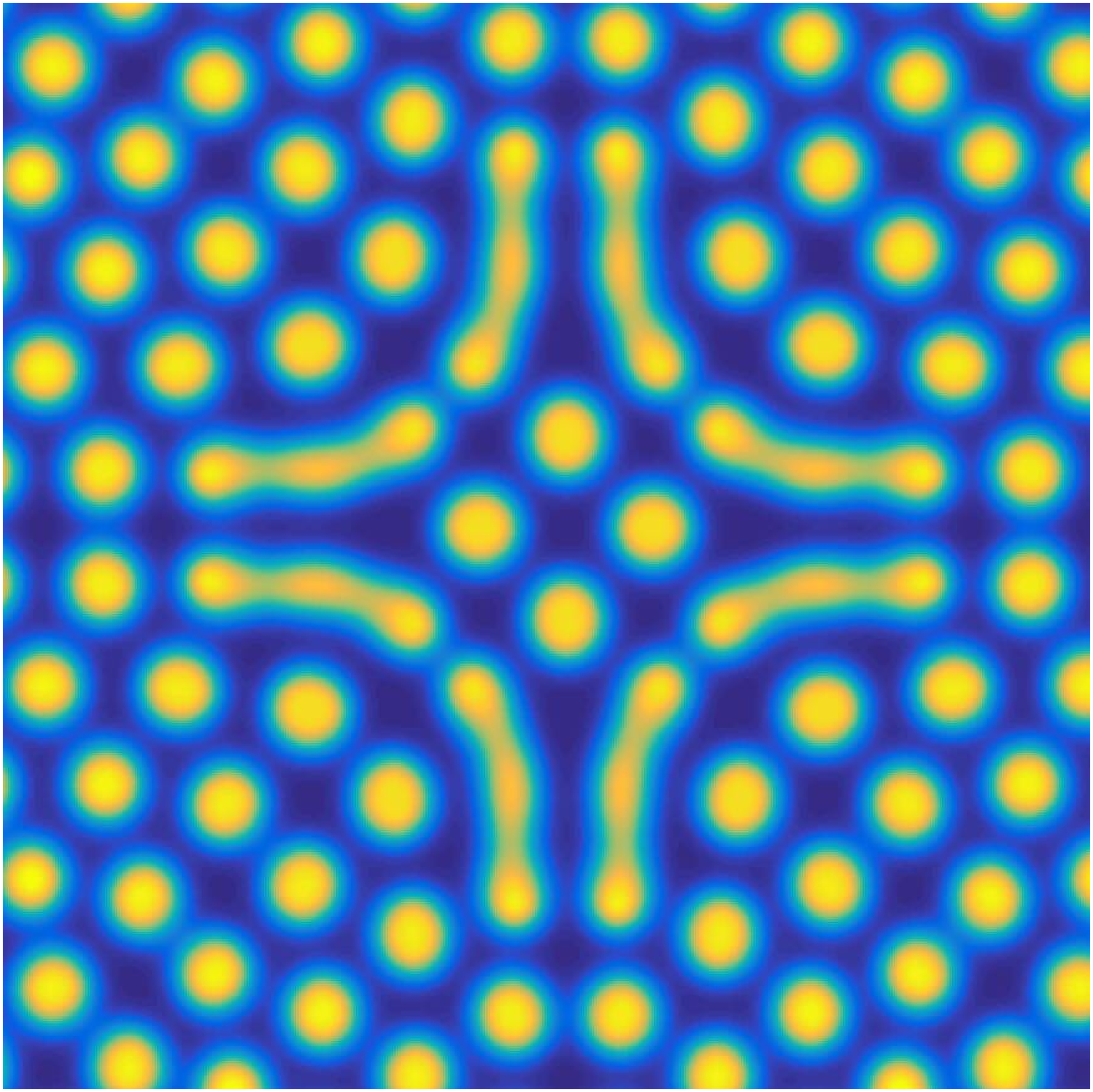}
		\end{minipage}}
		\subfigure[t=2000]{
			\begin{minipage}[t]{0.18\textwidth}
				\centering
				\includegraphics[width=1in]{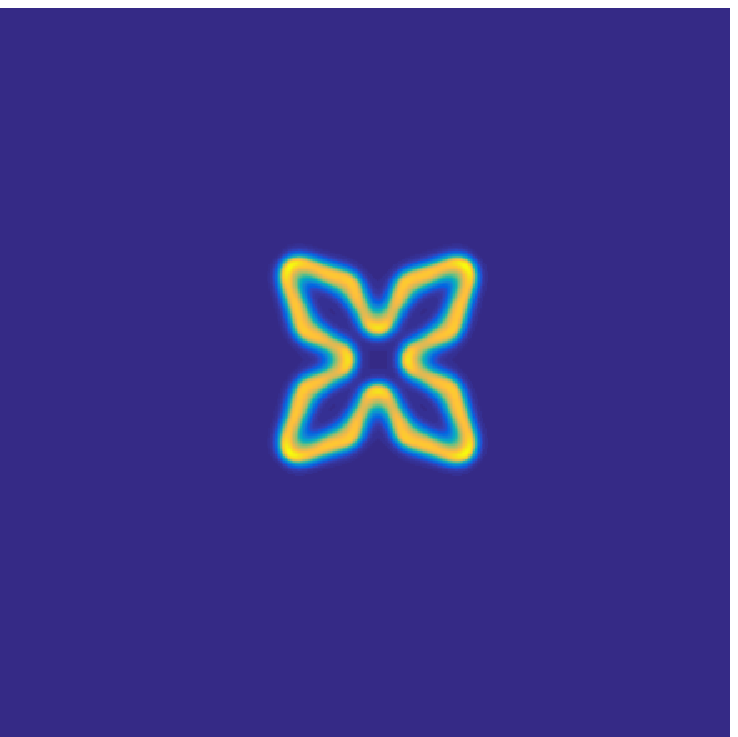}
		\end{minipage}}
		\subfigure[t=3400]{
			\begin{minipage}[t]{0.18\textwidth}
				\centering
				\includegraphics[width=1in]{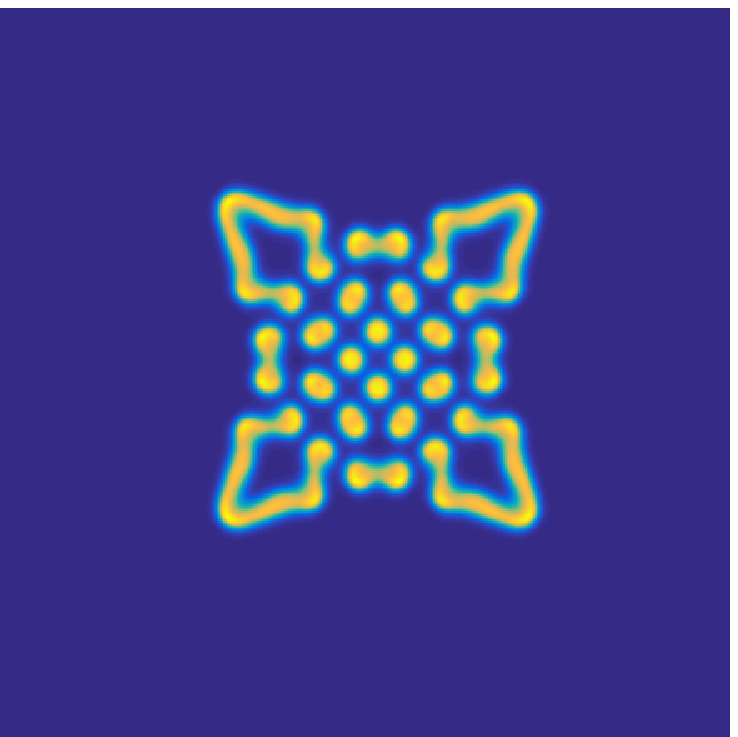}
		\end{minipage}}
		\subfigure[t=6000]{
			\begin{minipage}[t]{0.18\textwidth}
				\centering
				\includegraphics[width=1in]{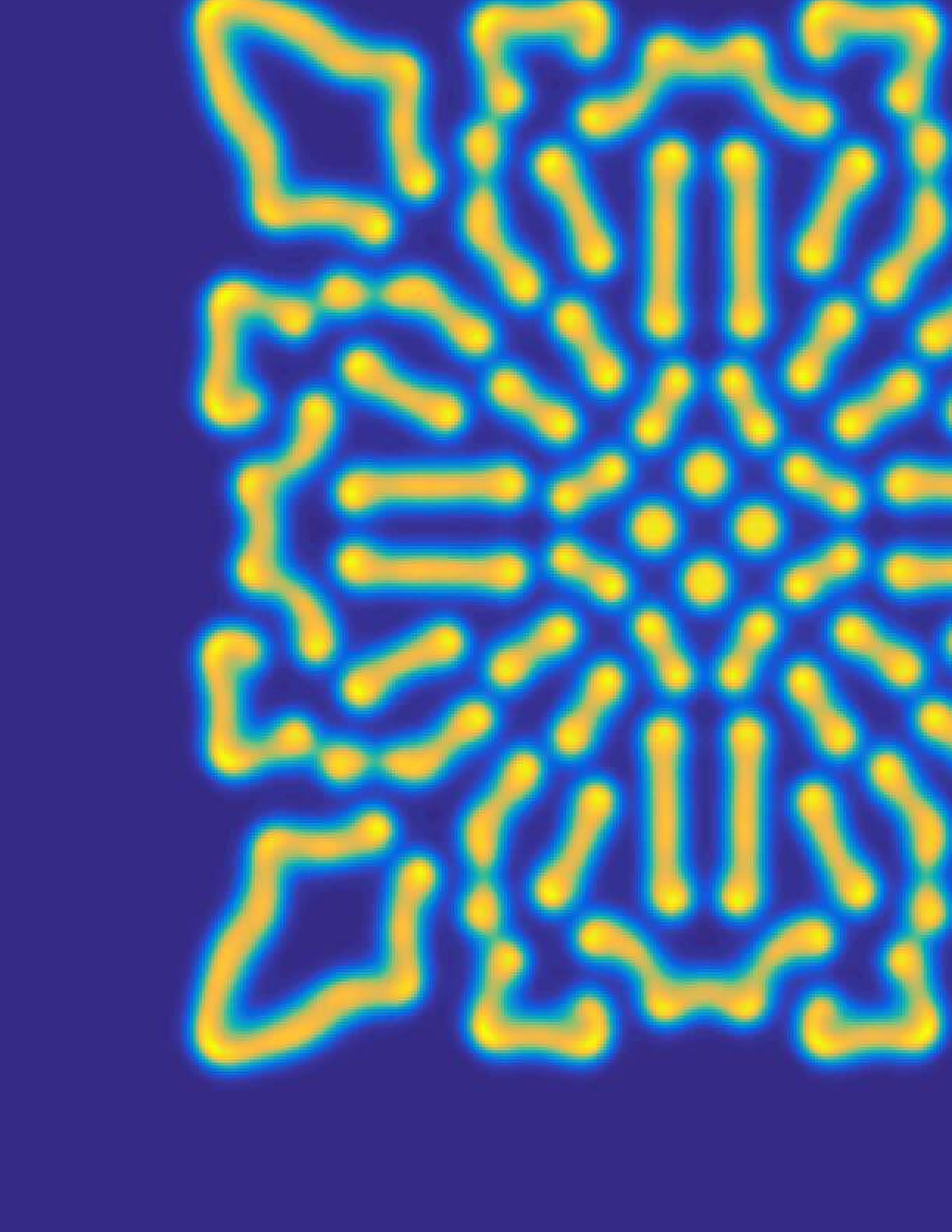}
		\end{minipage}}
		\subfigure[t=9000]{
			\begin{minipage}[t]{0.18\textwidth}
				\centering
				\includegraphics[width=1in]{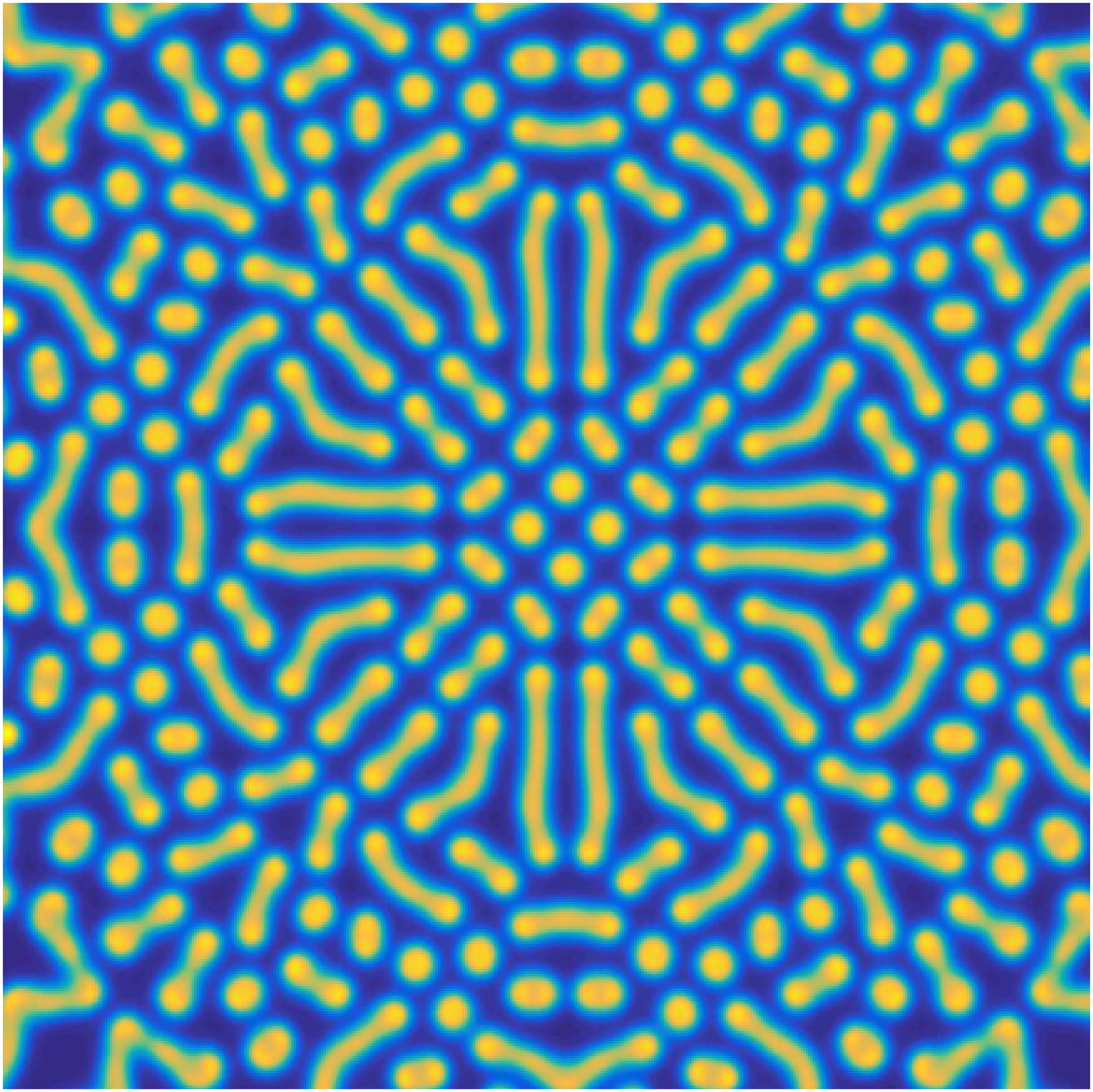}
		\end{minipage}}
		\subfigure[t=30000]{
			\begin{minipage}[t]{0.18\textwidth}
				\centering
				\includegraphics[width=1in]{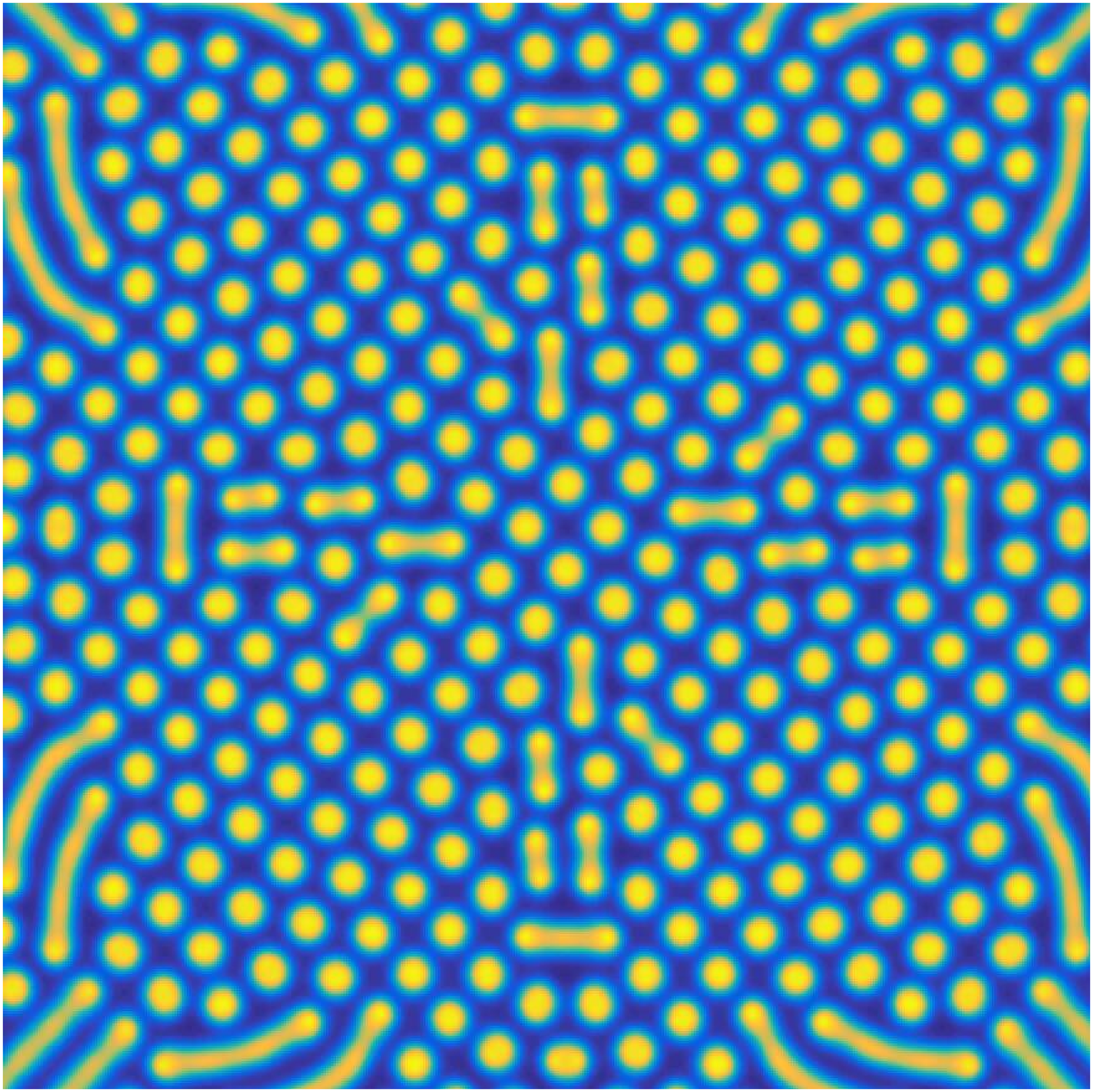}
		\end{minipage}}
		\subfigure[t=1000]{
			\begin{minipage}[t]{0.18\textwidth}
				\centering
				\includegraphics[width=1in]{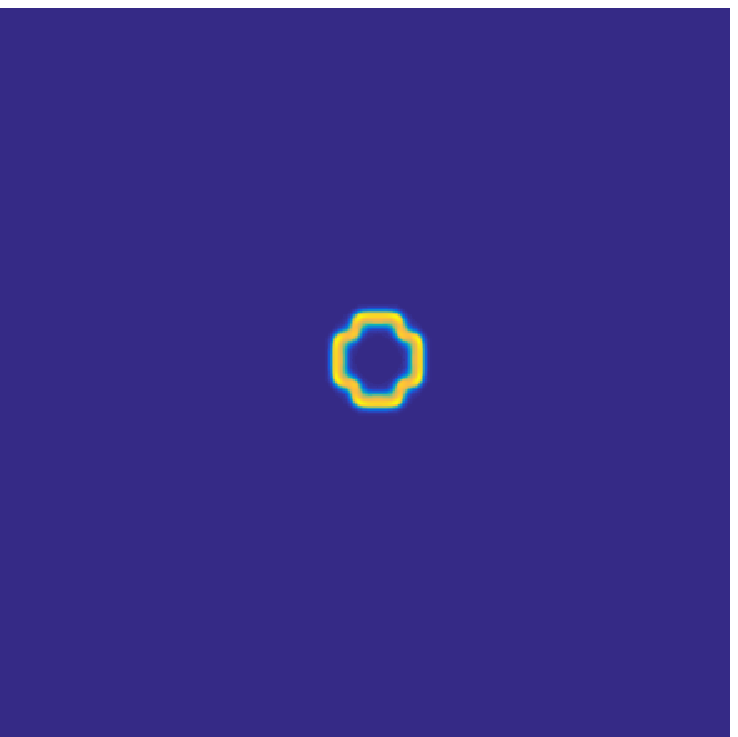}
		\end{minipage}}
		\subfigure[t=3000]{
			\begin{minipage}[t]{0.18\textwidth}
				\centering
				\includegraphics[width=1in]{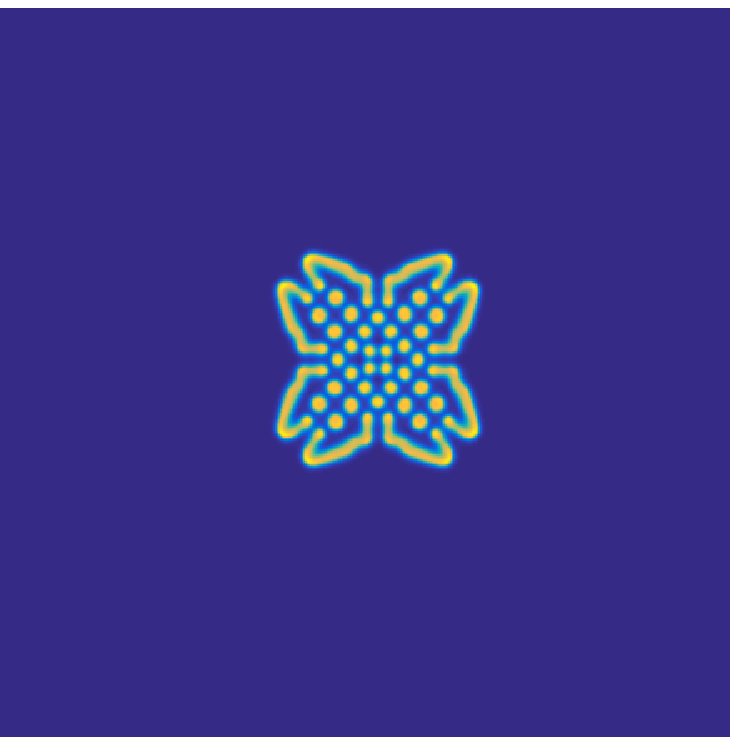}
		\end{minipage}}
		\subfigure[t=5000]{
			\begin{minipage}[t]{0.18\textwidth}
				\centering
				\includegraphics[width=1in]{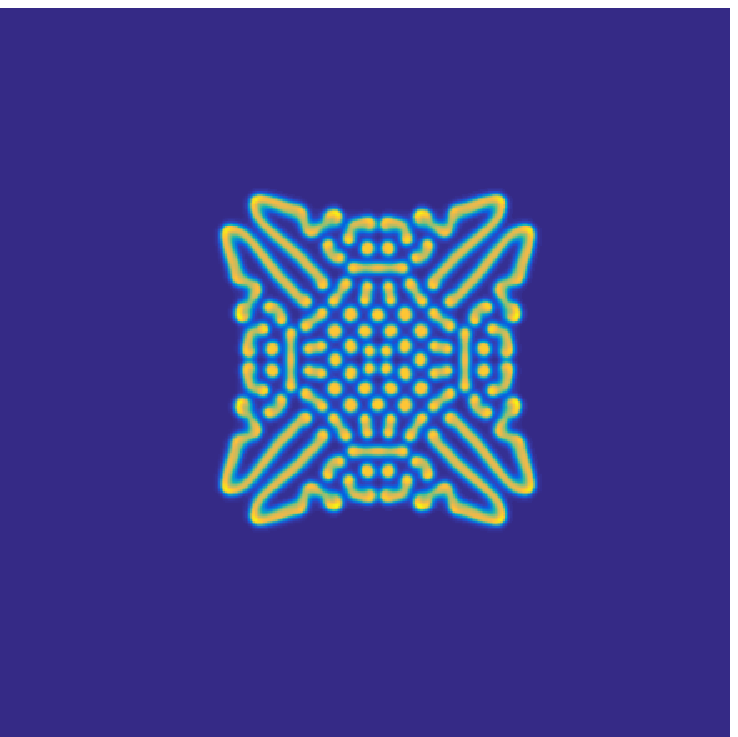}
		\end{minipage}}
		\subfigure[t=9000]{
			\begin{minipage}[t]{0.18\textwidth}
				\centering
				\includegraphics[width=1in]{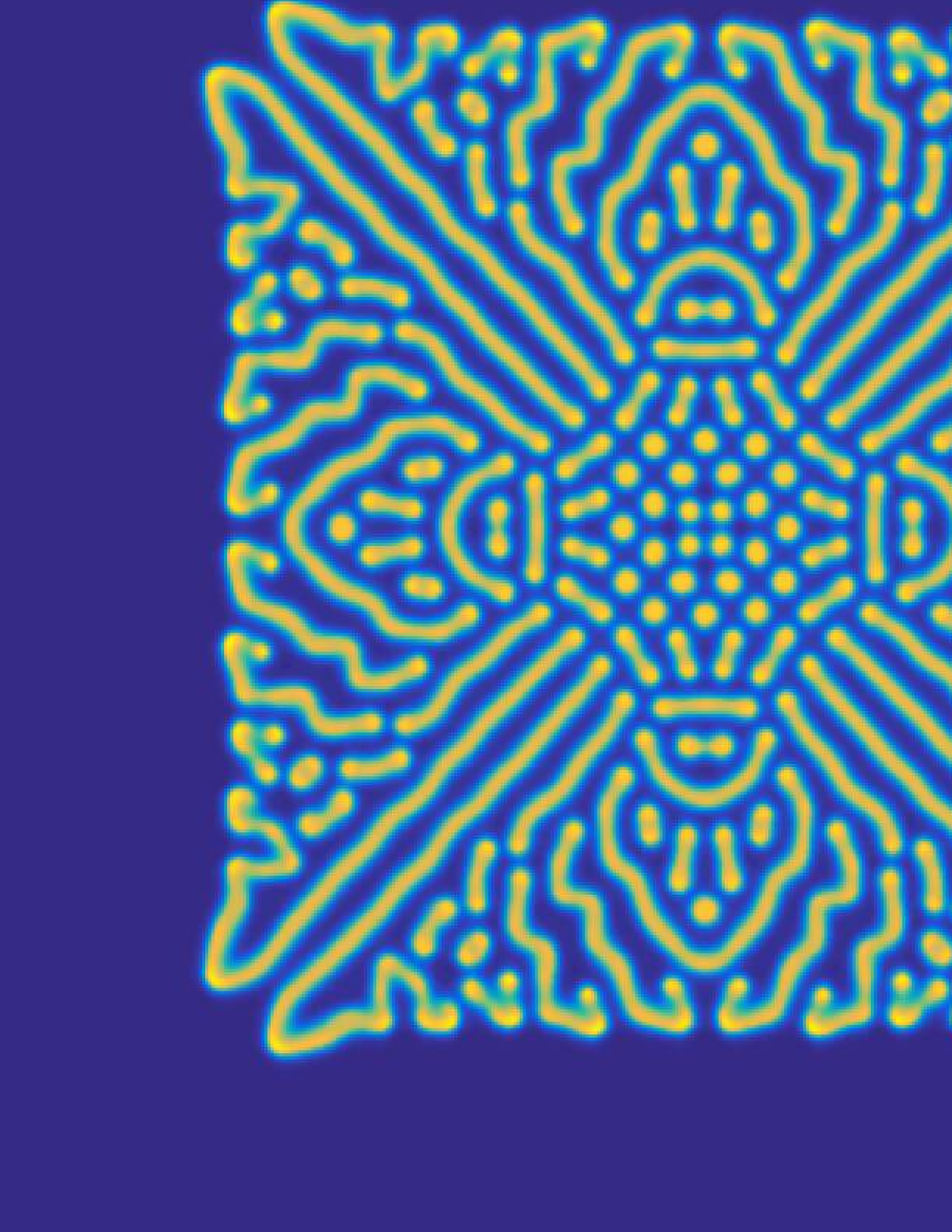}
		\end{minipage}}
		\subfigure[t=30000]{
			\begin{minipage}[t]{0.18\textwidth}
				\centering
				\includegraphics[width=1in]{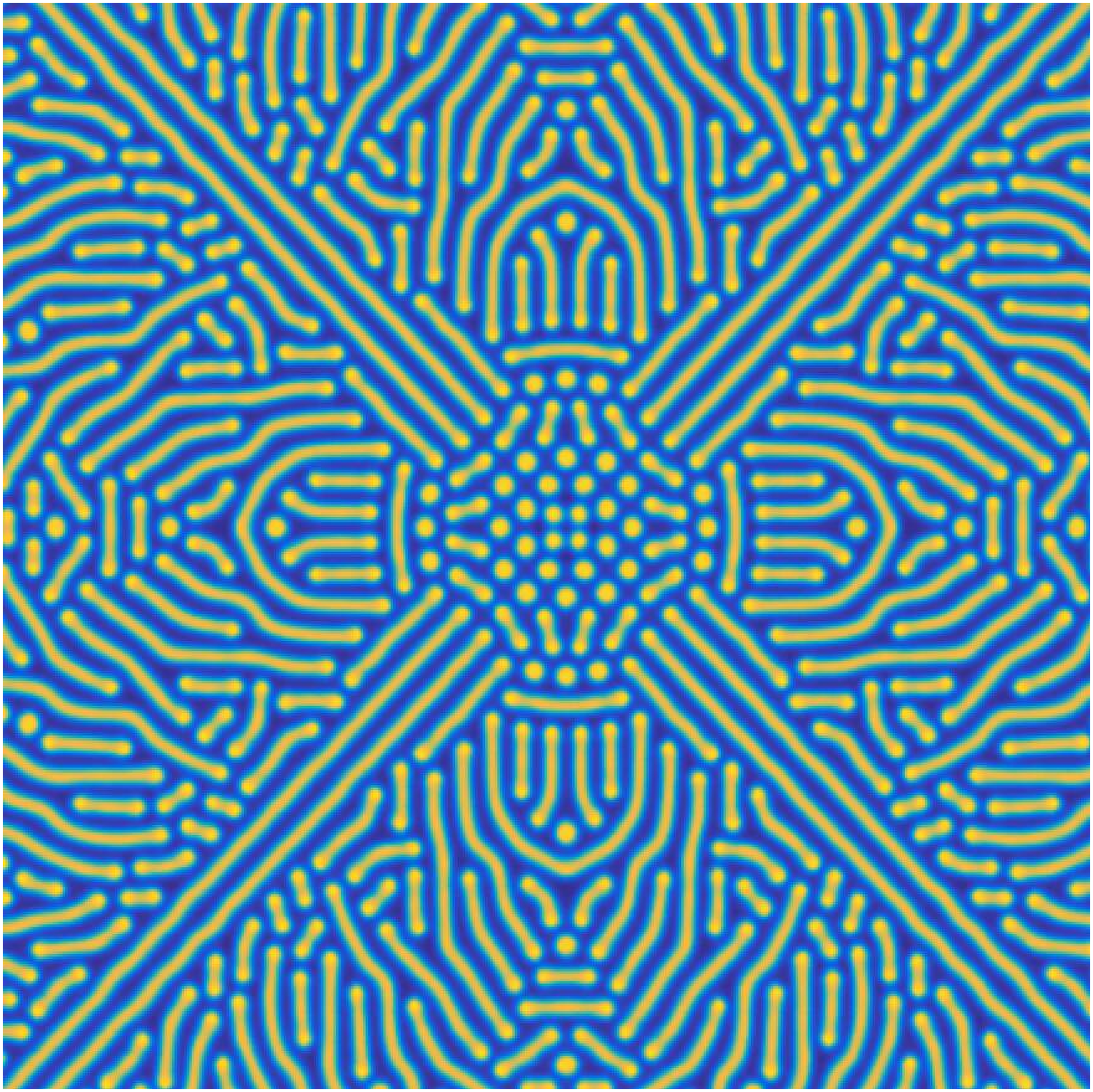}
		\end{minipage}}
		\caption{Example \ref{ex4.2}: Evolution of the solution $v$ with $\lambda= 0.061$: (a)-(e) $\alpha=2$; (f)-(j)  $\alpha=1.7$; (k)-(o) $\alpha=1.5$, the initial conditions are presented as equation \eqref{eq4.7}.}
		\label{fig3b}
	\end{figure}
	\begin{figure}[htbp]
		\centering
		\subfigure[t=200]{
			\begin{minipage}[t]{0.18\textwidth}
				\centering
				\includegraphics[width=1in]{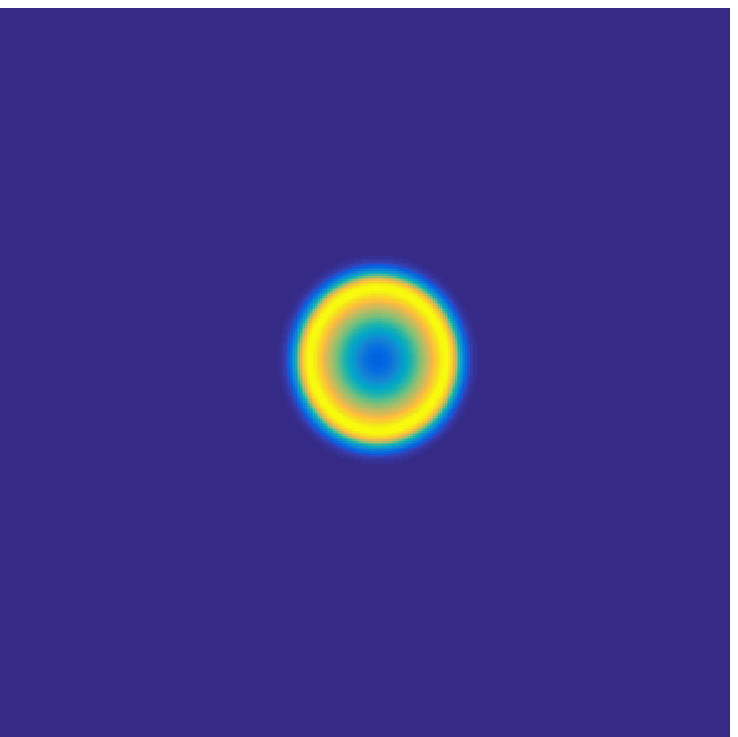}
		\end{minipage}}
		\subfigure[t=800]{
			\begin{minipage}[t]{0.18\textwidth}
				\centering
				\includegraphics[width=1in]{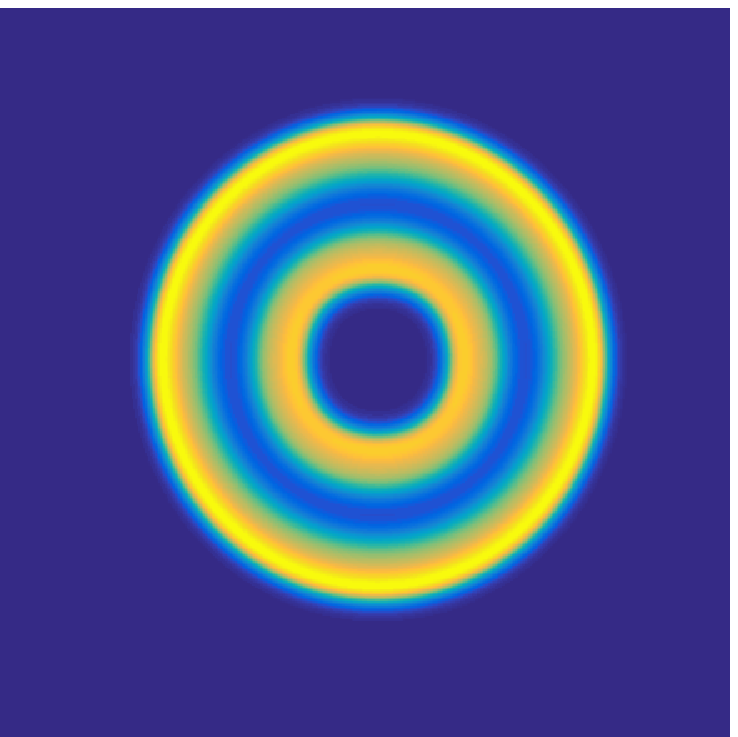}
		\end{minipage}}
		\subfigure[t=2000]{
			\begin{minipage}[t]{0.18\textwidth}
				\centering
				\includegraphics[width=1in]{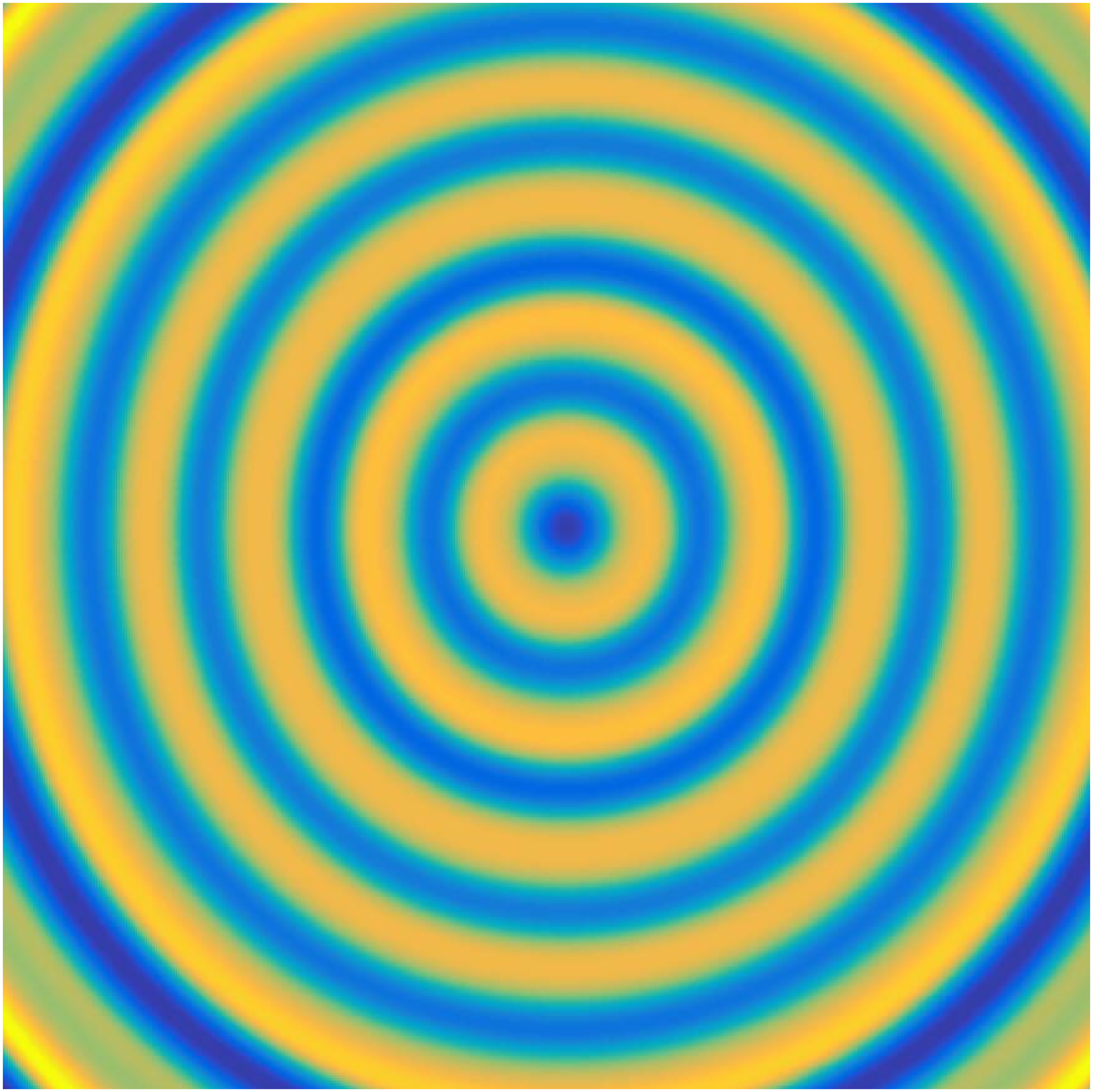}
		\end{minipage}}
		\subfigure[t=15000]{
			\begin{minipage}[t]{0.18\textwidth}
				\centering
				\includegraphics[width=1in]{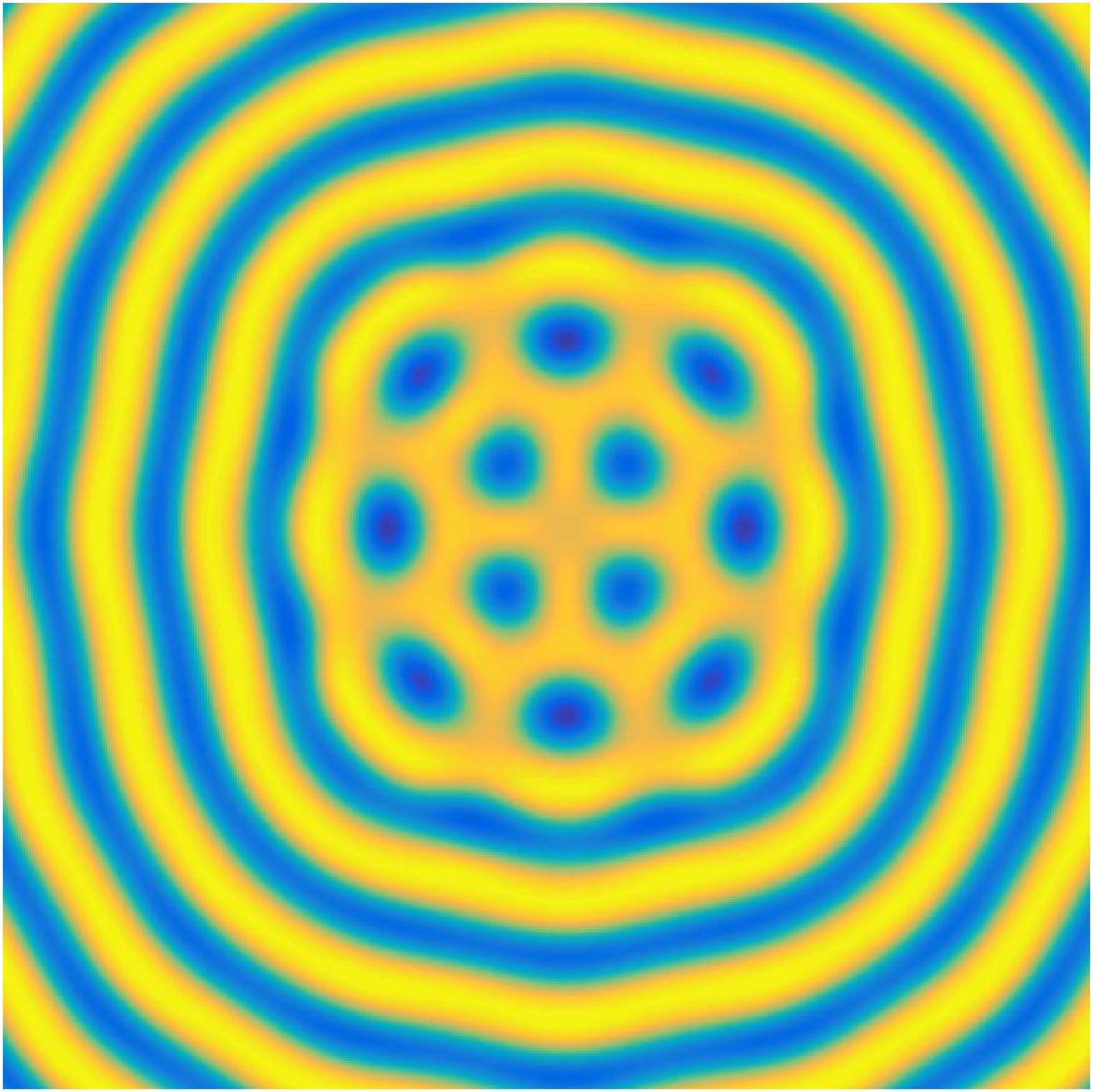}
		\end{minipage}}
		\subfigure[t=30000]{
			\begin{minipage}[t]{0.18\textwidth}
				\centering
				\includegraphics[width=1in]{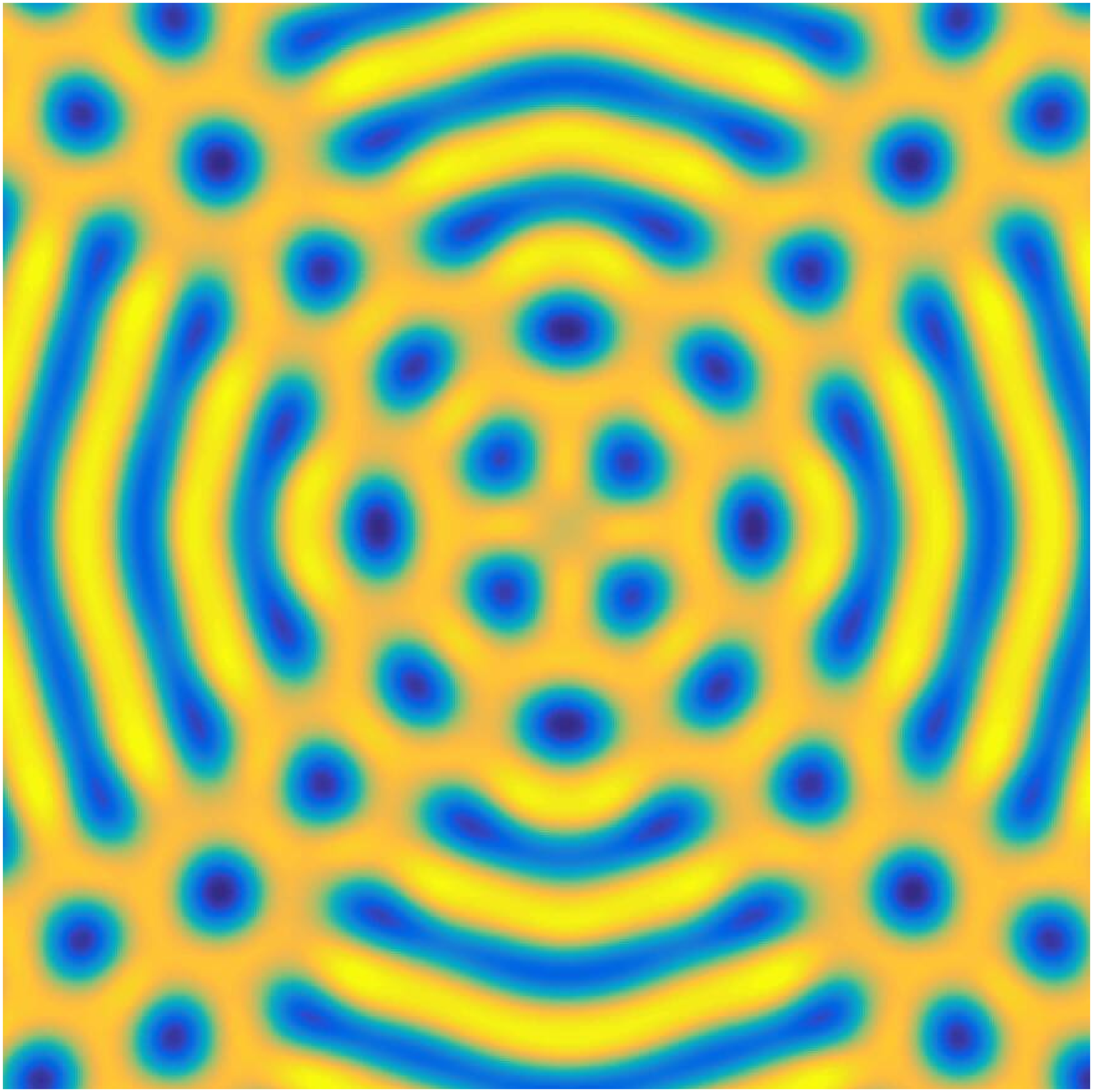}
		\end{minipage}}
		\subfigure[t=400]{
			\begin{minipage}[t]{0.18\textwidth}
				\centering
				\includegraphics[width=1in]{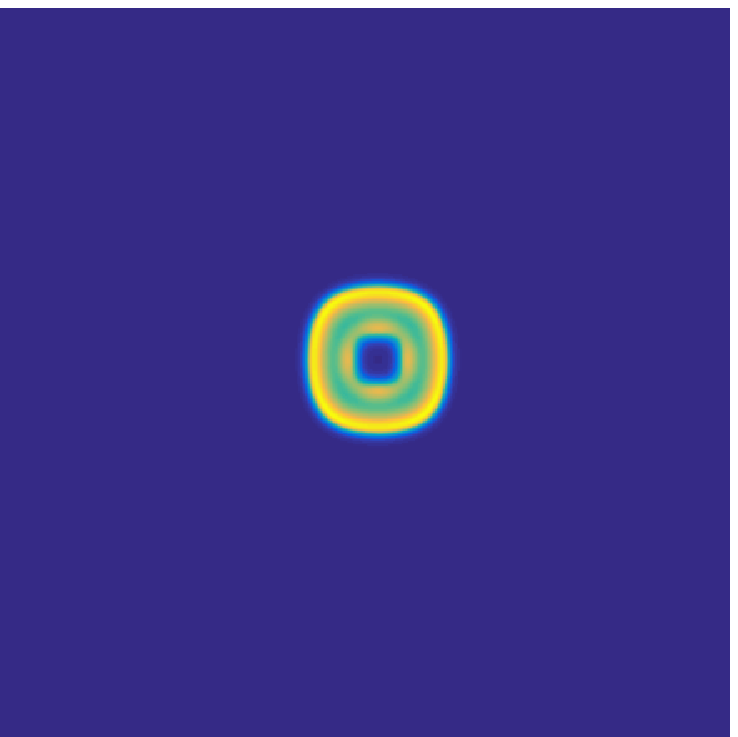}
		\end{minipage}}
		\subfigure[t=1600]{
			\begin{minipage}[t]{0.18\textwidth}
				\centering
				\includegraphics[width=1in]{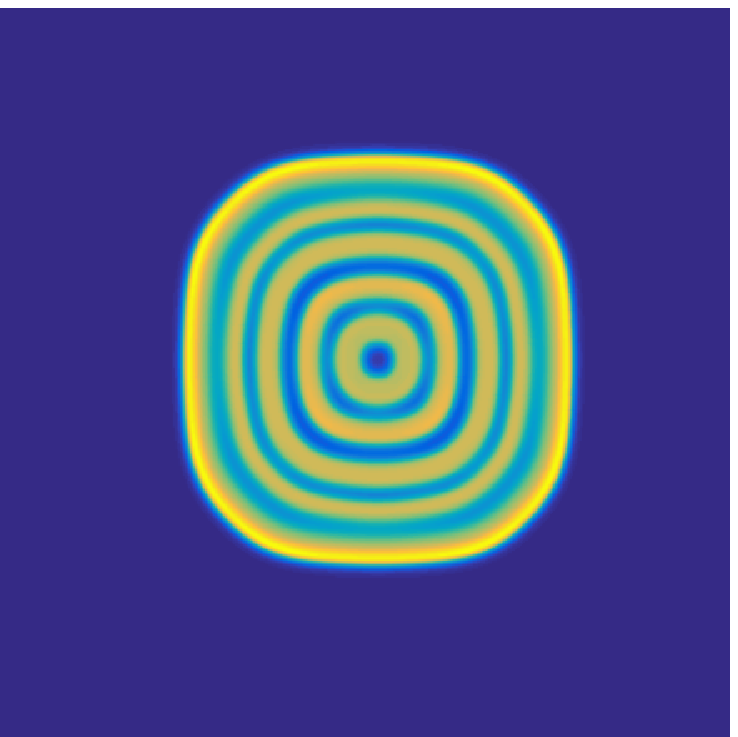}
		\end{minipage}}
		\subfigure[t=5000]{
			\begin{minipage}[t]{0.18\textwidth}
				\centering
				\includegraphics[width=1in]{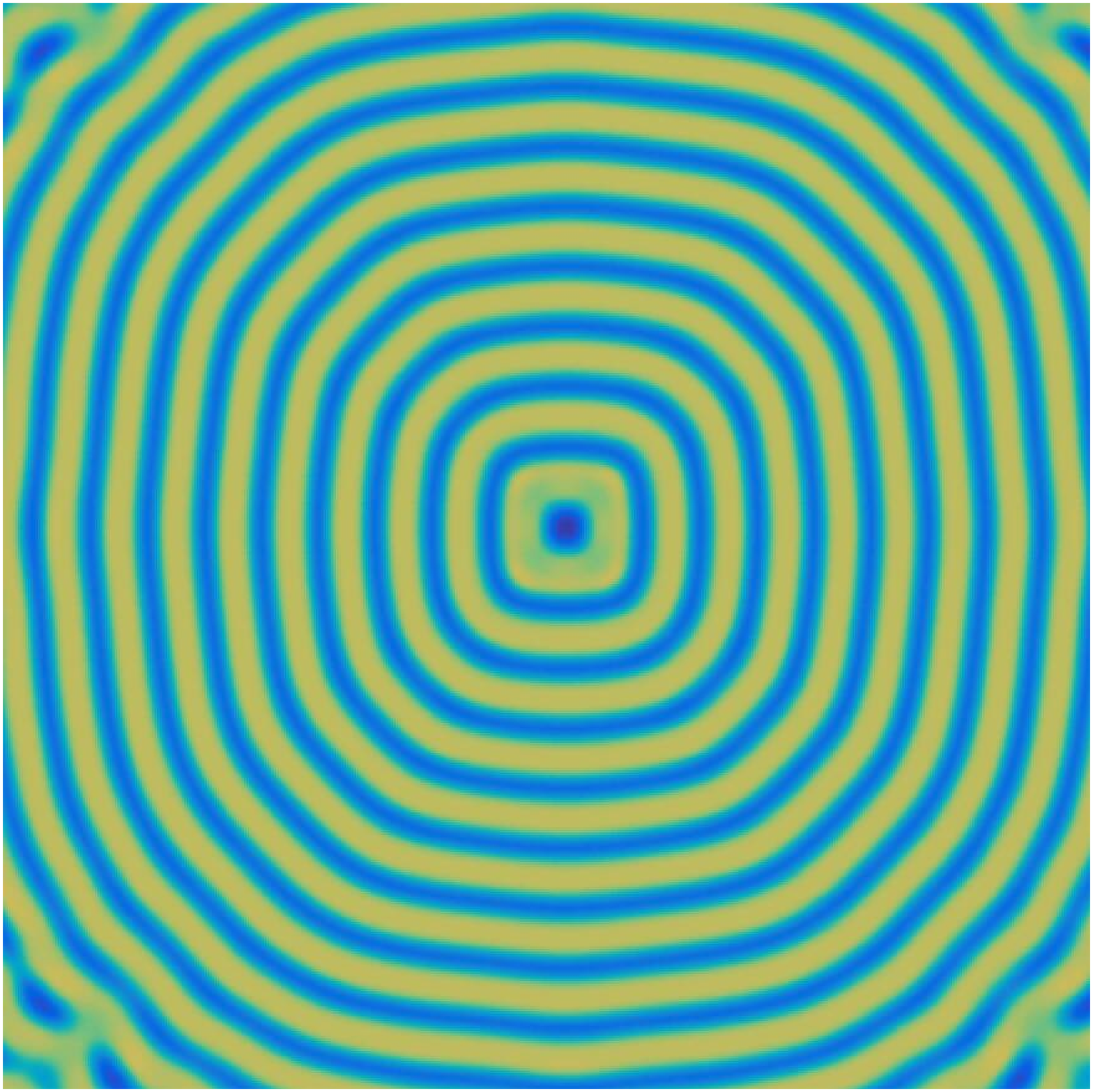}
		\end{minipage}}
		\subfigure[t=10000]{
			\begin{minipage}[t]{0.18\textwidth}
				\centering
				\includegraphics[width=1in]{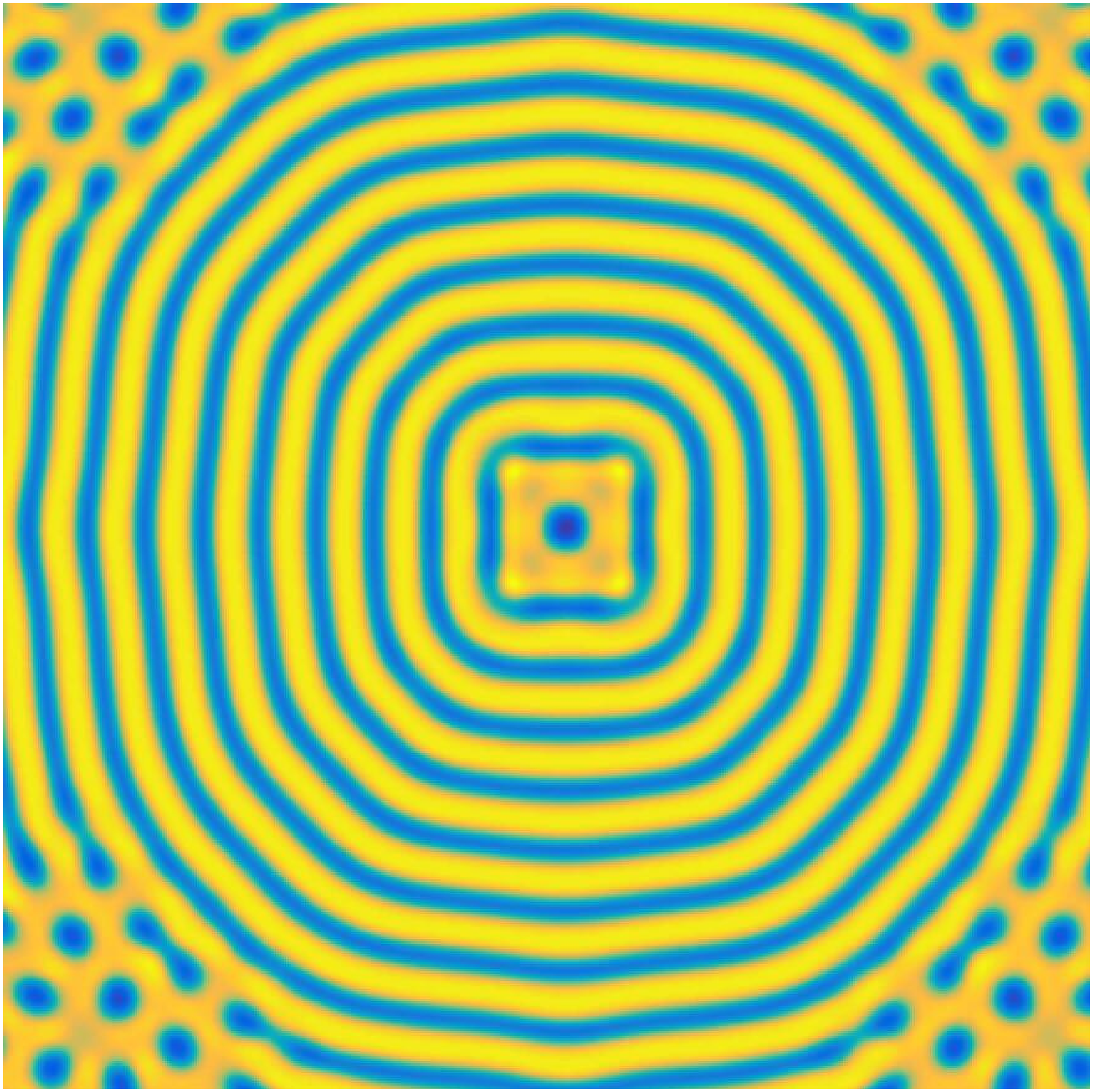}
		\end{minipage}}
		\subfigure[t=30000]{
			\begin{minipage}[t]{0.18\textwidth}
				\centering
				\includegraphics[width=1in]{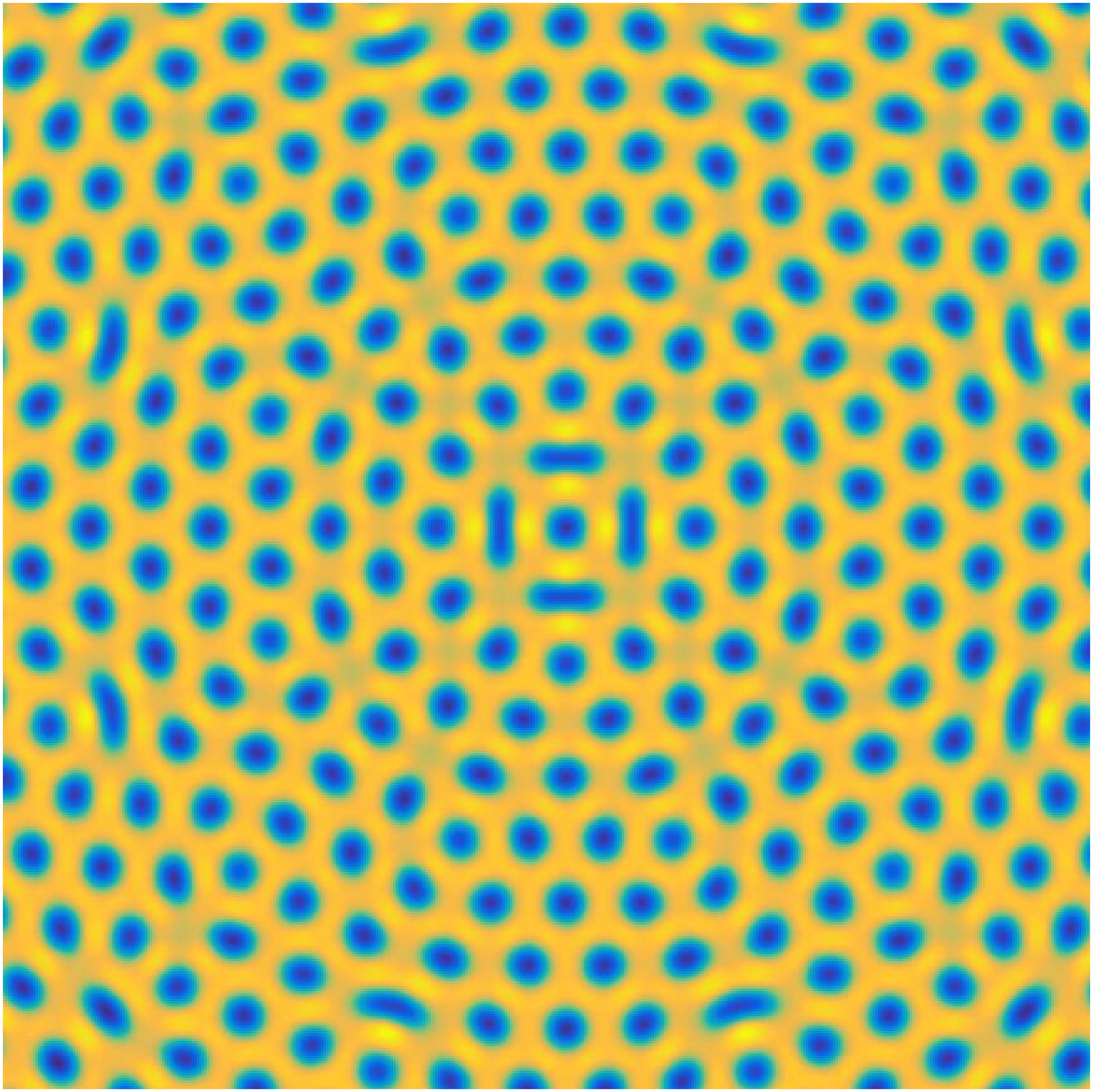}
		\end{minipage}}
		\subfigure[t=800]{
			\begin{minipage}[t]{0.18\textwidth}
				\centering
				\includegraphics[width=1in]{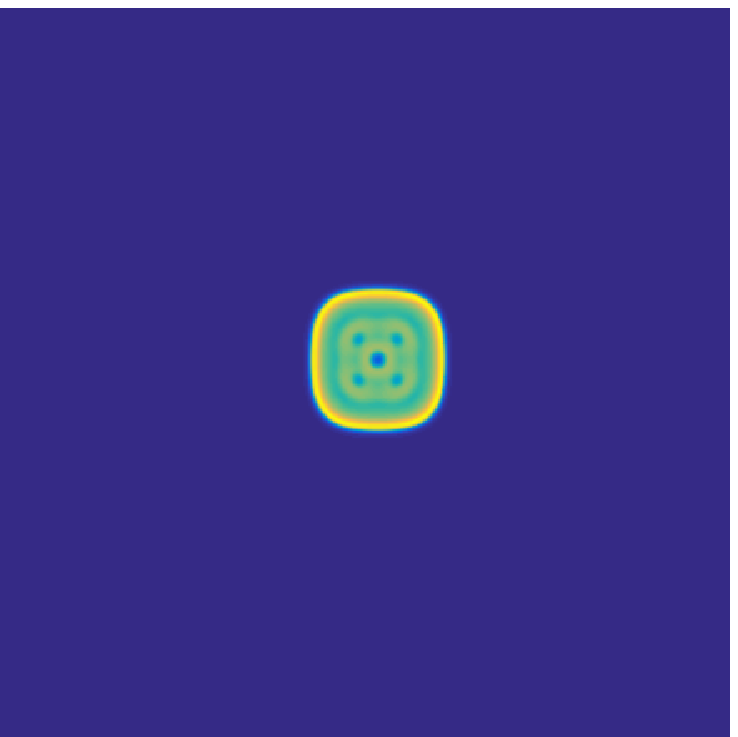}
		\end{minipage}}
		\subfigure[t=2600]{
			\begin{minipage}[t]{0.18\textwidth}
				\centering
				\includegraphics[width=1in]{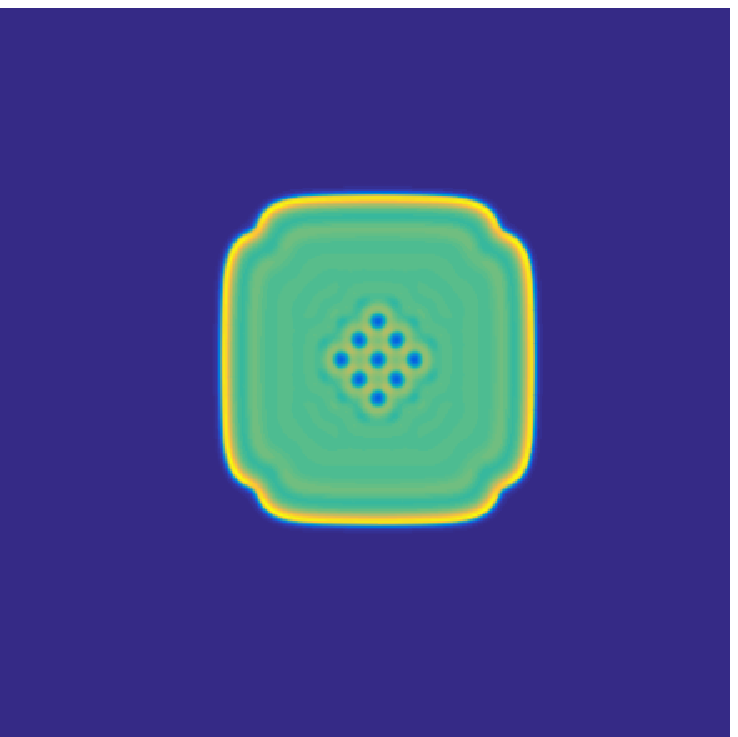}
		\end{minipage}}
		\subfigure[t=4000]{
			\begin{minipage}[t]{0.18\textwidth}
				\centering
				\includegraphics[width=1in]{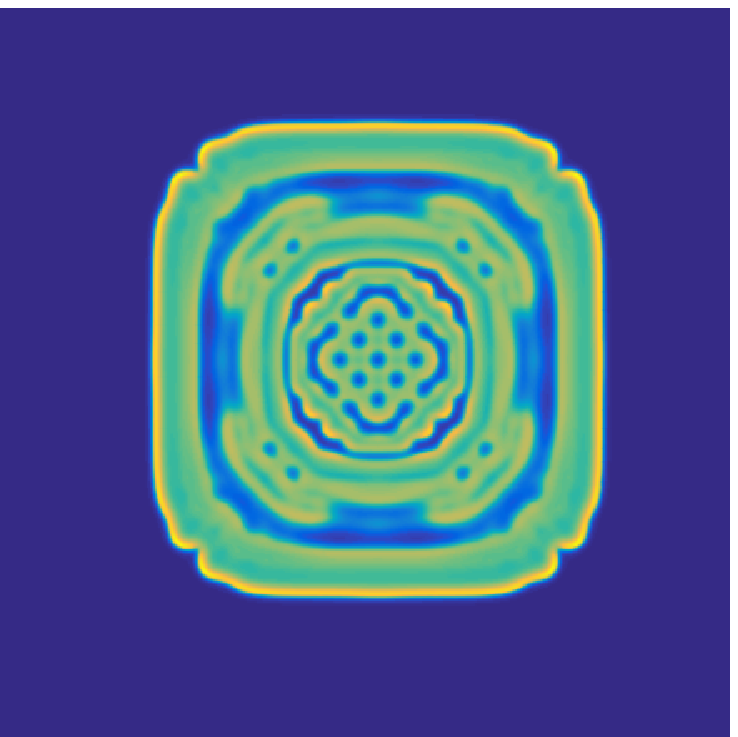}
		\end{minipage}}
		\subfigure[t=5800]{
			\begin{minipage}[t]{0.18\textwidth}
				\centering
				\includegraphics[width=1in]{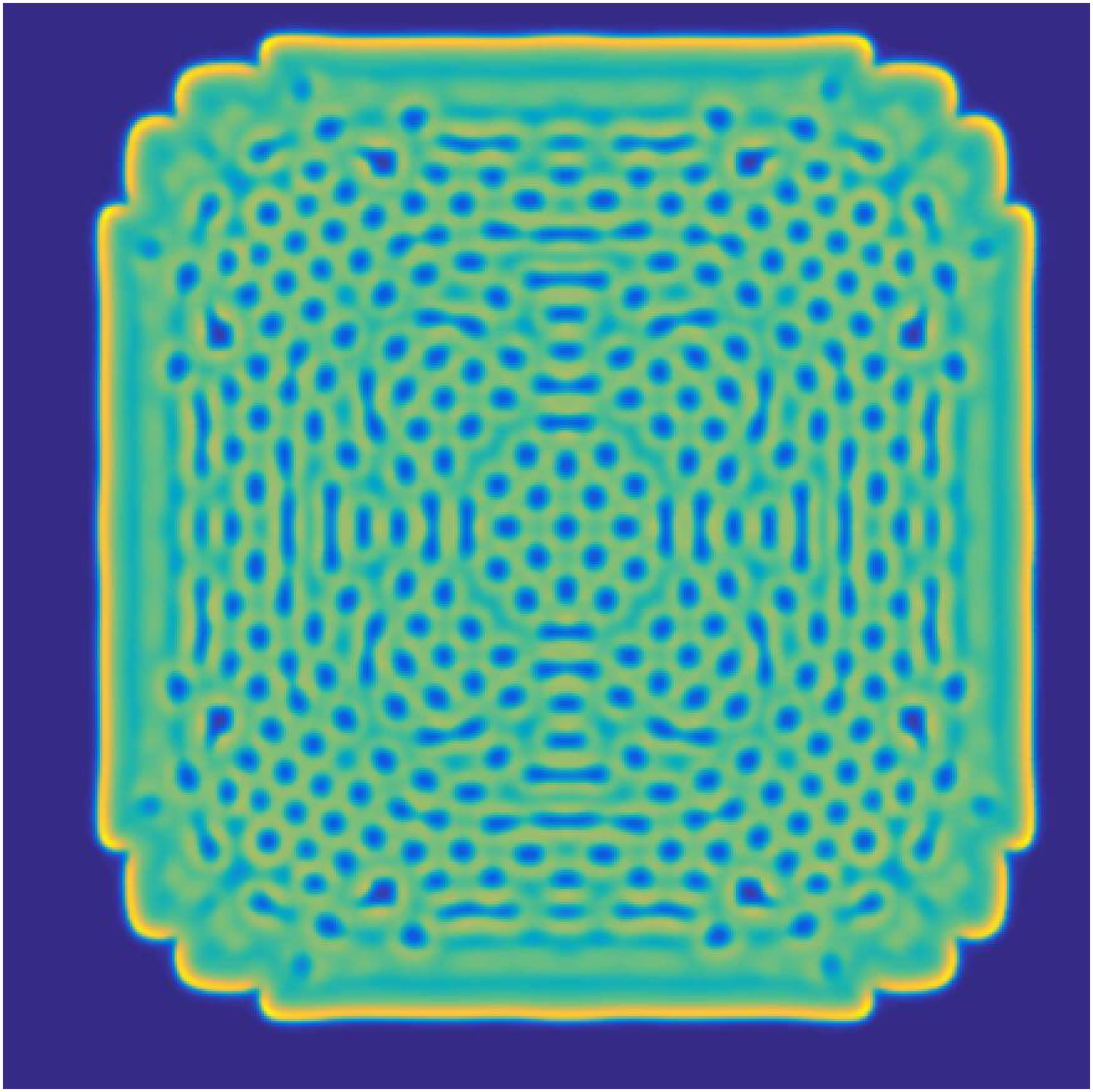}
		\end{minipage}}
		\subfigure[t=30000]{
			\begin{minipage}[t]{0.18\textwidth}
				\centering
				\includegraphics[width=1in]{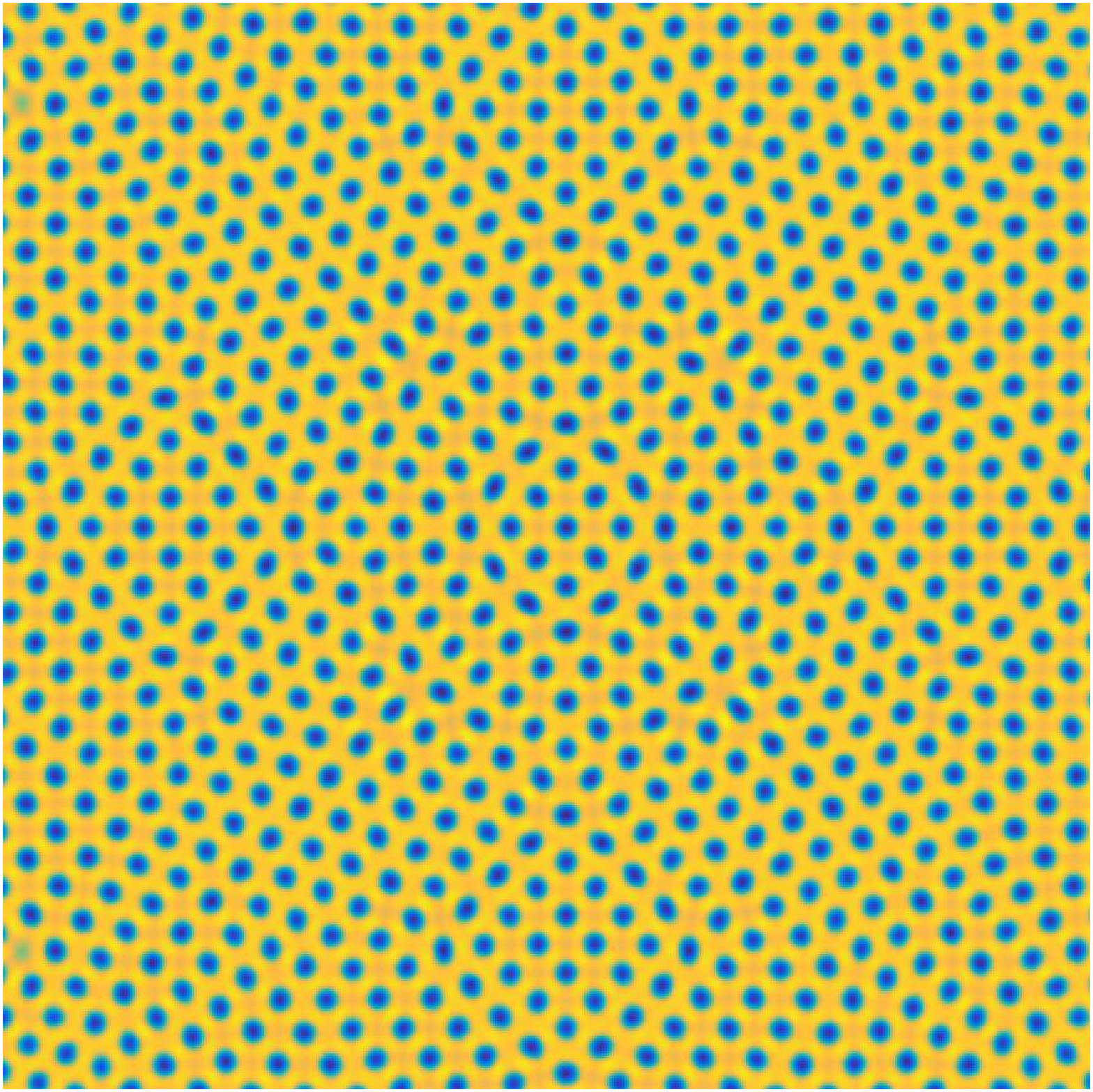}
		\end{minipage}}
		\caption{Example \ref{ex4.2}: Evolution of the solution $v$ with $\lambda=0.055$: (a)-(e)   $\alpha=2$; (f)-(j)  $\alpha=1.7$; (k)-(o) $\alpha=1.5$, the initial conditions are given by equation \eqref{eq4.7}}
		\label{fig3c}
	\end{figure}
	Take $\kappa=2$, $N=1024$, $\tau=0.1$, $K_{u}=2\times10^{-5}$, $K_{v}=K_{u}/2$, and $F=0.03$  and vary $\lambda$ in a range in which the standard diffusion model is known to exhibit interesting dynamics \cite{Pearson1993}. For a ratio of diffusion coefficients $K_{u}/K_{v}>1$, the model is known to generate different mechanisms of pattern formation depending on the values of the feed rate $F$ and decay rate $\lambda$.
	
	Figs. \ref{fig3a}-\ref{fig3c} show the evolutions of the numerical solution $v$ and summarize the effects of the super-diffusion for the fractional Gray--Scott model. The domain of interest is taken to be $(0,1)^{2}$. All the plots shown in Figs. \ref{fig3a}-\ref{fig3c}   are snapshots from the numerical solutions $v$ in the above domain. The speeds of pattern formation are different for different $\alpha$ due to the influence of fractional powers.
	
	For $\lambda=0.063$ (Fig. \ref{fig3a}), the Gray--Scott model exhibits patterns of mitosis under conditions of the standard diffusion ($\alpha=2$). When $\alpha=1.7$, the replication pattern has completely changed as the fractional order of the model is decreased. As is shown for $\alpha=1.5$, the patterns present different behavior and exert dynamical states where solitons and filaments may coexist. We observe that the structure of patterns and the size of the spots are different.
	In Fig. \ref{fig3b} with $\lambda=0.061$, the original model produces a wavefront propagation partially driven by curvature. When the order $\alpha$ of  the fractional Laplacian operator decreases, we can find that the area generated by diffusion propagation becomes larger. The final field is made up of much thinner filaments and transforms from point shape to line shape. For $\lambda=0.055$ (Fig. \ref{fig3c}), the model with standard diffusion $(\alpha = 2)$ is well known to organize in a steady state field of negative solitons. Moreover, the reduction of the fractional order $(\alpha = 1.7)$ generates a decrease in the velocity of propagation of the initial perturbation, and affects the size of patterns with smaller spots. For smaller values of the fractional power $(\alpha = 1.5)$, we observe a new process of nucleation of structures in the centre and in the boundaries of the domain. The process propagates outward until the entire area reaches the final steady state.
	
	\begin{figure}[htbp]
		\centering
		\subfigure[t=200]{
			\begin{minipage}[t]{0.18\textwidth}
				\centering
				\includegraphics[width=3.5cm,height=2.5cm]{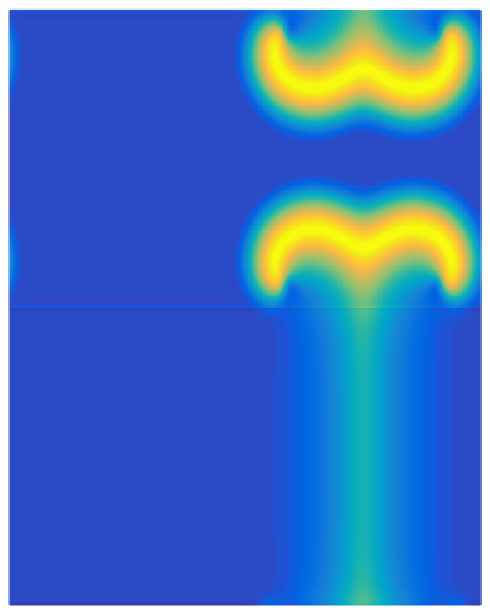}
		\end{minipage}}
		\subfigure[t=400]{
			\begin{minipage}[t]{0.18\textwidth}
				\centering
				\includegraphics[width=3.5cm,height=2.5cm]{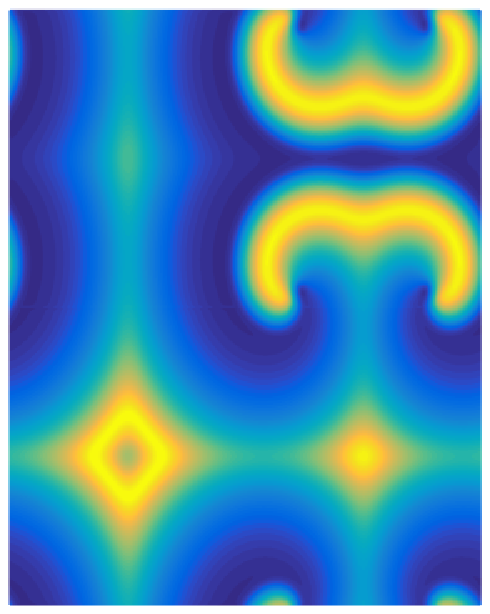}
		\end{minipage}}
		\subfigure[t=1000]{
			\begin{minipage}[t]{0.18\textwidth}
				\centering
				\includegraphics[width=3.5cm,height=2.5cm]{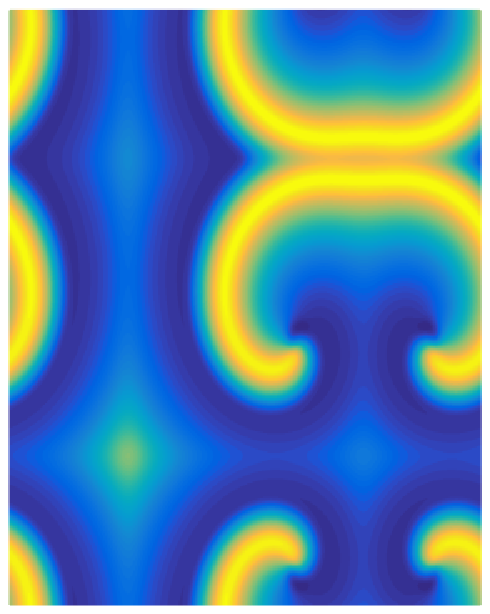}
		\end{minipage}}
		\subfigure[t=1500]{
			\begin{minipage}[t]{0.18\textwidth}
				\centering
				\includegraphics[width=3.5cm,height=2.5cm]{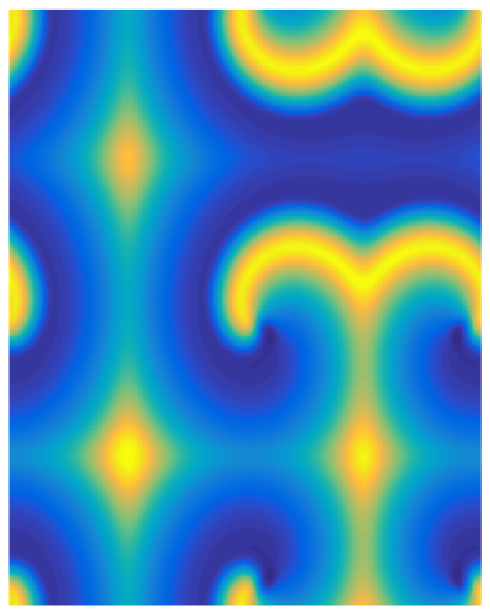}
		\end{minipage}}
		\subfigure[t=2000]{
			\begin{minipage}[t]{0.18\textwidth}
				\centering
				\includegraphics[width=3.5cm,height=2.5cm]{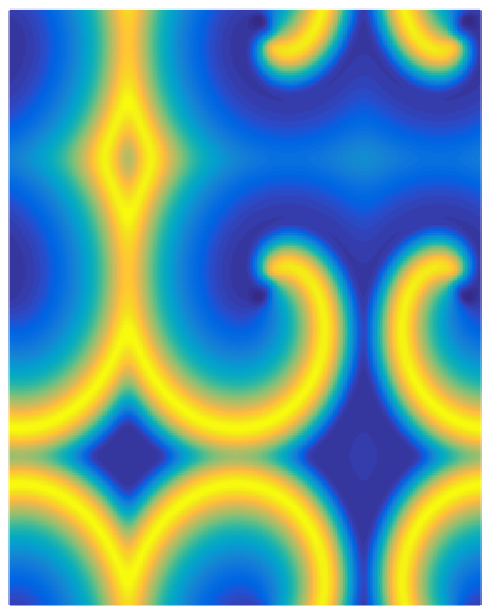}
		\end{minipage}}
		\subfigure[t=200]{
			\begin{minipage}[t]{0.18\textwidth}
				\centering
				\includegraphics[width=3.5cm,height=2.5cm]{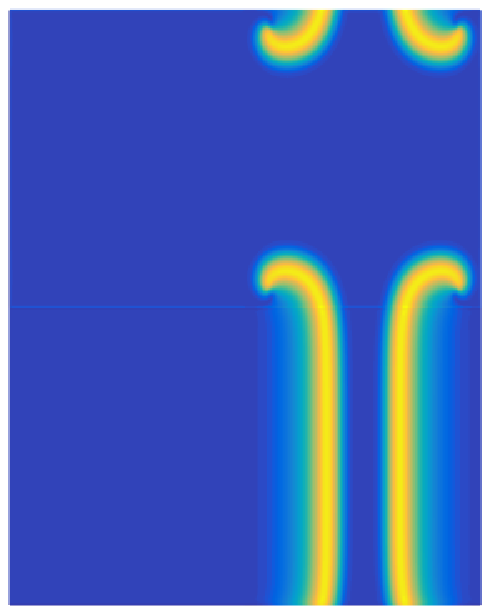}
		\end{minipage}}
		\subfigure[t=400]{
			\begin{minipage}[t]{0.18\textwidth}
				\centering
				\includegraphics[width=3.5cm,height=2.5cm]{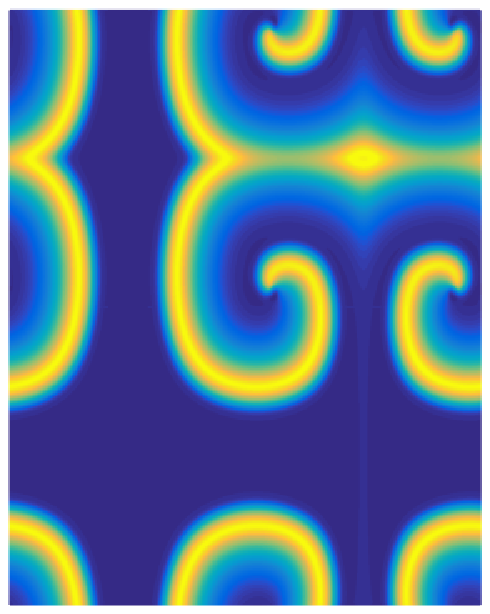}
		\end{minipage}}
		\subfigure[t=1000]{
			\begin{minipage}[t]{0.18\textwidth}
				\centering
				\includegraphics[width=3.5cm,height=2.5cm]{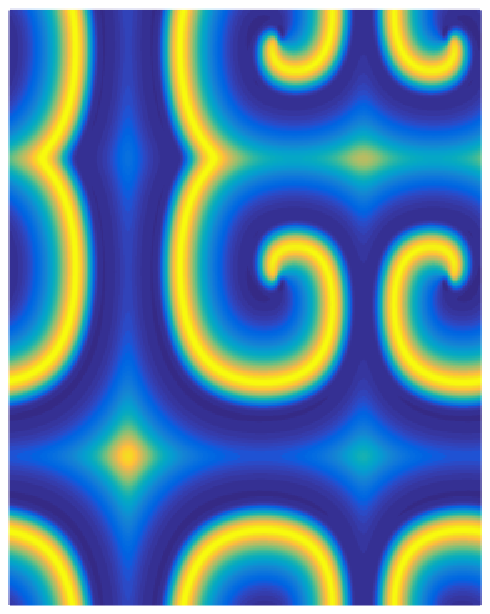}
		\end{minipage}}
		\subfigure[t=1500]{
			\begin{minipage}[t]{0.18\textwidth}
				\centering
				\includegraphics[width=3.5cm,height=2.5cm]{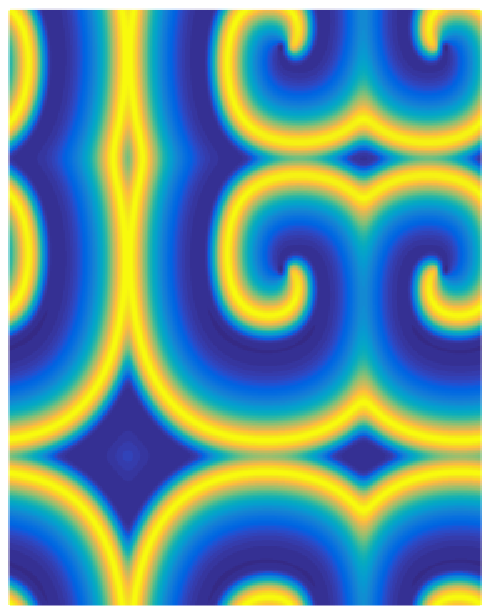}
		\end{minipage}}
		\subfigure[t=2000]{
			\begin{minipage}[t]{0.18\textwidth}
				\centering
				\includegraphics[width=3.5cm,height=2.5cm]{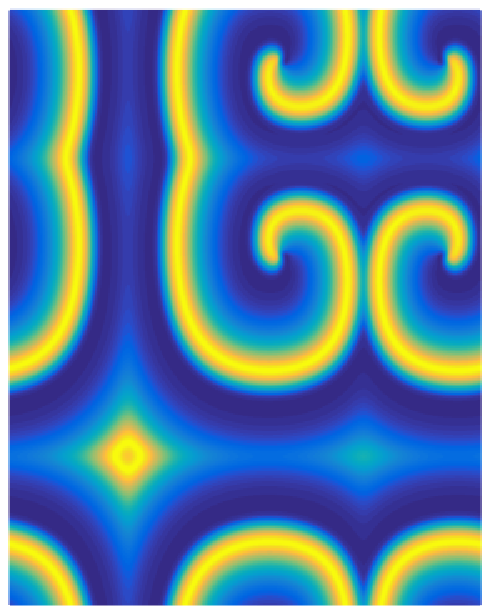}
		\end{minipage}}
		\subfigure[t=200]{
			\begin{minipage}[t]{0.18\textwidth}
				\centering
				\includegraphics[width=3.5cm,height=2.5cm]{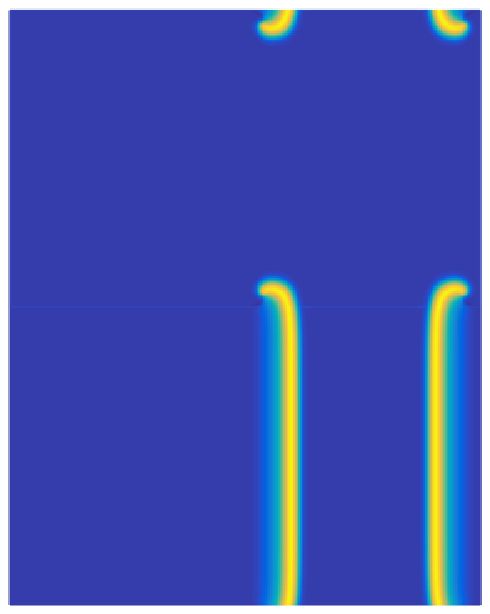}
		\end{minipage}}
		\subfigure[t=400]{
			\begin{minipage}[t]{0.18\textwidth}
				\centering
				\includegraphics[width=3.5cm,height=2.5cm]{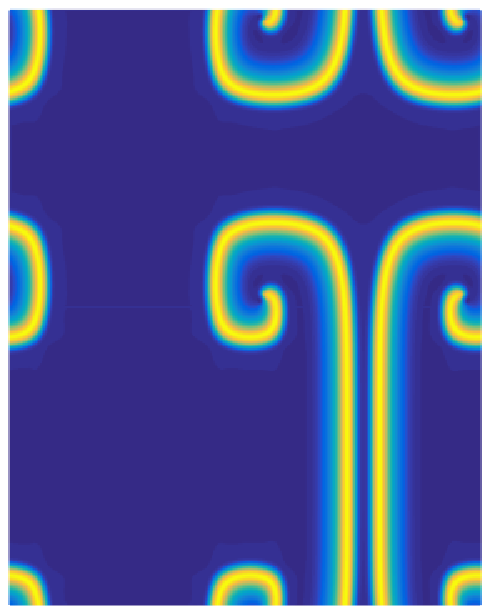}
		\end{minipage}}
		\subfigure[t=1000]{
			\begin{minipage}[t]{0.18\textwidth}
				\centering
				\includegraphics[width=3.5cm,height=2.5cm]{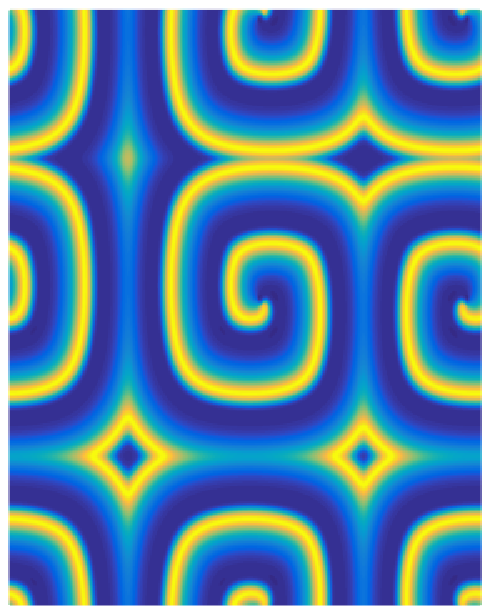}
		\end{minipage}}
		\subfigure[t=1500]{
			\begin{minipage}[t]{0.18\textwidth}
				\centering
				\includegraphics[width=3.5cm,height=2.5cm]{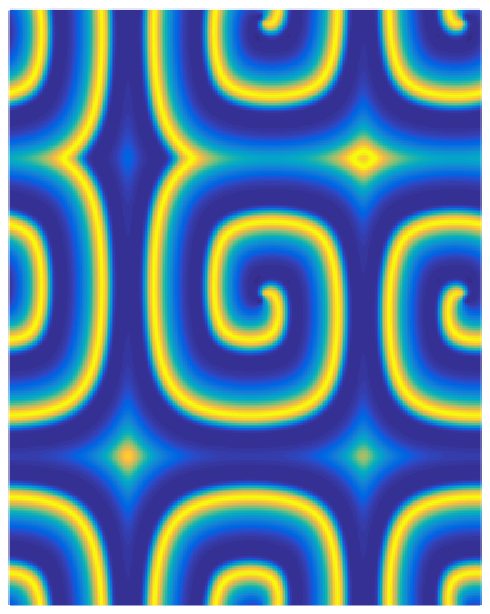}
		\end{minipage}}
		\subfigure[t=2000]{
			\begin{minipage}[t]{0.18\textwidth}
				\centering
				\includegraphics[width=3.5cm,height=2.5cm]{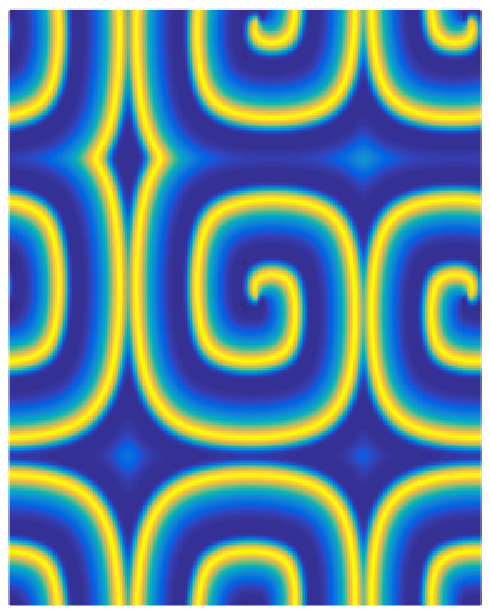}
		\end{minipage}}
		\caption{Example \ref{ex4.3}: Evolution of the solution $u$ with $K_{u}=1e-04$: (a)-(e)  $\alpha=2$; (f)-(j)  $\alpha=1.7$;	(k)-(o)  $\alpha=1.5$, the initial conditions are given by equations \eqref{eq4.5}-\eqref{eq4.6}.}
		\label{fig4a}
	\end{figure}
	\begin{figure}[htbp]
		\centering
		\subfigure[t=200]{
			\begin{minipage}[t]{0.18\textwidth}
				\centering
				\includegraphics[width=3.5cm,height=2.5cm]{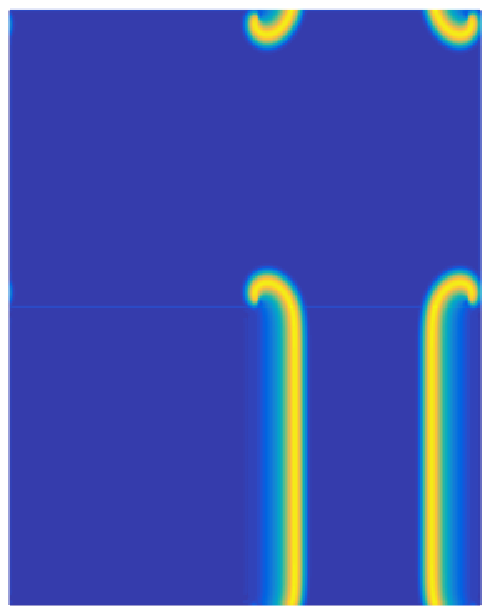}
		\end{minipage}}
		\subfigure[t=400]{
			\begin{minipage}[t]{0.18\textwidth}
				\centering
				\includegraphics[width=3.5cm,height=2.5cm]{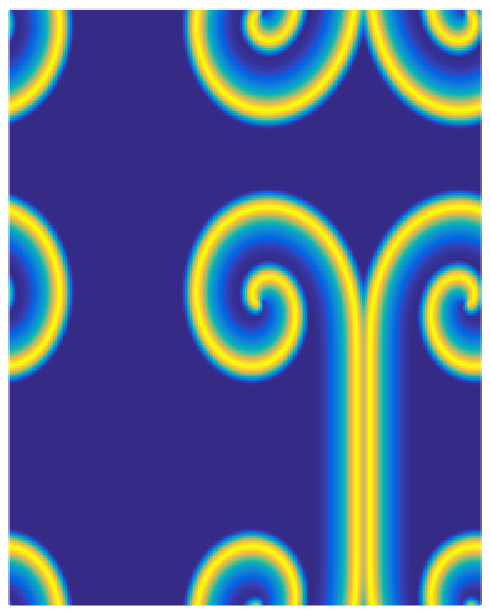}
		\end{minipage}}
		\subfigure[t=1000]{
			\begin{minipage}[t]{0.18\textwidth}
				\centering
				\includegraphics[width=3.5cm,height=2.5cm]{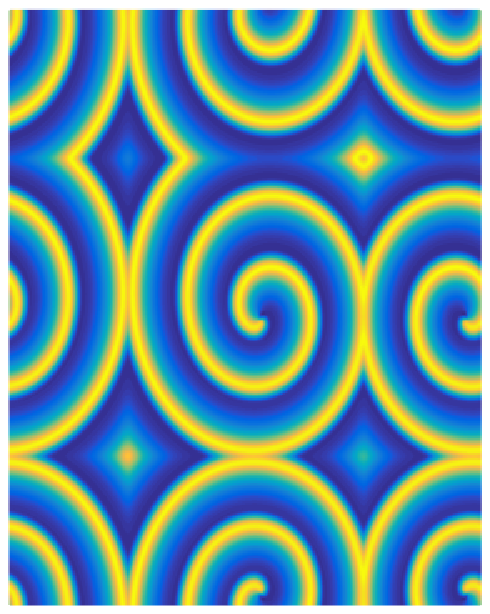}
		\end{minipage}}
		\subfigure[t=1500]{
			\begin{minipage}[t]{0.18\textwidth}
				\centering
				\includegraphics[width=3.5cm,height=2.5cm]{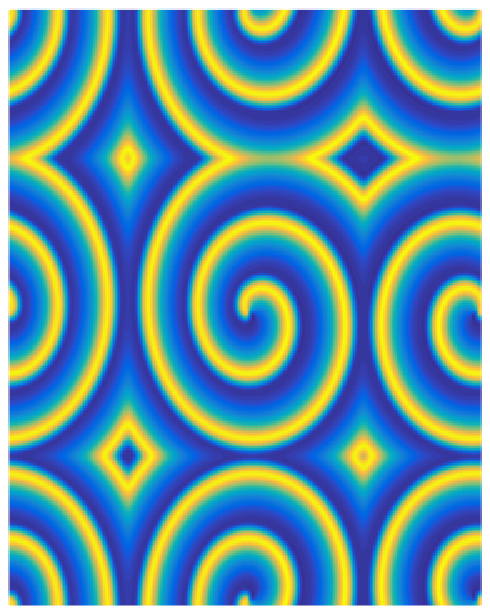}
		\end{minipage}}
		\subfigure[t=2000]{
			\begin{minipage}[t]{0.18\textwidth}
				\centering
				\includegraphics[width=3.5cm,height=2.5cm]{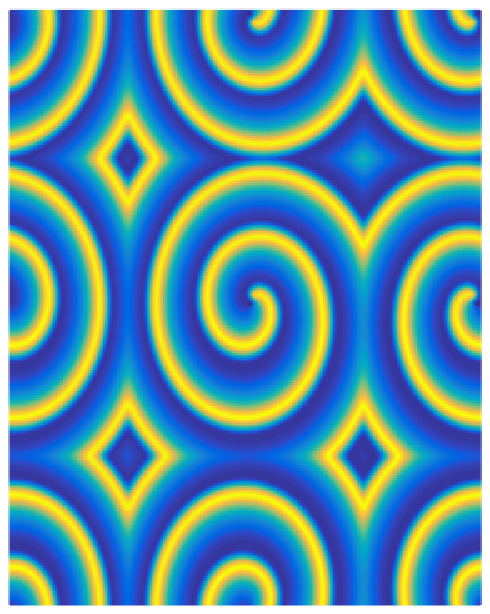}
		\end{minipage}}
		\subfigure[t=200]{
			\begin{minipage}[t]{0.18\textwidth}
				\centering
				\includegraphics[width=3.5cm,height=2.5cm]{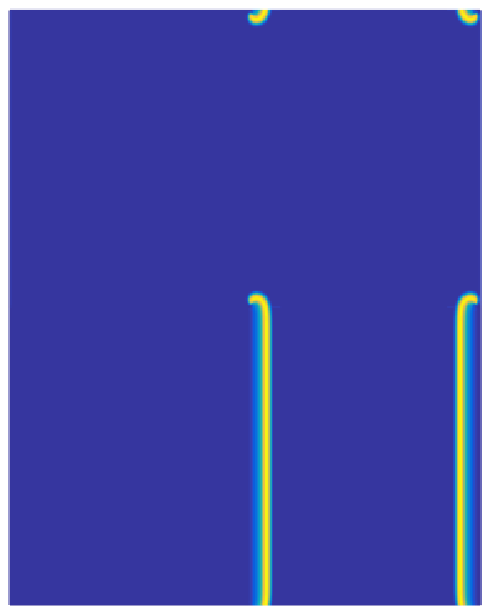}
		\end{minipage}}
		\subfigure[t=400]{
			\begin{minipage}[t]{0.18\textwidth}
				\centering
				\includegraphics[width=3.5cm,height=2.5cm]{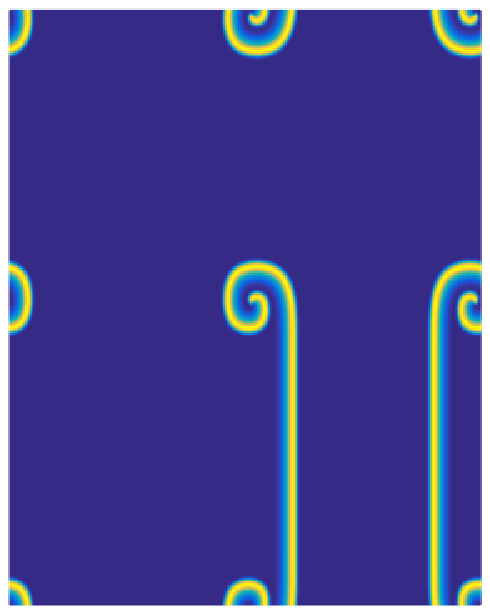}
		\end{minipage}}
		\subfigure[t=1000]{
			\begin{minipage}[t]{0.18\textwidth}
				\centering
				\includegraphics[width=3.5cm,height=2.5cm]{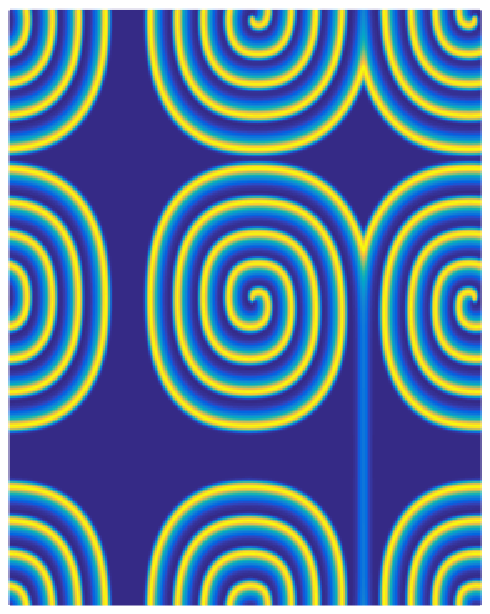}
		\end{minipage}}
		\subfigure[t=1500]{
			\begin{minipage}[t]{0.18\textwidth}
				\centering
				\includegraphics[width=3.5cm,height=2.5cm]{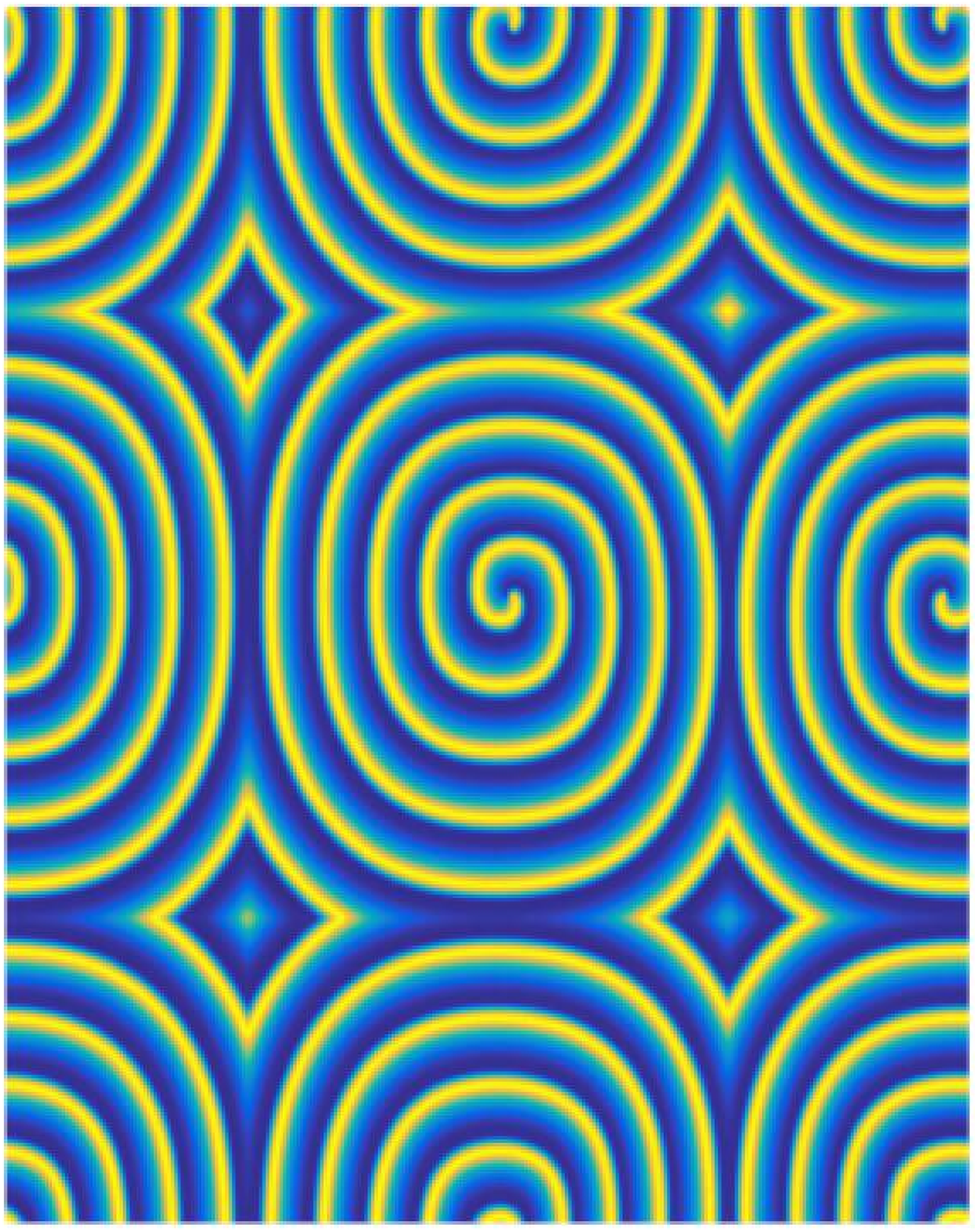}
		\end{minipage}}
		\subfigure[t=2000]{
			\begin{minipage}[t]{0.18\textwidth}
				\centering
				\includegraphics[width=3.5cm,height=2.5cm]{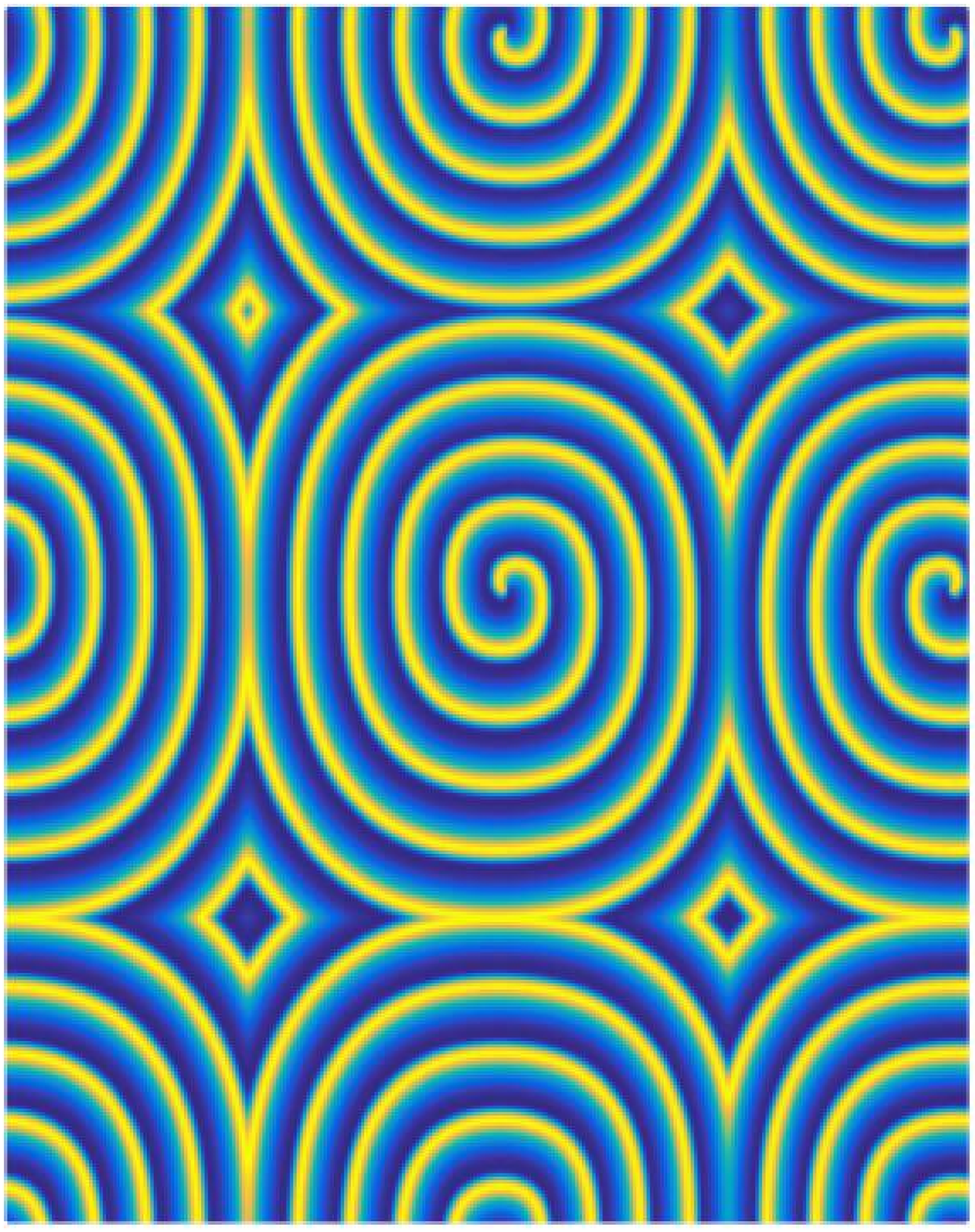}
		\end{minipage}}
		\subfigure[t=200]{
			\begin{minipage}[t]{0.18\textwidth}
				\centering
				\includegraphics[width=3.5cm,height=2.5cm]{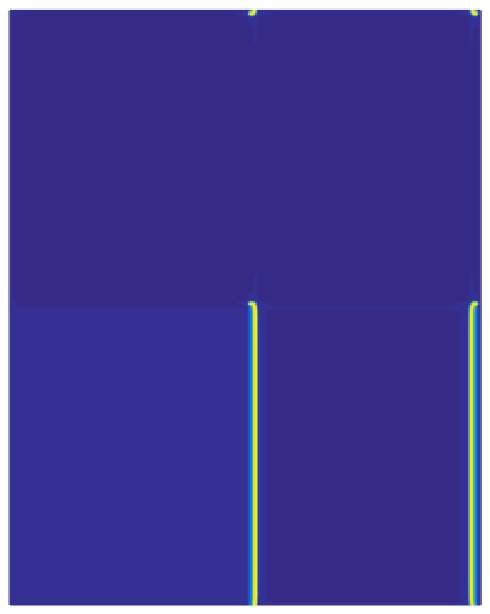}
		\end{minipage}}
		\subfigure[t=400]{
			\begin{minipage}[t]{0.18\textwidth}
				\centering
				\includegraphics[width=3.5cm,height=2.5cm]{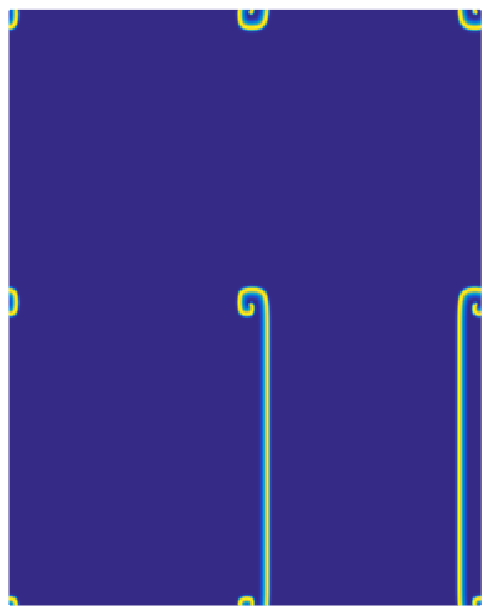}
		\end{minipage}}
		\subfigure[t=1000]{
			\begin{minipage}[t]{0.18\textwidth}
				\centering
				\includegraphics[width=3.5cm,height=2.5cm]{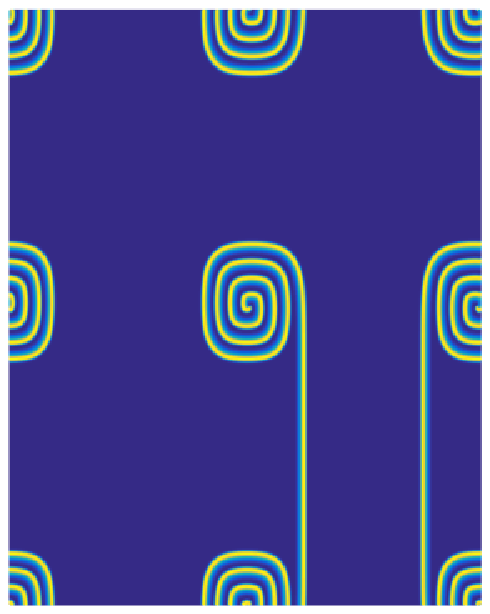}
		\end{minipage}}
		\subfigure[t=1500]{
			\begin{minipage}[t]{0.18\textwidth}
				\centering
				\includegraphics[width=3.5cm,height=2.5cm]{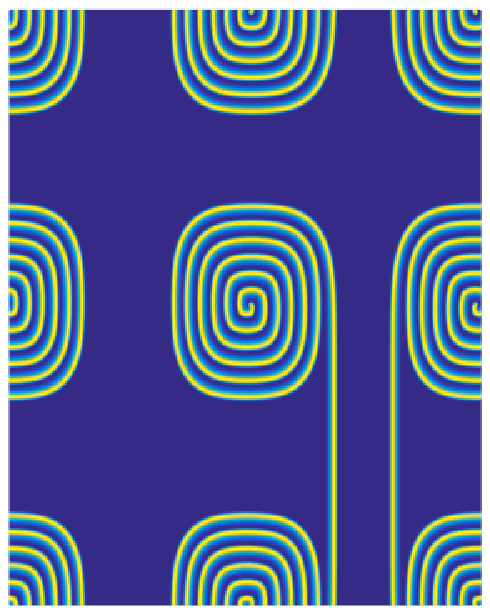}
		\end{minipage}}
		\subfigure[t=2000]{
			\begin{minipage}[t]{0.18\textwidth}
				\centering
				\includegraphics[width=3.5cm,height=2.5cm]{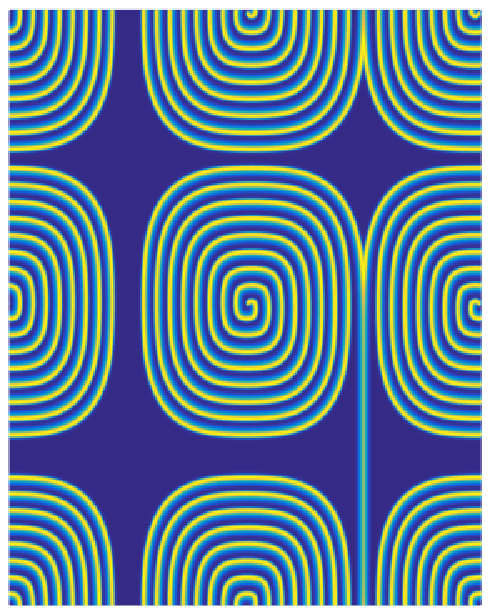}
		\end{minipage}}
		\caption{Example \ref{ex4.3}: Evolution of the solution $u$ with $K_{u}=1e-05$: (a)-(e)   $\alpha=2$; (f)-(j)   $\alpha=1.7$; (k)-(o)  $\alpha=1.5$, the initial conditions are shown as equations \eqref{eq4.5}-\eqref{eq4.6}.}
		\label{fig4b}
	\end{figure}

	\begin{example}[Fractional FitzHugh--Nagumo model]\label{ex4.3}	
		Here $\Omega=(0,2.5)\times (0,2.5)$, $u$ is the excitation variable which can be identified with the transmembrane potential, $v$ is the recovery variable, and $K_{u}$ is a small positive diffusion coefficient. In order to illustrate the effect of varying $\alpha$ in the model. The initial conditions are taken as
		\begin{equation}\begin{aligned}
		\label{eq4.5}
		u(x_1,x_2,0)=\left\{ \begin{array}{ll}
		1,\quad (x_1,x_2)\in (0,0.125]\times (0,0.125),\\
		0,\quad \mathrm{elsewhere},
		\end{array}\right.
		\end{aligned}\end{equation}
		\begin{equation}\begin{aligned}
		\label{eq4.6}
		v(x_1,x_2,0)=\left\{ \begin{array}{ll}
		0.1,\quad (x_1,x_2)\in (0,2.5)\times [0.125,2.5),\\
		0,\qquad  (x_1,x_2)\in(0,2.5)\times(0,0.125).
		\end{array}\right.
		\end{aligned}\end{equation}
	\end{example}
	We choose the parameters $\kappa=2$, $\mu=0.1$, $\varepsilon=0.01$, $\beta=0.5$, $\gamma=1$, $\delta=0$, which are known to generate stable patterns in the system in the form of spiral waves. Stable rotating solutions at different time are presented in Fig. \ref{fig4a} to illustrate the effect of fractional order in the FitzHugh--Nagumo model with $K_{u}=1\times10^{-4}$. We can observe that the width of the excitation wavefront is markedly reduced for decreasing $\alpha$, so is the wavelength of the system, with the domain being able to accommodate a larger number of wavefronts for smaller $\alpha$. By the comparison between Fig. \ref{fig4a} and Fig. \ref{fig4b} $(K_{u}=1\times10^{-5})$, it is obvious that the role of reducing the fractional power $\alpha$ is not equivalent to the influence of a decreased diffusion coefficient $K_{u}$ in the pure diffusion case. For approximately the same width of the excitation wavefront, the wavelength of the model is larger in the fractional diffusion case, due to the long-tailed mechanisms of the fractional Laplacian operator. These results are consistent with those of Engler \cite{Engler2010} for the fractional reaction-diffusion Fisher equation, showing distinct effects of fractional diffusion to those of reduced conductivity for a family of traveling wave solutions. They also illustrate the use of fractional diffusion as a modeling tool to characterize intermediate dynamic states not solely described by pure diffusion mechanisms.
	
	\section{Conclusions}\label{sec6}
	Fractional differential equations have been proved to be useful tools for modeling diffusive processes associated with anomalous diffusion or spatial heterogeneity. In this paper, we study a second-order stabilized semi-implicit time-stepping Fourier spectral method for the reaction-diffusion systems of equations with space described by the fractional Laplacian in bounded rectangular domains. An optimal error estimate of the numerical scheme is obtained without imposing the CFL condition by adopting the temporal-spatial error splitting argument. We also analyze the linear stability of the stabilized semi-implicit method and obtain a practical criteria to choose the time step size to guarantee the stability of the semi-implicit method in real applications. Our approach is illustrated by solving several problems of practical interest, including the fractional Allen-Cahn, FitzHugh-Nagumo and Gray-Scott models show that such systems can have different dynamics to pure diffusion, and drastically different evolving patterns compared to integer-order models.
	\section*{Acknowledgments}
	This work was supported by the National Natural Science Foundation of China (Grants Nos. 11102102, 11672163), the Natural Science Foundation of Shandong Province (Grant ZR2017MA030) and the MURI/ARO on ``Fractional PDEs for Conservation Laws and Beyond: Theory, Numerics and Applications (W911NF-15-1-0562)".
	
	\section*{Appendix A. Proof of Theorem \ref{th2.1}}
	In the following, we denote $C,C_q,\widetilde{C}_q$,
	and $\widehat{C}_q$ as positive constants that are independent of the
	time step size $\tau$ and the degree of approximation space $N$,
	where $q=1,2,\cdots$.
	
	\renewcommand\thelemma{A.1}
	We first introduce some Lemmas that will be used in the numerical analysis later.
	\begin{lemma}[\cite{Ainsworth2017}]\label{le3.3}
		Let $s,r\geq0$, then for any $u,v \in H_{per}^{s+r}(\Omega)$, it holds that
		\begin{displaymath}\label{eq2.11}
		((-\Delta)^{s+r}u,v)=((-\Delta)^{s}u,(-\Delta)^{r}v).
		\end{displaymath}
	\end{lemma}
	\renewcommand\thelemma{A.2}
	\begin{lemma}[\cite{Shen2011,Canuto2006}]\label{le3.1}
		Suppose that $u\in H_{per}^{r}(\Omega)$. Then the following estimate holds for all $0\leq s\leq r$,
		\begin{displaymath}
		\|u-P_{N}u\|_{s}\leq C_1N^{s-r}\|u\|_{r}.
		\end{displaymath}
	\end{lemma}
	\renewcommand\thelemma{A.3}
	\begin{lemma}[\cite{Adams2003}]\label{le3.6}
		Let $\Omega$ be a domain in $\mathbb{R}^{n}$ satisfying the cone condition. If $mp>n$, let $p\leq q \leq \infty$; if $mp=n$, let $p\leq q < \infty$; if $mp<n$, let $p\leq q \leq q^{*}=np/(n-mp)$. Then there exists a constant $C_2$ depending on $m$, $n$, $p$, $q$, and the dimensions of $\Omega$ such that for all $u\in W^{m,p}(\Omega)$,
		\begin{displaymath}\label{eq2.133}
		\|u\|_{q}\leq C_2 \|u\|_{m,p}^{\varepsilon}\|u\|_{p}^{1-\varepsilon},
		\end{displaymath}
		where $\varepsilon=(n/mp)-(n/mq)$.
	\end{lemma}
	
	\renewcommand\thelemma{A.4}
	\begin{lemma}[\cite{Ainsworth2017}]\label{le3.4}
		If $s\leq r$ and $u\in H_{per}^{r}(\Omega)$, then
		\begin{displaymath}\label{eq2.12}
		\|(-\Delta)^{s}u\|\leq\|(-\Delta)^{r}u\|.
		\end{displaymath}
	\end{lemma}
	\renewcommand\thelemma{A.5}
	\begin{lemma}[\cite{Ainsworth2017}]\label{le3.5}
		If $u,v\in H_{per}^{2s}(\Omega)$ for all $s\geq0$, then
		\begin{displaymath}\label{eq2.13}
		\|(-\Delta)^{s}(uv)\|^{2}\leq C_{s}\left[\|v\|_{\infty}^{2}\|(-\Delta)^{s}u\|^{2}+ \|(-\Delta)^{s}v\|_{\infty}^{2}\|u\|^{2}\right],
		\end{displaymath}where $C_{s}=\max{\{1,2^{2s-1}\}}$.
	\end{lemma}
	
	\renewcommand\thelemma{A.6}
	\begin{lemma}[\cite{Zeng2014}]\label{inverse}
		For any $u\in X_N$, the following inverse inequality holds,
		\begin{equation*}
		\|u\|_{L^{\infty}}\leq C_3N\|u\|.
		\end{equation*}
	\end{lemma}
	
	We assume that the  solution $u$ of \eqref{e1.1}  satisfies
	\begin{equation}\label{e2.5}	\|u_{0}\|_{H^{r}_{per}}+\|u\|_{C(I;H^{r}_{per})}+\|u_{t}\|_{C(I;H^{r}_{per})}
	+\|u_{tt}\|_{C(I;H^{\alpha}_{per})}\leq M,
	\end{equation}
	where $I=(0,T]$, $M$ is a constant independent of $n$, $N$, and $\tau$. The above regularity assumption is used in the convergence analysis.
	
	We introduce a time-discrete system,
	\begin{equation}\label{e3.1}\left\{\begin{aligned}	&D_{2}U^{n}+K_{\alpha}(-\Delta)^{\alpha/2}U^n=2G(U^{n-1})-G(U^{n-2})-\tau\kappa(D_1U^n-D_1U^{n-1},v),\\
	&U^0=u^0,{\quad} U^1=u^1.
	\end{aligned}\right.\end{equation}
	Thus, we split the error $u^n-U_N^n$ into two parts, that is,
	$u^{n}-U_{N}^{n} = (u^{n}-U^{n})+(U^{n}-U_{N}^{n})$.
	The boundedness of $u^n-U_N^n$ can be obtained by estimating $u^{n}-U^{n}$
	and $U^{n}-U_{N}^{n}$ separately.
	
	Denote by $e^{n}=u^{n}-U^{n}$. Combining \eqref{e3.4} and \eqref{e3.1} yields
	\begin{equation}\label{e3.10}
	D_{2}e^{n}+K_{\alpha}(-\Delta)^{\alpha/2}e^n=R^{n}+F_{1}^{n-1}
	-\tau\kappa( D_1e^n-D_1e^{n-1}),{\quad}2\leq n \leq K,
	\end{equation}
	where
	\begin{equation}\label{e3.11}
	F_{1}^{n-1}=2G(u^{n-1})-G(u^{n-2})-(2G(U^{n-1})-G(U^{n-2}).
	\end{equation}
	
	Define
	\begin{equation}
	\label{e3.15}
	M_{1}=\underset{1\leq n\leq K}{\max}\|u^{n}\|_{\infty}+4.
	\end{equation}According to (\ref{e2.5}) and (\ref{e3.15}), we can find that $M_{1}$ is a positive constant independent of $n$, $\tau$, and $N$.
	
	We have the following theorem for the time-discrete system \eqref{e3.1}.
	\renewcommand\thetheorem{A.1}
	\begin{theorem} \label{th3.1}
		Suppose that \eqref{e1.1} has a unique solution $u(\cdot,t)\in H^{\alpha}_{per}(\Omega)$
		satisfying \eqref{e2.5}. Then, the time-discrete system \eqref{e3.1} has a unique solution $U^{n}$.
		If $U^n\in H^{\alpha}_{per}(\Omega)$, then
		there exists a positive constant $\tau_{1}^{*}>0$,
		when $\tau\leq \tau_{1}^{*}$, it holds,
		\begin{equation}
		\label{e3.24}
		\|e^{n}\|_{\alpha}\leq \tau,
		\end{equation}
		\begin{equation}
		\|U^{n}\|_{\infty}\leq 2M_{1}.\label{e3.25}
		\end{equation}
	\end{theorem}
	\begin{proof}
		When $n=1$, $e^{1}=0$. The proof is obvious.
		
		Assume that (\ref{e3.24}) holds for $2\leq n\leq k-1$. By Lemma \ref{le3.6}, we have
		\begin{equation}\begin{aligned}
		\label{e3.26}
		\|U^{n}\|_{\infty}&\leq \|u^{n}\|_{\infty}+\|e^{n}\|_{\infty}
		\leq \|u^{n}\|_{\infty}+C_2\|e^{n}\|_{\alpha}\\
		&\leq \|u^{n}\|_{\infty}+C_2\tau
		\leq \|u^{n}\|_{\infty}+1\\
		&\leq M_{1},\\
		\end{aligned}\end{equation}	where $\tau\leq C_2^{-1}$.
		
		Next, we prove (\ref{e3.24}) for $n\leq k$. Multiplying (\ref{e3.10}) by
		$2\tau(-\Delta)^{\alpha/2}D_{1}e^{n}$ and integrating the resulting system
		over $\Omega$, we obtain
		\begin{equation}\label{e3.29}\begin{aligned}
		&\quad	\frac{\tau}{2}|D_{1}e^{n}|_{\alpha/2}^{2}+{K_{\alpha}}|e^{n}|^{2}_{\alpha}+\kappa\tau^2|D_1e^{n}|^{2}_{\alpha/2}\\
		&\leq\frac{\tau}{2}|D_{1}e^{n-1}|_{\alpha/2}^{2}+{K_{\alpha}}|e^{n-1}|^{2}_{\alpha}+\kappa\tau^2|D_1e^{n-1}|^{2}_{\alpha/2}
		+\tau|R^{n}|_{\alpha/2}^{2}+\tau\|(-\Delta)^{\alpha/4}F_{1}^{n-1}\|^{2},
		\end{aligned}\end{equation}
		where Lemma \ref{le3.3} and the Cauchy--Schwartz inequality is used. Summing $n$ in (\ref{e3.29}) from $2$ to $\hat{k}\,(2\leq\hat{k}\leq k)$ and using
		$e^0=e^1=0$ gives
		\begin{equation}\label{e3.29-2}\begin{aligned}
		\frac{\tau}{2}|D_{1}e^{\hat{k}}|_{\alpha/2}^{2}+{K_{\alpha}}|e^{\hat{k}}|^{2}_{\alpha}+\kappa\tau^2|D_1e^{\hat{k}}|^{2}_{\alpha/2}
		\leq&\tau\sum_{n=2}^{\hat{k}}
		\left(|R^{n}|_{\alpha/2}^{2}+\|(-\Delta)^{\alpha/4}F_{1}^{n-1}\|^{2}\right).	
		\end{aligned}\end{equation}
		
		Next, we estimate $\|(-\Delta)^{\alpha/4}F_{1}^{n-1}\|^{2}$.
		From Lemma \ref{le3.5}, \eqref{e3.26},  and the mean value theorem, we derive
		\begin{equation}\begin{aligned}\label{e3.30}
		&\|(-\Delta)^{\alpha/4}F_{1}^{n-1}\|^{2}=\|(-\Delta)^{\alpha/4}
		\left(2G'(\xi^{n-1})e^{n-1}-G'(\xi^{n-2})e^{n-2}\right)\|^{2}\\
		\leq& C_{\alpha}\left[\|G'(\xi^{n-1})\|_{\infty}^{2}\|(-\Delta)^{\alpha/4}e^{n-1}\|^{2}
		+\|e^{n-1}\|^{2}\|(-\Delta)^{\alpha/4}G'(\xi^{n-1})\|_{\infty}^{2}\right]\\
		&+C_{\alpha}\left[\|G'(\xi^{n-2})\|_{\infty}^{2}\|(-\Delta)^{\alpha/4}e^{n-2}\|^{2}
		+\|e^{n-2}\|^{2}\|(-\Delta)^{\alpha/4}G'(\xi^{n-2})\|_{\infty}^{2}\right]\\
		\leq& C_{4}(|e^{n-1}|_{\alpha}^{2}+|e^{n-2}|_{\alpha}^{2}),\\
		\end{aligned}\end{equation}
		where $\xi^{l}\in \left[\min\{u^{l},U^{l}\},\max\{u^{l},U^{l}\}\right]$,
		$|\xi^{l}|\leq \max{\{M,M_{1}\}}$, $l=n-2,n-1$.
		
		Combining \eqref{e3.29-2} and \eqref{e3.30} leads to
		\begin{equation*}\label{e3.31}	\begin{aligned}
		|e^{\hat{k}}|^{2}_{\alpha}
		&\leq \frac{1}{{K_{\alpha}}}
		\left(\frac{\tau}{2}|D_{1}e^{\hat{k}}|_{\alpha/2}^{2}+{K_{\alpha}}|e^{\hat{k}}|^{2}_{\alpha}+\kappa\tau^2|D_1e^{\hat{k}}|^{2}_{\alpha/2}\right)\\
		&\leq C_5(\tau^{4}+\tau\sum_{n=0}^{\hat{k}-1}|e^{n}|_{\alpha}^{2}),{\quad}2 \leq\hat{k}\leq k.
		\end{aligned}\end{equation*}
		By the Gronwall inequality, one has $|e^{\hat{k}}|_{\alpha}\leq C_6\tau^{2}$,
		which implies
		\begin{equation*}\label{e3.33}
		\|e^{\hat{k}}\|_{\alpha}\leq \tau,{\quad} 2 \leq\hat{k}\leq k,
		\end{equation*}
		when $\tau\leq (2C_{6})^{-1}$.
		
		If $\tau\leq C_2^{-1}$, then
		\begin{eqnarray*}
			\|U^{k}\|_{\infty}\leq \|u^{k}\|_{\infty}+C_2\|e^{k}\|_{\alpha}
			\leq \|u^{k}\|_{\infty}+C_2\tau\leq M_{1}.\label{e3.35}
		\end{eqnarray*}
		Taking $\tau_{1}^{*}=\min\{C_2^{-1},(2C_{6})^{-1}\}$, we complete the proof.	
	\end{proof}
	
	From Lemma \ref{le3.1} and Lemma \ref{le3.6} ($m=\alpha,p=n=2,q=\infty$), we have $\|P_{N}v\|_{\infty}\leq \widetilde{C}_1\|v\|_{\alpha}$ for any $v\in H^{\alpha}_{per}(\Omega)$. Therefore, we can obtain the boundedness of $\|P_{N}U^{n}\|_{\infty}$ for $n=1,2,\ldots,K$.  Define the following constant
	\begin{equation}
	\label{e3.36}
	M_{2}=\underset{1\leq n\leq K}{\max}\|P_{N}U^{n}\|_{\infty}+1.
	\end{equation}
	
	Let $\eta_{N}^{n}=P_{N}U^{n}-U^{n}_{N}$.
	From (\ref{e2.3}) and (\ref{e3.1}), we obtain
	\begin{equation}
	\label{e3.38}
	\mathcal{A}^n(\eta_N,v)=(F_{2}^{n-1},v)-\tau\kappa(D_1\eta_N^n-D_1\eta_N^{n-1},v),\ \forall v\in X_N,\ 2\leq n\leq K,
	\end{equation}
	where
	\begin{equation}\label{e3.39}
	F_{2}^{n-1}=2G(U^{n-1})-G(U^{n-2})-P_N(2G(U^{n-1}_{N})-G(U^{n-2}_{N})).
	\end{equation}
	
	The following Theorem gives a bound of $U_{N}^{n}$.
	\renewcommand\thetheorem{A.2}
	\begin{theorem} \label{th3.2}
		Suppose that  \eqref{e3.1} has a unique solution
		$U^n\in H^{\alpha}_{per}(\Omega)$
		and $U^{n}_{N}$ is the solution of  \eqref{e2.3},
		$0\leq n \leq K$. Then there  exists
		two positive constants $\widetilde{C}_{4}$ and $N_{1}^{*}$, when $N\geq N_{1}^{*}$, it holds
		\begin{align}
		&\|\eta_{N}^{n}\|\leq \widetilde{C}_{4}N^{-\alpha},	\label{e3.2222}\\
		&\|U^{n}_{N}\|_{\infty}\leq M_{2}.\label{e3.3333}
		\end{align}
	\end{theorem}
	\begin{proof}
		For $n=0,1$, we have $\eta_{N}^{0}=\eta_{N}^{1}=0$. The proof is obvious.	
		
		Suppose that (\ref{e3.2222}) holds for $2\leq n\leq k-1$. By Lemma \ref{inverse}, we have
		\begin{equation}\begin{aligned}\label{e3.45}
		\|U^{k}_{N}\|_{\infty}&\leq \|P_{N}U^{k}\|_{\infty}+\|\eta^{k}_N\|_{\infty}
		\leq\|P_{N}U^{k}\|_{\infty}+C_{3}N\|\eta^{k}_N\|\\
		&\leq \|P_{N}U^{k}\|_{\infty}+C_{3}\widetilde{C}_{4}N^{1-\alpha}\leq \|P_{N}U^{k}\|_{\infty} + 1\\
		&\leq M_{2},
		\end{aligned}\end{equation}
		where $N\geq (C_{3}\widetilde{C}_{4})^{1/(\alpha-1)}$.
		
		Now, we prove \eqref{e3.2222} for $n\leq k$.
		Letting $v=2\tau D_{1}\eta^{n}_{N}$ in (\ref{e3.38}), it can be seen that
		\begin{equation}\label{eq3.544}\begin{aligned}
		&\quad\frac{5\tau}{2}\|D_{1}\eta_{N}^{n}\|^{2}+{K_{\alpha}}|\eta_{N}^{n}|^{2}_{\alpha/2}+\kappa\tau^2\|D_1\eta_{N}^{n}\|^2\\&\leq\frac{\tau}{2}\|D_{1}\eta^{n-1}_{N}\|^{2}+K_{\alpha}|\eta^{n-1}_{N}|_{\alpha/2}^{2}+\kappa\tau^2\|D_1\eta_{N}^{n-1}\|^2
		+2\tau(F_{2}^{n-1},D_{1}\eta^{n}_{N}).
		\end{aligned}\end{equation}
		Similar to \eqref{e3.30}, we obtain
		\begin{equation}\begin{aligned}
		\label{eqq3.58}
		&\quad2\tau(F_{2}^{n-1},D_{1}\eta^{n}_{N})\\&\leq \tau\|D_{1}\eta^{n}_{N}\|^{2}+\widetilde{C}_{2}\tau(N^{-2\alpha}
		+\|\eta_{N}^{n-1}\|^{2}+\|\eta_{N}^{n-2}\|^{2}).
		\end{aligned}\end{equation}
		According to \eqref{eq3.544} and \eqref{eqq3.58}, one has
		\begin{equation}\begin{aligned}
		\label{eq3.54}
		&\quad\frac{3\tau}{2}\|D_{1}\eta_{N}^{n}\|^{2}+{K_{\alpha}}|\eta_{N}^{n}|^{2}_{\alpha/2}+\kappa\tau^2\|D_1\eta_{N}^{n}\|^2
		\\&\leq\frac{\tau}{2}\|D_{1}\eta_{N}^{n-1}\|^{2}+K_{\alpha}|\eta^{n-1}_{N}|_{\alpha/2}^{2}+\kappa\tau^2\|D_1\eta_{N}^{n-1}\|^2\\&\quad+
		\widetilde{C}_{2}\tau(N^{-2\alpha}
		+\|\eta_{N}^{n-1}\|^{2}+\|\eta_{N}^{n-2}\|^{2}).
		\end{aligned}\end{equation}
		Summing $n$ from $2$ to $\hat{k}$ $(2\leq \hat{k}\leq k)$, it arrives at
		\begin{equation}\begin{aligned}\label{eq3.55}
		&\quad\tau\sum_{n=2}^{\hat{k}}\|D_{1}\eta_{N}^{n}\|^{2}+\frac{\tau}{2}\|D_{1}\eta_{N}^{\hat{k}}\|^{2}
		+K_{\alpha}|\eta^{\hat{k}}_{N}|_{\alpha/2}^{2}+\kappa\tau^2\|D_1\eta_{N}^{\hat{k}}\|^2\\
		&\leq\widetilde{C}_{2}(TN^{-2\alpha}+2\tau\sum_{n=0}^{\hat{k}-1}\|\eta_{N}^{n}\|^{2}),
		\end{aligned}\end{equation}
		where $\eta_{N}^{1}=0$ has been used. Combining \eqref{eq3.55} and
		$\|\eta^{\hat{k}}_{N}\|^2= \tau^2\|\sum_{n=2}^{\hat{k}}D_{1}\eta^{n}_{N}\|^2
		\leq T\tau\sum_{n=1}^{\hat{k}}$ $\|D_{1}\eta^{n}_{N}\|^2$ presents
		\begin{equation}\begin{aligned}\label{eq3.66}
		\|\eta^{\hat{k}}_{N}\|^{2}
		\leq\widetilde{C}_{3}(N^{-2\alpha}+2\tau\sum_{n=0}^{\hat{k}-1}\|\eta_{N}^{n}\|^{2}).
		\end{aligned}\end{equation}
		Applying the Gronwall inequality,
		\begin{equation}\begin{aligned}\label{eq3.57}
		\|\eta^{\hat{k}}_{N}\|\leq \widetilde{C}_{4}N^{-\alpha},
		\quad 2\leq \hat{k}\leq k.
		\end{aligned}\end{equation}
		It is obvious that
		\begin{equation}\begin{aligned}\label{e3.59}
		\|U^{k}_{N}\|_{\infty}&\leq \|P_{N}U^{k}\|_{\infty}+\|\eta^{k}_N\|_{\infty}
		\leq\|P_{N}U^{k}\|_{\infty}+C_{3}N\|\eta^{k}_N\|\\
		&\leq \|P_{N}U^{k}\|_{\infty}+C_{3}\widetilde{C}_{4}N^{1-\alpha}\leq \|P_{N}U^{k}\|_{\infty} + 1\\
		&\leq M_{2},
		\end{aligned}\end{equation}
		provided that $N\geq (C_{3}\widetilde{C}_{4})^{1/(\alpha-1)}$. Taking
		$N_{1}^{*}=(C_{3}\widetilde{C}_{4})^{1/(\alpha-1)}$, this proof is completed.
	\end{proof}
	
	Next, we prove the error estimate of of
	fully discrete scheme \eqref{e2.3}.
	Denote
	\begin{equation}\label{e4.1}
	\theta_{N}^{n}=P_{N}u^{n}-U_{N}^{n},\quad n=1,2,\ldots,K.
	\end{equation}
	We can obtain the error equation of \eqref{e2.3} as
	\begin{equation}\label{e4.2} \mathcal{A}^n(\theta_{N},v)=(F_{3}^{n-1},v)+(R^{n},v)-\tau\kappa(D_1\theta_{N}^n-D_1\theta_{N}^{n-1},v),\quad 2\leq n\leq K,
	\end{equation}
	where
	\begin{equation}
	\label{e4.3}
	F_{3}^{n-1}=2G(u^{n-1})-G(u^{n-2})-P_N(2G(U^{n-1}_{N})-G(U^{n-2}_{N})).
	\end{equation}
	
	We are now in a position to prove Theorem \ref{th2.1}.
	\begin{proof}
		It is obvious that \eqref{eq2.6} holds for $n=0,1$.
		Taking $v=2\tau D_{1}\theta_{N}^{n}$ in (\ref{e4.2}) leads to
		\begin{equation}\label{e3.544}\begin{aligned}
		&\quad\frac{5\tau}{2}\|D_{1}\theta_{N}^{n}\|^{2}+{K_{\alpha}}|\theta_{N}^{n}|^{2}_{\alpha/2}+\kappa\tau^2\|D_{1}\theta_{N}^{n}\|^{2}\\
		&\leq\frac{\tau}{2}\|D_{1}\theta^{n-1}_{N}\|^{2}+K_{\alpha}|\theta^{n-1}_{N}|_{\alpha/2}^{2}+\kappa\tau^2\|D_{1}\theta_{N}^{n-1}\|^{2}\\
		&\quad+2\tau(F_{2}^{n-1},D_{1}\theta^{n}_{N})+2\tau(R^{n},D_{1}\theta_{N}^{n}).
		\end{aligned}\end{equation}
		From Theorem \ref{th3.2}, we derive that $U_N^n$ is bounded,
		which indicates
		\begin{equation}\begin{aligned}
		\label{eq3.58}
		&2\tau(F_{2}^{n-1}, D_{1}\theta^{n}_{N})\leq \frac{\tau}{2}\|D_{1}\theta^{n}_{N}\|^{2}+\widehat{C}_{1}\tau(N^{-2r}
		+\|\theta_{N}^{n-1}\|^{2}+\|\theta_{N}^{n-2}\|^{2}).
		\end{aligned}\end{equation}
		The following bound is easily obtained
		\begin{eqnarray}
		2\tau(R^{n},D_{1}\theta_{N}^{n})
		\leq \frac{\tau}{2}\|D_{1}\theta^{n}_{N}\|^{2}+\widehat{C}_{2}\tau^5.\label{eq3.577}
		\end{eqnarray}
		Combining \eqref{e3.544}, \eqref{eq3.58}, and \eqref{eq3.577} yields
		\begin{equation}\begin{aligned}
		\label{e3.54}
		&\quad\frac{3\tau}{2}\|D_{1}\theta_{N}^{n}\|^{2}+{K_{\alpha}}|\theta_{N}^{n}|^{2}_{\alpha/2}+\kappa\tau^2\|D_{1}\theta_{N}^{n}\|^{2}
		\\&\leq\frac{\tau}{2}\|D_{1}\theta^{n-1}_{N}\|^{2}+K_{\alpha}|\theta^{n-1}_{N}|_{\alpha/2}^{2}+\kappa\tau^2\|D_{1}\theta_{N}^{n-1}\|^{2}\\
		&\quad+\widehat{C}_{1}\tau N^{-2r}+\widehat{C}_{2}\tau^5+
		\widehat{C}_{1}\tau (\|\theta_{N}^{n-1}\|^{2}+\|\theta_{N}^{n-2}\|^{2}).
		\end{aligned}\end{equation}
		Summing $n$ from $2$ to $\hat{k}$ $(2\leq \hat{k}\leq k)$ presents
		\begin{equation}\begin{aligned}\label{e3.55}
		&\quad\tau\sum_{n=2}^{\hat{k}}\|D_{1}\theta_{N}^{n}\|^{2}+\frac{\tau}{2}\|D_{1}\theta_{N}^{\hat{k}}\|^{2}+
		K_{\alpha}|\theta^{\hat{k}}_{N}|_{\alpha/2}^{2}+\kappa\tau^2\|D_{1}\theta_{N}^{\hat{k}}\|^{2}\\
		&\leq \widehat{C}_{3}(N^{-2r}+\tau^4+\tau\sum_{n=0}^{\hat{k}-1}\|\theta_{N}^{n}\|^{2}),
		\end{aligned}\end{equation}
		where $\theta_{N}^{0}=\theta_{N}^{1}=0$ was used.
		Combining \eqref{e3.55} and
		$\|\theta_{N}^{\hat{k}}\|^2= \tau^2\|\sum_{n=2}^{\hat{k}}D_{1}\theta^{n}_{N}\|^2
		\leq T\tau\sum_{n=1}^{\hat{k}}$ $\|D_{1}\theta^{n}_{N}\|^2$ gives
		\begin{equation}\begin{aligned}\label{e3.66}
		\|\theta^{\hat{k}}_{N}\|^{2}
		\leq \widehat{C}_{4}(N^{-2r}+\tau^4+\tau\sum_{n=0}^{\hat{k}-1}\|\theta_{N}^{n}\|^{2}).
		\end{aligned}\end{equation}
		By the Gronwall inequality, we obtain
		\begin{equation}\begin{aligned}\label{e4.9}
		\|\theta^{\hat{k}}_{N}\|\leq \widehat{C}_{5}(\tau^{2}+N^{-r}).
		\end{aligned}\end{equation}
		The above inequality and Lemma \ref{le3.1} imply
		\begin{equation*} \label{e4.10}
		\|u^{\hat{k}}-U^{\hat{k}}_{N}\|\leq\|u^{\hat{k}}-P_{N}u^{\hat{k}}\|+\|\theta^{\hat{k}}_{N}\|
		\leq C(\tau^{2}+N^{-r}),
		\end{equation*}
		which ends the proof.
	\end{proof}

\end{document}